\RequirePackage{fix-cm}
\documentclass[twocolumn]{svjour3}          
\smartqed  
\usepackage[english]{babel}
\usepackage[T1]{fontenc}  			
\usepackage{lmodern}
\usepackage[utf8]{inputenc}
\usepackage{lipsum}
\usepackage{amsmath}
\usepackage{mathtools}				
\usepackage{amssymb}
\usepackage{amsfonts}
\usepackage{commath}
\usepackage{mathrsfs}
\usepackage{microtype}
\usepackage{verbatim}
\usepackage[sort,comma]{natbib}
\usepackage[dvipsnames]{xcolor}

\usepackage{tikz}
\usepackage{pgfplots}
\usepgfplotslibrary{groupplots}
\usepgfplotslibrary{external}
\usepgfplotslibrary{statistics}
\usetikzlibrary{decorations}
\usetikzlibrary{decorations.markings}
\usetikzlibrary{pgfplots.statistics}
\usepgfplotslibrary{fillbetween}
\pgfplotsset{compat = 1.3}

\newcommand{\T}{\ensuremath{\mathsf{T}}}

\spnewtheorem{assumption}{Assumption}{\bf}{\it}

\journalname{Statistics and Computing}
\begin{document}

\twocolumn

\title{Bayesian ODE Solvers: The Maximum A Posteriori Estimate
}


\author{Filip Tronarp         \and
        Simo S\"arkk\"a       \and
        Philipp Hennig
}


\institute{Filip Tronarp \at
              University of T\"ubingen \\
              \email{filip.tronarp@uni-tueingen.de}
           \and
           Simo S\"arkk\"a \at
           Aalto University \\
              \email{simo.sarkka@aalto.fi}
           \and
           Philipp Hennig \at
           University of T\"ubingen and MPI for Intelligent Systems \\
              \email{philipp.hennig@uni-tuebingen.de}
}

\date{Received: date / Accepted: date}

\maketitle

\begin{abstract}
There is a growing interest in probabilistic numerical solutions to ordinary differential equations. In this paper, the \emph{maximum a posteriori estimate} is studied under the class of $\nu$ times differentiable linear time invariant Gauss--Markov priors, which can be computed with an iterated extended Kalman smoother. The maximum a posteriori estimate corresponds to an optimal interpolant in the reproducing kernel Hilbert space associated with the prior, which in the present case is equivalent to a Sobolev space of smoothness $\nu+1$. Subject to mild conditions on the vector field, convergence rates of the maximum a posteriori estmate are then obtained via methods from nonlinear analysis and scattered data approximation. These results closely resemble classical convergence results in the sense that a $\nu$ times differentiable prior process obtains a global order of $\nu$, which is demonstrated in numerical examples.
\end{abstract}

\section{Introduction}
Let $\mathbb{T} = [0,T], \ T < \infty$, $f \colon \mathbb{T} \times \mathbb{R}^d \to \mathbb{R}^d$, $y_0 \in \mathbb{R}^d$ and consider the following ordinary differential equation (ODE):
\begin{equation}\label{eq:ode}
Dy(t) = f(t,y(t)), \ y(0) = y_0,
\end{equation}
where $D$ denotes the time derivative operator. Approximately solving  \eqref{eq:ode} on a discrete mesh $\mathbb{T}_N = \{t_n\}_{n=0}^N, \ 0 = t_0 < t_1 < \ldots < t_N = T$, involves finding a function $\hat{y}$ such that $\hat{y}(t_n) \approx y(t_n), \ n = 0, 1, \ldots, N$ and a procedure for finding $\hat{y}$ is called a \emph{numerical solver}. This is an important problem in science and engineering, and vast base of knowledge has accumulated on it \citep{Deuflhard2002,Hairer87,Hairer1996,Butcher2008}.

Classically, the error of a numerical solver is quantified in terms of the worst case error. However, in applications where a numerical solution is sought as a component of a larger statistical inference problem (see, e.g., \citealt{Matsuda2019,Kersting2020}), it is desirable that the error can be quantified with the same semantic, that is to say, \emph{probabilistically} \citep{Hennig2015,Oates2019a}. Hence the recent endeavour to develop probabilistic ODE solvers.

Probabilistic ODE solvers can roughly be divided into two classes, sampling based solvers and deterministic solvers. The former class includes classical ODE solvers that are stochastically perturbed \citep{Teymur2016,Conrad2017,Teymur2018a,Abdulle2020,Lie2019}, solvers that approximately sample from a Bayesian inference problem \citep{Tronarp2019c}, and solvers that perform Gaussian process regression on stochastically generated data \citep{Chkrebtii2016}. Deterministic solvers formulate the problem as a Gaussian process regression problem, either with a data generation mechanism \citep{Skilling1991,Hennig2014,Schober2014,Kersting2016,Magnani2017,Schober2019} or by attempting to constrain the estimate to satisfy the ODE on the mesh \citep{Tronarp2019c,John2019}. For computational reasons it is fruitful to select the Gaussian process prior to be Markovian \citep{Kersting2016,Magnani2017,Schober2019,Tronarp2019c}, as this reduces cost of inference from $O(N^3)$ to $O(N)$ \citep{Sarkka2013a,Hartikainen2010}. Due to the connection between inference with Gauss--Markov processes priors and spline interpolation \citep{Kimeldorf1970,Weinert1974,Sidhu1979}, the Gaussian process regression approaches are intimately connected with the spline approach to ODEs \citep{Schumaker1982,Wahba1973}. Convergence analysis for the deterministic solvers has been initiated, but the theory is as of yet not complete \citep{Kersting2018}.

The formal notion of Bayesian solvers was defined by \cite{Cockayne2019a}. Under particular conditions on the vector field, the solvers of \cite{Kersting2016,Magnani2017,Schober2019,Tronarp2019c} produce the exact posterior, if in addition a smoothing recursion is implemented, which corresponds to solving the batch problem as posed by \cite{John2019}. In some cases, the exact Bayesian solution can also be obtained by exploiting Lie theory \citep{Wang2018}.

In this paper, the Bayesian formalism of \cite{Cockayne2019a} is adopted for probabilistic solvers and priors of Gauss--Markov type are considered. However, rather than the exact posterior, the maximum a posteriori (MAP) estimate is studied. Many of the aforementioned Gaussian inference approaches are related to the MAP estimate. Due to the Gauss--Markov prior, the MAP estimate can be computed efficiently by the iterated extended Kalman smoother \citep{Bell1994}. Furthermore, the Gauss--Markov prior corresponds to a reproducing kernel Hilbert space (RKHS) of Sobolev type and the MAP estimate is equivalent to an optimal interpolant in this space. This enables the use of results from scattered data approximation \citep{Arcangeli2007} to establish, under mild conditions, that the MAP estimate converges to the true solution at a high polynomial rate in terms of the fill-distance (or equivalently, the maximum step size).

The rest of the paper is organised as follows. In Section \ref{sec:ssm}, the solution of the ODE \eqref{eq:ode} is formulated as a Bayesian inference problem. In Section \ref{sec:map_estimation}, the associated MAP problem is stated and the iterated extended Kalman smoother for computing it is presented \citep{Bell1994}. In Section \ref{sec:RKHS}, the connection between MAP estimation and optimisation in a certain reproducing kernel Hilbert space is reviewed.  In Section \ref{sec:convergence}, the error of the MAP estimate is analysed, for which polynomial convergence rates in the fill-distance are obtained. These rates are demonstrated in Section \ref{sec:numerical_examples}, and the paper is finally concluded by a discussion  in Section \ref{sec:discussion}.

\subsection{Notation}
Let $\Omega \subset \mathbb{R}$, then for a (weakly) differentiable function $u \colon \Omega \to \mathbb{R}^d$, its (weak) derivative is denoted by $Du$, or sometimes $\dot{u}$. The space of $m$ times continuously differentiable functions from $\Omega$ to $\mathbb{R}^d$ is denoted by $C^m(\Omega,\mathbb{R}^d)$. The space of absolutely continuous functions is denoted by $\mathrm{AC}(\Omega,\mathbb{R}^d)$. The vector valued Lesbegue spaces are denoted by $\mathcal{L}_p(\Omega,\mathbb{R}^d)$ and the related Sobolev spaces of $m$ times weakly differentiable functions are denoted by $H_p^m(\Omega,\mathbb{R}^d)$, that is, if $u \in H^m_p(\Omega,\mathbb{R}^d)$ then $D^m u \in \mathcal{L}_p(\Omega,\mathbb{R}^d)$. The norm of $y \in \mathcal{L}_p(\Omega,\mathbb{R}^d)$ is given by
\begin{equation*}
\norm{y}_{\mathcal{L}_p(\Omega,\mathbb{R}^d)} = \sum_{i=1}^d \norm{y_i}_{\mathcal{L}_p(\Omega,\mathbb{R})}.
\end{equation*}
If $p = 2$, the equivalent norm
\begin{equation*}
\norm{y}_{\mathcal{L}_p(\Omega,\mathbb{R}^d)} = \sqrt{\sum_{i=1}^d \norm{y_i}_{\mathcal{L}_p(\Omega,\mathbb{R})}^2}
\end{equation*}
is sometimes used. The Sobolev (semi-)norms are given by \citep{Adams2003,Valent2013}
\begin{subequations}
\begin{align*}
\abs{y}_{H_p^\alpha(\Omega,\mathbb{R})} &= \norm{D^\alpha y}_{\mathcal{L}_p(\Omega,\mathbb{R})}, \\
\norm{y}_{H_p^\alpha(\Omega,\mathbb{R})} &= \Big(\sum_{m=1}^\alpha \abs{y}_{H_p^m(\Omega,\mathbb{R})}^p\Big)^{1/p}, \\
\norm{y}_{H_p^\alpha(\Omega,\mathbb{R}^d)} &= \sum_{i=1}^d \norm{y_i}_{H_p^\alpha(\Omega,\mathbb{R})},
\end{align*}
\end{subequations}
and an equivalent norm on $H_p^\alpha(\Omega,\mathbb{R}^d)$ is
\begin{equation*}
\norm{y}'_{H_p^\alpha(\Omega,\mathbb{R}^d)} = \Big(\sum_{i=1}^d \norm{y_i}_{H_p^\alpha(\Omega,\mathbb{R})}^p\Big)^{1/p}.
\end{equation*}
Henceforth the domain and codomain of the function spaces will be omitted unless required for clarity.

For a positive definite matrix $\Sigma$, its symmetric square root is denoted by $\Sigma^{1/2}$, and the associated Mahalanobis norm of a vector $a$ is denoted by $\norm{a}_\Sigma = a^\T \Sigma^{-1} a$.

\section{A Probabilistic State-Space Model}\label{sec:ssm}
The present approach involves defining a probabilistic state-space model, from which the approximate solution to  \eqref{eq:ode} is inferred. This is essentially the same approach as that of \cite{Tronarp2019c}. The class of priors considered is defined in Section \ref{subsec:prior} and the data model is introduced in Section \ref{subsec:data_model}.

\subsection{The Prior}\label{subsec:prior}
Let $\nu$ be a positive integer, the solution of \eqref{eq:ode} is then modelled by a $\nu$-times differentiable stochastic process prior $Y(t)$ with a state-space representation. That is, the stochastic process $X(t)$ defined by
\begin{equation*}
X^\T(t) = \begin{pmatrix} Y^T(t) & D Y^\T(t) & \ldots & D^\nu Y^\T(t) \end{pmatrix}
\end{equation*}
solves a certain stochastic differential equation. Furthermore, let $\{\mathrm{e}_m\}_{m=0}^\nu$ be the canonical basis on $\mathbb{R}^{\nu+1}$ and $\mathrm{I}_d$ is the identity matrix in $\mathbb{R}^{d\times d}$, it is then convenient to define the matrices $\mathrm{E}_m = \mathrm{e}_m \otimes \mathrm{I}_d, \ 0 \leq m \leq \nu$. That is, the $m$th subvector of $X$ is given by
\begin{equation*}
X^m(t) = \mathrm{E}_m^\T X(t) = D^m Y(t), \quad 0 \leq m \leq \nu.
\end{equation*}
Now let $F_m \in \mathbb{R}^{d\times d}$, $0 \leq m \leq \nu$ and, $\Gamma \in \mathbb{R}^{d\times d}$ a positive definite matrix, and define the following differential operator
\begin{equation*}
\mathcal{A} = \Gamma^{-1/2}\Big( \mathrm{I}_d D^{\nu+1}  - \sum_{m=0}^\nu F_m D^m  \Big)
\end{equation*}
and the matrix $F \in \mathbb{R}^{d(\nu+1)\times d(\nu+1)}$ whose non-zero $d\times d$ blocks are given by
\begin{equation*}
F_{ij} =
\begin{cases}
\mathrm{I}_d,\ j = i +1,\ 0 \leq i,j < \nu, \\
F_j, \ i = \nu,\ 0 \leq j \leq \nu.
\end{cases}
\end{equation*}

The class of priors considered herein is then given by
\begin{equation}\label{eq:prior_measure_y}
Y(t) = \mathrm{E}_0^\T \exp(F t) X(0) + \int_0^T G_Y(t,\tau) \dif W(\tau),
\end{equation}
where $W$ is a standard Wiener process onto $\mathbb{R}^d$, $X(0) \sim \mathcal{N}(0,\Sigma(t_0^-))$, and $G_Y$ is the Green's function associated with $\mathcal{A}$ on $\mathbb{T}$ with initial condition $D^m y(t_0) = 0$, $m =0,\ldots,\nu$. The Green's function is given by
\begin{subequations}\label{eq:greens_functions}
\begin{align}
G_Y(t,\tau) &= \mathrm{E}_0^\T G_X(t,\tau), \label{eq:green_y} \\
G_X(t,\tau) &= \theta(t-\tau) \exp( F(t-\tau) ) \mathrm{E}_\nu \Gamma^{1/2},
\end{align}
\end{subequations}
where $\theta$ is Heaviside's step function.
By construction, \eqref{eq:prior_measure_y} has a state-space representation, which is given by the following stochastic differential equation \citep{Oksendal2003}
\begin{equation}\label{eq:prior_measure_x}
\dif X(t) = FX(t)\dif t + \mathrm{E}_\nu \Gamma^{1/2} \dif W(t), \ X(0) \sim \mathcal{N}(0,\Sigma(t_0^-)),
\end{equation}
where $X$ takes values in $\mathbb{R}^{d(\nu+1)}$ and the $m$th sub-vector of $X$ is given by $X^m = D^m Y$ and takes values in $\mathbb{R}^d$ for $0 \leq m \leq \nu$. The transition densities for $X$ are given by \citep{Sarkka2019}
\begin{equation}\label{eq:discrete_time_model}
X(t + h) \mid X(t) \sim \mathcal{N}( A(h)X(t), Q(h)),
\end{equation}
where $\mathcal{N}(\mu,\Sigma)$ denotes the Normal distribution with mean and covariance $\mu$ and $\Sigma$, respectively, and
\begin{subequations}\label{eq:discrete_time_parameters}
\begin{align}
A(h) &=  \exp ( Fh ),\\
Q(h) &= \int_0^T G_X(h,\tau)G_X^\T(h,\tau) \dif \tau. \label{eq:discrete_time_process_noise}
\end{align}
\end{subequations}

Note that the integrand in \eqref{eq:discrete_time_process_noise} has limited support, that is, the effective interval of integration is $[0,h]$. These parameters can practically be computed via the matrix fraction decomposition method \citep{Sarkka2019}. Details are given in Appendix \ref{app:discretisation}.

\subsubsection{The Selection of Prior}\label{subsec:prior_selection}
While $\nu$ determines the smoothness of the prior, the actual estimator will be of smoothness $\nu+1$ (see Section \ref{sec:RKHS}) and the convergence results of Section \ref{sec:convergence} pertain to the case when the solution is of smoothness $\nu+1$ as well. Consequently, if it is known that the solution is of smoothness $\alpha \geq 2$ then setting $\nu = \alpha - 1$ ensures the present convergence guarantees are in effect. Though it is likely convergence rates can be obtained for priors that are ``too smooth'' as well (see \citealt{Kanagawa2020a} for such results pertaining to numerical integration).

Once the degree of smoothness $\nu$ has been selected, the parameters $\Sigma(t_0^-)$, $\{F_m\}_{m=0}^\nu$, and $\Gamma$ need to be selected. Some common sub-classes of \eqref{eq:prior_measure_y} are listed below.
\begin{itemize}
\item (Released $\nu$ times integrated Wiener process onto $\mathbb{R}^d$). The process $Y$ is a  $\nu$ times integrated Wiener process if $F_m = 0, \ m = 1,\ldots,\nu$. The parameters $\Sigma(t_0^-)$ and $\Gamma$ are free. Though it is advisable to set $\Gamma = \sigma^2 \mathrm{I}_d$  for some scalar $\sigma^2$. In this case $\sigma^2$ can be fit (estimated) to the particular ODE being solved (see Appendix \ref{app:uncertainty_calibration}). This class of processes is denoted by $Y \sim \mathrm{IWP}(\Gamma,\nu)$.
\item ($\nu$ times integrated Ornstein--Uhlenbeck process onto $\mathbb{R}^d$). The process $Y$ is a  $\nu$ times integrated Ornstein--Uhlenbeck process if $F_m = 0, \ m = 1,\ldots,\nu-1$. The parameters $\Sigma(t_0^-)$, $F_\nu$, and $\Gamma$ are free. As with $\mathrm{IWP}(\Gamma,\nu)$, it is advisable to set $\Gamma = \sigma^2 \mathrm{I}_d$. These processes are denoted by $Y \sim \mathrm{IOUP}(F_\nu,\Gamma,\nu)$.
\item (Mate\'rn processes of smoothness $\nu$ onto $\mathbb{R}$). If $d = 1$ then $Y$ is a Mate\'rn process of smoothness $\nu$ if (cf. \citealt{Hartikainen2010})
\begin{subequations}
\begin{align*}
F_m &= -\binom{\nu + 1}{m} \lambda^{\nu+1 - m},\quad  m = 0,\ldots,\nu,  \\
\Gamma &= 2 \sigma^2 \lambda^{2\nu+1},
\end{align*}
\end{subequations}
for some $\lambda,\sigma^2 > 0$, and $\Sigma(t_0^-)$ is set to the stationary covariance matrix of the resulting $X$ process. If $d > 1$ then each coordinate of the solution can be modelled by an individual Mate\'rn process.
\end{itemize}

\begin{remark}
Many popular choices of Gaussian processes not mentioned here also have state-space representations or can be approximated by a state-space model \citep{Karvonen2016,Tronarp2018d,Hartikainen2010,Solin2014}. A notable example is Gaussian processes with squared exponential kernel \citep{Hartikainen2010}. See Chapter 12 of \cite{Sarkka2019}, for a thorough exposition.
\end{remark}

\subsection{The Data Model}\label{subsec:data_model}
For the Bayesian formulation of probabilistic numerical methods, the data model is defined in terms of an \emph{information operator} \citep{Cockayne2019a}. In this paper, the information operator is given by
\begin{equation}\label{eq:information_operator}
\mathcal{Z} = D - \mathcal{S}_f,
\end{equation}
where $\mathcal{S}_f$ is the Nemytsky operator associated with the vector field $f$ \citep{Marcus1973},\footnote{Nemytsky operators are also known as composition operators and superposition operators.} that is,
\begin{equation}\label{eq:nemytsky_operator_def}
\mathcal{S}_f[y](t) = f(t,y(t)).
\end{equation}
Clearly, $\mathcal{Z}$ maps the solution of \eqref{eq:ode} to a known quantity, the zero function. Consequently, inferring $Y$ reduces to conditioning on
\begin{equation*}
\mathcal{Z}[Y](t) = 0, \quad t \in \mathbb{T}_N.
\end{equation*}
The function $\mathcal{Z}[Y](t)$ can be expressed in simpler terms by use of the process $X$. That is, define the function
\begin{equation*}
z(t,x) \coloneqq x^{1} - f(t,x^{0}),
\end{equation*}
then $\mathcal{Z}[Y](t) = \mathcal{S}_z[X](t) = z(t,X(t))$. Furthermore, it is necessary to account for the initial condition, $X^0(0) = y_0$, and with small additional cost the initial condition of the derivative can also be enforced $X^1(0) = f(0,y_0)$.

\begin{remark}
The properties of the Nemytsky operator are entirely determined by the vector field $f$. For instance, if $f \in C^\alpha(\mathbb{T}\times \mathbb{R}^d, \mathbb{R}^d)$, $\alpha \geq 0$, then $\mathcal{S}_f$ maps $C^\nu(\mathbb{T},\mathbb{R}^d)$ to $C^{\min (\nu,\alpha)}(\mathbb{T},\mathbb{R})$, which is fine for present purposes. However, in the subsequent convergence analysis it is more appropriate to view $\mathcal{S}_f$ (and $\mathcal{Z}$) as a mapping between different Sobolev spaces, which is possible if $\alpha$ is sufficiently large \citep{Valent2013}.
\end{remark}

\section{Maximum A Posteriori Estimation}\label{sec:map_estimation}
The MAP estimate for $Y$, or equivalently for $X$, is in view of \eqref{eq:discrete_time_model} the solution to the optimisation problem
\begin{subequations}\label{eq:map_estimate_x}
\begin{alignat}{2}
&\!\min_{x(t_{0:N})}  & \mathcal{V}(x(t_{0:N}))  \\
&\text{subject to} &      & \mathrm{E}_0^\T x(t_0) - y_0 = 0, \label{eq:initial_value_constraint}\\
&                  &      & \mathrm{E}_1^\T x(t_0) - f(t_0,y_0) = 0,\label{eq:initial_derivative_constraint}\\
&                  &      &  z(t_n,x(t_n)) = 0,\ n = 1,\ldots,N, \label{eq:information_constraint}
\end{alignat}
\end{subequations}
where $h_n = t_n - t_{n-1}$ is the step size sequence and $\mathcal{V}$ is up to a constant, the negative log-density
\begin{equation}
\begin{split}
\mathcal{V}(x(t_{0:N})) &= \frac{1}{2}\sum_{n=1}^N \norm{x(t_n) - A(h_n)x(t_{n-1})}_{Q(h_n)}^2\\
&\quad + \frac{1}{2}\norm{x(t_0)}_{\Sigma(t_0^-)}^2  .
\end{split}
\end{equation}
If the vector field is affine in $y$, then the MAP estimate and the full posterior can be computed exactly via Gaussian filtering and smoothing \citep{Sarkka2013}. However, when this is not the case then, for instance, a Gauss--Newton method can be used, which can be efficiently implemented by Gaussian filtering and smoothing as well. This method for MAP estimation is known as the \emph{iterated extended Kalman smoother} \citep{Bell1994}.

\subsection{Inference with Affine Vector Fields}\label{subsec:affine_vector_field}
If the vector field is affine
\begin{equation*}
f(t,y) = \Lambda(t)y + \zeta(t),
\end{equation*}
then the information operator reduces to
\begin{equation*}
z(t,x) = x^1 - \Lambda(t) x^0 - \zeta(t),
\end{equation*}
and the inference problem reduces to Gaussian process regression \citep{Rasmussen2006} with a linear combination of function and derivative observations.
In the spline literature this is known as (extended) Hermite--Birkhoff data \citep{Sidhu1979}. In this case, the inference problem can be solved exactly with Gaussian filtering and smoothing \citep{Kalman1960,KalmanBucy1961,RauchTungStriebel1965,Sarkka2013,Sarkka2019}. Define the information sets
\begin{subequations}
\begin{align*}
\mathscr{Z}(t) = \{ z(\tau,X(\tau)) = 0 \colon \tau \in \mathbb{T}_N, \ \tau \leq t \}, \\
\mathscr{Z}(t^-) = \{ z(\tau,X(\tau)) = 0 \colon \tau \in \mathbb{T}_N, \ \tau < t\}.
\end{align*}
\end{subequations}
In Gaussian filtering and smoothing, only the mean and covariance matrix of $X(t)$ are tracked. The mean and covariance at time $t$, conditioned on $\mathscr{Z}(t)$ are denoted by $\mu_F(t)$ and $\Sigma_F(t)$, respectively, and $\mu_F(t^-)$ and $\Sigma_F(t^-)$ correspond to conditioning on $\mathscr{Z}(t^-)$, which are limits from the left. The mean and covariance conditioned on $\mathscr{Z}(T)$ at time $t$ are denoted by $\mu_S(t)$ and $\Sigma_S(t)$, respectively.

Before starting the filtering and smoothing recursions, the process $X$ needs to be conditioned on the initial values
\begin{equation*}
\mathrm{E}_0^\T X(0) = y_0, \quad \mathrm{E}_1^\T X(0) = f(t_0,y_0).
\end{equation*}
This is can be done by a Kalman update
\begin{subequations}
\begin{align}
C^\T(t_0) &= \begin{pmatrix} \mathrm{E}_0 & \mathrm{E}_1 \end{pmatrix},\\
S(t_0) &= C(t_0) \Sigma(t_0^-) C^\T(t_0), \\
K(t_0) &= \Sigma(t_0^-) C^\T(t_0) S^{-1}(t_0), \\
\mu_F(t_0) &= K(t_0) \begin{pmatrix}  y_0 \\ f(t_0,y_0)\end{pmatrix}, \\
\Sigma_F(t_0) &= \Sigma(t_0^-) - K(t_0) S(t_0) K^\T(t_0).
\end{align}
\end{subequations}
The filtering  mean and covariance on the mesh evolve as
\begin{subequations}\label{eq:dt_prediction}
\begin{align}
\mu_F(t_n^-)  &= A(h_n) \mu_F(t_{n-1}), \\
\Sigma_F(t_n^-) &= A(h_n) \Sigma_F(t_{n-1}) A^\T(h_n) + Q(h_n). \label{eq:dt_covariance_prediction}
\end{align}
\end{subequations}
The prediction moments at $t \in \mathbb{T}_N$ are then corrected according to the Kalman update
\begin{subequations}\label{eq:kalman_update}
\begin{align}
C(t_n) &= \mathrm{E}_1^\T - \Lambda(t_n)\mathrm{E}_0^\T,  \label{eq:measurement_matrix}\\
S(t_n) &= C(t_n) \Sigma_F(t_n^-) C^\T(t_n), \label{eq:marginal_variance}\\
K(t_n) &= \Sigma_F(t_n^-) C^\T(t_n) S^{-1}(t_n), \label{eq:kalman_gain}\\
\mu_F(t_n) &= \mu_F(t_n^-) + K(t_n)\big( \zeta(t_n) - C(t_n)\mu_F(t_n^-) \big), \label{eq:filter_mean_update} \\
\Sigma_F(t_n) &= \Sigma_F(t_n^-) - K(t_n) S(t_n) K^\T(t_n). \label{eq:filter_covariance_update}
\end{align}
\end{subequations}
On the mesh $\mathbb{T}_N$, the smoothing moments are given by
\begin{subequations}\label{eq:dt_smoothing}
\begin{align}
G(t_n) &= \Sigma_F(t_n) A^\T(h_{n+1}) \Sigma_F^{-1}(t_{n+1}^-), \label{eq:dt_smoothing_gain}\\
\mu_S(t_n) &= \mu_F(t_n) + G(t_n)(\mu_S(t_{n+1})  - \mu_F(t_{n+1}^-)), \\
\begin{split}
\Sigma_S(t_n) &=  G(t_n)\big( \Sigma_S(t_{n+1}) - \Sigma_F(t_{n+1}^-) \big)G^\T(t_n) \label{eq:dt_smoothing_covariance}\\
&\quad + \Sigma_F(t_n),
\end{split}
\end{align}
\end{subequations}
with terminal conditions $\mu_S(t_N) = \mu_F(t_N)$, and $\Sigma_S(t_N) = \Sigma_F(t_N)$. The MAP estimate and its derivatives, on the mesh, are then given by
\begin{equation*}
D^m \hat{y}(t) = \mathrm{E}_m^\T \mu_S(t), \quad t \in \mathbb{T}_N, \ m = 0,\ldots,\nu.
\end{equation*}

\begin{remark}
The filtering covariance can be written as
\begin{equation*}
\begin{split}
\Sigma_F(t_n) &= \Sigma_F^{1/2}(t_n^-)\Big( \mathrm{I} -  \operatorname{Proj}\Big(\Sigma_F^{1/2}(t_n^-) C^\T(t_n) \Big) \Big) \\
&\quad \times \Sigma_F^{1/2}(t_n^-),
\end{split}
\end{equation*}
where $\operatorname{Proj}(A) =  A (A^\T A)^{-1} A^\T$ is the projection matrix onto the column space of $A$. By \eqref{eq:measurement_matrix} and $\Sigma_F(t_n^-) \succ 0$, the dimension of the column space of $\Sigma_F^{1/2}(t_n^-) C^\T(t_n)$ is readily seen to be $d$. That is, the null-space of $\Sigma_F(t_n)$ is of dimension $d$. By \eqref{eq:dt_smoothing_gain} and \eqref{eq:dt_smoothing_covariance}, it is also seen that $\Sigma_F(t_n)$ and $\Sigma_S(t_n)$ share null-space. This rank deficiency is not a problem in principle since the addition of $Q(h_n)$ in \eqref{eq:dt_covariance_prediction} ensures $\Sigma_F(t_n^-)$ is of full rank. However, in practice $Q(h_n)$ may become numerically singular for very small step sizes.
\end{remark}

While Gaussian filtering and smoothing only provides the posterior for affine vector fields, it forms the template for nonlinear problems as well. That is, the vector field is replaced by an affine approximation \citep{Schober2019,Tronarp2019c,Magnani2017}. The iterated extended Kalman smoother approach for doing so is discussed in the following.

\subsection{The Iterated Extended Kalman Smoother}\label{subsec:approximate_gaussian_inference}
For non-affine vector fields, only the update becomes intractable. Approximation methods involve different ways of approximating the vector field with an affine function
\begin{equation*}
f(t,y) \approx \hat{\Lambda}(t)y + \hat{\zeta}(t),
\end{equation*}
whereafter approximate filter means and covariances are obtained by plugging $\hat{\Lambda}$ and $\hat{\zeta}$ into \eqref{eq:kalman_update}. The iterated extended Kalman smoother linearises $f$ around the smoothing mean in an iterative fashion. That is,
\begin{subequations}\label{eq:ieks_linearisation}
\begin{align}
\hat{\Lambda}^l(t_n) &= J_f(t_n,\mathrm{E}_0^\T \mu_S^l(t_n)), \\
\hat{\zeta}^l(t_n) &= f(t_n,\mathrm{E}_0^\T \mu_S^l(t_n)) - J_f(t_n,\mathrm{E}_0^\T \mu_S^l(t_n))\mathrm{E}_0^\T \mu_S^l(t_n).
\end{align}
\end{subequations}
The smoothing mean and covariance at iteration $l+1$, $\mu_S^{l+1}(t)$ and $\Sigma_S^{l+1}(t)$, are then obtained by running the filter and smoother with the parameters in \eqref{eq:ieks_linearisation}.

As mentioned, this is just the Gauss--Newton algorithm for the maximum a posteriori trajectory \citep{Bell1994}, and it can be shown that, under some conditions on the Jacobian of the vector field, the fixed-point is at least a local optimum to the MAP problem \eqref{eq:map_estimate_x} \citep{Knoth1989}. Moreover, the IEKS is just a clever implementation of the method of \cite{John2019} whenever the prior process has a state-space representation.

\subsubsection{Initialisation}
In order to implement the IEKS, a method of initialisation needs to be devised. Fortunately, there exists non-iterative Gaussian solvers for this purpose  \citep{Schober2019,Tronarp2019c}. These methods also employ Taylor series expansions to construct an affine approximation of the vector field. These methods select an expansion point at the prediction estimates $\mathrm{E}_0^\T \mu_F(t_n^-)$, and consequently the affine approximation can be constructed on the fly within the filter recursion. The affine approximation due to a zeroth order expansion gives the parameters \citep{Schober2019}
\begin{subequations}\label{eq:eks0_linearisation}
\begin{align}
\hat{\Lambda}(t_n) &= 0, \\
\hat{\zeta}(t_n) &= f(t_n,\mathrm{E}_0^\T \mu_F(t_n)),
\end{align}
\end{subequations}
and will be referred to as the zeroth order extended Kalman smoother (EKS0). The affine expansion due to a first order expansion \citep{Tronarp2019c} gives the parameters
\begin{subequations}\label{eq:eks1_linearisation}
\begin{align}
\hat{\Lambda}(t_n) &= J_f(t_n,\mathrm{E}_0^\T \mu_F(t_n)), \\
\hat{\zeta}(t_n) &= f(t_n,\mathrm{E}_0^\T \mu_F(t_n)) - J_f(t_n,\mathrm{E}_0^\T \mu_F(t_n))\mathrm{E}_0^\T \mu_F(t_n),
\end{align}
\end{subequations}
and will be referred to as the first order extended Kalman smoother (EKS1). Note that EKS0 computes the exact MAP estimate in the event that the vector field $f$ is constant in $y$, while EKS1 computes the exact MAP estimate in the more general case when $f$ is affine in $y$. Consequently, as EKS1 makes a more accurate approximation of the vector field than EKS0, it is expected to perform better.

Furthermore, as Jacobians of the vector field will be computed in the IEKS iteration anyway, the preferred method of initialisation is EKS1, which is the method used in the subsequent experiments.

\subsubsection{Computational Complexity}
The computational complexity of a Gaussian filtering and smoothing method for approximating the solution of \eqref{eq:ode} can be separated into two parts (i) the cost of linearisation and (ii) the cost of inference. The cost of inference here refers to the computatational cost associated with the filtering and smoothing recursion, which for affine systems is $\mathcal{O}(N d^3\nu^3 )$. Since EKS0 and EKS1 perform the filtering and smoothing recursion once, their cost of inference is the same, $\mathcal{O}(N d^3 \nu^3 )$. Furthermore, the linearisation cost of EKS0 amounts to $N+1$ evaluations of $f$ and no evaluations of $J_f$, while EKS1 evaluates $f$ $N+1$ times and $J_f$ $N$ times, respectively. Assuming IEKS is initialised by EKS1 using $L$ iterations, including the initialisation, then the cost of inference is  $\mathcal{O}(LN d^3\nu^3 )$, $f$ is evaluated $LN + 1$ times, and $J_f$ is evaluated $LN$ times. A summary of the computational costs is given in Table \ref{tab:computational_complexity}.

\begin{table}
\caption{Comparison of the computational cost between EKS0, EKS1, and IEKS, where $L$ denotes the total number of iterations for IEKS and it is assumed that IEKS is initialised by EKS1.}\label{tab:computational_complexity}
\begin{tabular}{|c|c|c|c|}
\hline
  & EKS0 & EKS1 & IEKS \\ \hline
Inference cost & $\mathcal{O}(N d^3\nu^3 )$ & $\mathcal{O}(N d^3\nu^3 )$ & $\mathcal{O}(LN d^3\nu^3 )$\\ \hline
\# Evals of $f$ & $N+1$ & $N+1$ & $LN+1$ \\ \hline
\# Evals of $J_f$ & 0 & $N$ & $LN$ \\ \hline
\end{tabular}
\end{table}

\section{Interpolation in Reproducing Kernel Hilbert Space}\label{sec:RKHS}
The correspondence between inference in stochastic processes and optimisation in reproducing kernel Hilbert spaces is well known \citep{Kimeldorf1970,Weinert1974,Sidhu1979}. This correspondence is indeed present in the current setting as well, in the sense that MAP estimation as discussed in Section \ref{sec:map_estimation} is equivalent to optimisation in the reproducing kernel Hilbert space (RKHS) associated with $Y$ and $X$ (see \citealt[Proposition 3.6]{Kanagawa2018a} for standard Gaussian process regression). The purpose of this section is thus to establish that the RKHS associated with $Y$, which establishes what function space the MAP estimator lie in. Furthermore, it is shown that the MAP estimate is equivalent to an interpolation problem in this RKHS, which implies properties on its norm. These results will then be used in the convergence analysis of the MAP estimate in Section \ref{sec:convergence}.

\subsection{The Reproducing Kernel Hilbert Space of the Prior}
The RKHS of the Wiener process with domain $\mathbb{T}$ and codomain $\mathbb{R}^d$ is the set (cf. \citealt{VanDerVaart2008}, section 10)
\begin{equation*}
\mathbb{W}_0 = \{ w \colon w \in \mathrm{AC}(\mathbb{T},\mathbb{R}^d),\ w(0) = 0,\ \dot{w} \in \mathcal{L}_2(\mathbb{T},\mathbb{R}^d) \},
\end{equation*}
with inner product given by
\begin{equation*}
\langle w , w' \rangle_{\mathbb{W}_0} = \int_0^T \dot{w}^\T(\tau) \dot{w}'(\tau) \dif \tau = \langle \dot{w}, \dot{w}' \rangle_{\mathcal{L}_2}.
\end{equation*}
Let $\mathbb{Y}^{\nu+1}$ denote the reproducing kernel Hilbert space associated with the prior process $Y$ as defined by \eqref{eq:prior_measure_y}, then $\mathbb{Y}^{\nu+1}$ is given by the image of the operator \cite[lemmas 7.1, 8.1, and 9.1]{VanDerVaart2008}
\begin{equation*}
\mathcal{T}(\vec{y}_0,\dot{w}_y)(t) = \mathrm{E}_0^\T \exp (F t) \vec{y}_0 + \int_0^T G_Y(t,\tau) \dot{w}_y(\tau) \dif \tau,
\end{equation*}
where $\vec{y}_0 \in \mathbb{R}^{d(\nu+1)}$ and $\dot{w}_y \in \mathcal{L}_2(\mathbb{T},\mathbb{R}^d)$. That is,
\begin{equation*}
\begin{split}
&\mathbb{Y}^{\nu+1} = \\
&\quad\{ y \colon y = \mathcal{T}(\vec{y}_0,\dot{w}_y), \ \vec{y}_0 \in \mathbb{R}^{d(\nu+1)}, \ \dot{w}_y \in \mathcal{L}_2(\mathbb{T},\mathbb{R}^d) \},
\end{split}
\end{equation*}
and inner product is given by
\begin{equation*}
\begin{split}
\langle y, y' \rangle_{\mathbb{Y}^{\nu+1}}  &=  \vec{y}_0^\T \Sigma^{-1}(t_0^-) \vec{y}_0' + \langle \mathcal{A}y, \mathcal{A} y' \rangle_{\mathcal{L}_2}\\
&=     \vec{y}_0^\T \Sigma^{-1}(t_0^-) \vec{y}_0' + \langle \dot{w}_y, \dot{w}_{y'} \rangle_{\mathcal{L}_2}.
\end{split}
\end{equation*}

\begin{remark}
For an element $y \in \mathbb{Y}^{\nu+1}$, the vector $\vec{y}_0$ contains the initial values for $D^m y(t)$, $m = 0,\ldots,\nu$, in similarity with the vector $X(0)$ in the definition of the prior process $Y$ in \eqref{eq:prior_measure_y}. That is, $\vec{y}_0$ should not be confused with the initial value of \eqref{eq:ode}.
\end{remark}

Since $G_Y$ is the Green's function of a differential operator of order $\nu+1$ with smooth coefficients, $\mathbb{Y}^{\nu+1}$ can be identified as follows. A function $y \colon \mathbb{T} \to \mathbb{R}^d$ is in $\mathbb{Y}^{\nu+1}$ if and only if
\begin{subequations}\label{eq:rkhs_def}
\begin{align}
D^my &\in \mathrm{AC}(\mathbb{T},\mathbb{R}^d),\ m = 0,\ldots,\nu, \\
D^{\nu+1}y &\in \mathcal{L}_2(\mathbb{T},\mathbb{R}^d).
\end{align}
\end{subequations}
Hence by similar arguments as for the released $\nu$ times integrated Wiener process, Proposition \ref{prop:rkhs_equals_sobolev} holds (see proposition 2.6.24 and remark 2.6.25 of \citealt{Gine2016}).

\begin{proposition}\label{prop:rkhs_equals_sobolev}
The reproducing kernel Hilbert space $\mathbb{Y}^{\nu+1}$ as a set is equal to the Sobolev space $H_2^{\nu+1}(\mathbb{T},\mathbb{R}^d)$ and their norms are equivalent.
\end{proposition}

The reproducing kernel of $\mathbb{Y}^{\nu+1}$ is given by (cf. \citealt{Sidhu1979})
\begin{equation*}
\begin{split}
R(t,s) &= \mathrm{E}_0^\T \exp(Ft) \Sigma(t_0^-) \exp(F^\T s) \mathrm{E}_0 \\
&\quad + \int_0^T G_Y(t,\tau)G_Y^\T(s,\tau) \dif \tau,
\end{split}
\end{equation*}
which is also the covariance function of $Y$. The linear functionals
\begin{equation*}
y \mapsto v^\T D^m y(s), \quad v \in \mathbb{R}^d, \ t \in \mathbb{T}, \ m=0,\ldots,\nu,
\end{equation*}
are continuous and their representers are given by
\begin{subequations}
\begin{align*}
\eta_s^{m,v} &= R^{(0,m)}(t,s)v, \\
\langle \eta_s^{m,v}, y\rangle_{\mathbb{Y}^{\nu+1}} &= v^\T D^m y(s),
\end{align*}
\end{subequations}
where $R^{(m,k)}$ denotes $R$ differentiated $m$ and $k$ times with respect to the first and second arguments, respectively. Furthermore, define the matrix
\begin{equation*}
\eta_s^m = \begin{pmatrix} \eta_s^{m,\mathrm{e}_1} & \ldots & \eta_s^{m,\mathrm{e}_d} \end{pmatrix},
\end{equation*}
and with notation overloaded in the obvious way, the following identities hold
\begin{subequations}
\begin{align*}
D^m y(t) &= \langle \eta_t^m , y \rangle_{\mathbb{Y}^{\nu+1}}, \\
R^{(m,k)}(t,s) &= \langle \eta_t^m, \eta_s^k \rangle_{\mathbb{Y}^{\nu+1}}.
\end{align*}
\end{subequations}
Since there is a one-to-one correspondence between the processes $Y$ and $X$, the RKHS associated with $X$ is isometrically isomorphic to $\mathbb{Y}^{\nu+1}$, and it is given by
\begin{equation*}
\mathbb{X}^{\nu+1} = \{x \colon x^0 \in \mathbb{Y}^{\nu+1}, \ x^m = D^m x^0,\ m=1,\ldots,\nu\},
\end{equation*}
where $x^m$ is the $m$th sub-vector of $x$ of dimension $d$. The kernel associated with $\mathbb{X}^{\nu+1}$ is given by
\begin{equation}\label{eq:kernel_x}
\begin{split}
P(t,s) &=  \exp(Ft) \Sigma(t_0^-) \exp(F^\T s)  \\
&\quad + \int_0^T G_X(t,\tau)G_X^\T(s,\tau) \dif \tau,
\end{split}
\end{equation}
and the $d\times d$ blocks of $P$ are given by
\begin{equation*}
P_{m,k}(t,s) = R^{(m,k)}(t,s),
\end{equation*}
and $\psi_s = P(t,s)$ is the representer of evaluation at $s$,
\begin{equation*}
x(s) = \langle \psi_s, x \rangle_{\mathbb{X}^{\nu+1}}.
\end{equation*}
In the following, the short-hands $\mathbb{Y} = \mathbb{Y}^{\nu+1}$ and $\mathbb{X} = \mathbb{X}^{\nu+1}$ are in effect.

\subsection{Nonlinear Kernel Interpolation}
Consider the interpolation problem
\begin{equation}\label{eq:nonlin_kernel_interpolation}
\hat{y} = \mathop{\arg \, \min}_{y \in \mathcal{I}_N} \ \frac{1}{2}\norm{y}_\mathbb{Y}^2,
\end{equation}
where the feasible set is given by
\begin{equation*}
\begin{split}
\mathcal{I}_N &= \{ y \in \mathbb{Y} \colon y(0) = y_0, \   \dot{y}(0) = f(0,y_0)\}\\
&\quad\cap \{ y \in \mathbb{Y} \colon \mathcal{Z}[y](t) = 0, \ t \in \mathbb{T}_N\}.
\end{split}
\end{equation*}
Define the following subspaces of $\mathbb{Y}$
\begin{equation*}
\mathcal{R}_N(m) = \operatorname{span}\big\{ \eta_{t_n}^{l,\mathrm{e}_i}\}_{l = 0, n = 0, i = 1}^{m,N,d}, \quad m \leq \nu+1.
\end{equation*}
Similarly to other situations \citep{Kimeldorf1971,Cox1990,Girosi1995} our optimum can be expanded in a finite sub-space spanned by representers, which is the statement of Proposition \ref{prop:representer}.

\begin{proposition}\label{prop:representer}
The solution to \eqref{eq:nonlin_kernel_interpolation} is contained in $\mathcal{R}_N(1)$.
\end{proposition}

\begin{proof}
Any $y \in \mathbb{Y}$ has the orthogonal decomposition $y = y_\parallel + y_\perp$, where $y_\parallel \in \mathcal{R}_N(1)$ and $y_\perp \in \mathcal{R}_N^\perp(1)$. However, it must be the case that $\norm{y_\perp}_\mathbb{Y} = 0$, since
\begin{equation*}
\frac{1}{2}\norm{y}_\mathbb{Y}^2 = \frac{1}{2}\norm[0]{y_\parallel}_\mathbb{Y}^2 + \frac{1}{2}\norm[0]{y_\perp}_\mathbb{Y}^2 \geq \frac{1}{2}\norm[0]{y_\parallel}_\mathbb{Y}^2
\end{equation*}
and
\begin{subequations}
\begin{align*}
D^m y(0) &= \langle \eta_0^m, y_\parallel \rangle_\mathbb{Y}, \quad m = 0, \ldots, \nu + 1,\\
\mathcal{Z}[y](t) &= \langle \eta_t^1, y_\parallel \rangle_\mathbb{Y} - f\big((t, \langle \eta_t^0, y_\parallel \rangle_\mathbb{Y} \big), \\
\end{align*}
\end{subequations}
for all $t \in \mathbb{T}_N$. \qed
\end{proof}

By Proposition \ref{prop:representer} the optimal point of  \eqref{eq:nonlin_kernel_interpolation} can be written as
\begin{equation*}
y = \sum_{n=0}^N \begin{pmatrix} \eta_{t_n}^0 & \eta_{t_n}^1 \end{pmatrix} \begin{pmatrix} b_0(t_n) \\ b_1(t_n) \end{pmatrix}.
\end{equation*}
However, it is more convenient to expand the optimal point in the larger subspace, $\mathcal{R}_N(\nu) \supset \mathcal{R}_N(1)$
\begin{subequations}
\begin{align}
b(t_n) &= \begin{pmatrix} b_0^\T(t_n) & \ldots &  b_\nu^\T(t_n) \end{pmatrix}^\T, \\
y &= \sum_{n=0}^N \begin{pmatrix} \eta_{t_n}^0 & \ldots & \eta_{t_n}^\nu \end{pmatrix} b(t_n),\\
x &= \sum_{n=0}^N \psi_{t_n} b(t_n),
\end{align}
\end{subequations}
where $x$ is the equivalent element in $\mathbb{X}$ and
\begin{equation*}
\norm{y}_\mathbb{Y}^2 = \norm{x}_\mathbb{X}^2 = \sum_{n,m=0}^N b^\T(t_n) P(t_n,t_m) b(t_m),
\end{equation*}
or more compactly
\begin{equation}\label{eq:rkhs_norm_equals_negloglike}
\norm[0]{x}_\mathbb{X}^2 = \boldsymbol{x}^\T \boldsymbol{P}^{-1} \boldsymbol{x},
\end{equation}
where
\begin{equation*}
\boldsymbol{x} = \begin{pmatrix} x^\T(t_0) & \ldots & x^\T(t_N) \end{pmatrix}^\T, \quad \boldsymbol{P}_{n,m} = P(t_n,t_m).
\end{equation*}
Here $\boldsymbol{P}$ is the kernel matrix associated with function value observations of $X$ at $\mathbb{T}_N$. That is, \eqref{eq:rkhs_norm_equals_negloglike} is up to a constant equal to the negative log-density of $X$ restricted to $\mathbb{T}_N$. Proposition \ref{prop:ssm_map} immediately follows.
\begin{proposition}\label{prop:ssm_map}
The optimisation problem \eqref{eq:nonlin_kernel_interpolation} is equivalent to the MAP problem \eqref{eq:map_estimate_x}.
\end{proposition}

\section{Convergence Analysis}\label{sec:convergence}
In this section, convergence rates of the kernel interpolant $\hat{y}$ as defined by \eqref{eq:nonlin_kernel_interpolation}, and by Proposition \ref{prop:ssm_map} the MAP estimate are obtained. These rates will be in terms of the fill-distance of the mesh $\mathbb{T}_N$, which is\footnote{Classically the error of a numerical integrator is assessed in terms of the maximum step size which is twice the fill-distance.}
\begin{equation}\label{eq:fill_distance}
\delta = \sup_{t \in \mathbb{T}} \min_{n = 0, \ldots,N} \abs{t-t_n}.
\end{equation}
In the following results from the scattered data approximation literature \citep{Arcangeli2007} are employed. More specifically, for any $y \in \mathbb{Y}$, which satisfies the initial condition $y(0) = y_0$, formally has the following representation
\begin{equation*}
y(t) = y_0 + \int_0^t f(\tau,y(\tau)) \dif \tau + \mathcal{E}[y](t),
\end{equation*}
where the error operator $\mathcal{E}$ is defined as
\begin{equation*}
\mathcal{E}[y](t) = \int_0^t \mathcal{Z}[y](\tau) \dif \tau.
\end{equation*}
Of course any reasonable estimator $\hat{y}'$ ought to have the property that $\mathcal{Z}[\hat{y}'](t) \approx 0$ for $t \in \mathbb{T}_N$. The approach is thus to bound $\mathcal{Z}[\hat{y}'](t)$ in some suitable norm, which in turn gives a bound on $\mathcal{E}[\hat{y}'](t)$.

Throughout the discussion $\nu \geq 1$ is some fixed integer, which corresponds to the differentiability of the prior, that is, the kernel interpolant is in $H_2^{\nu+1}(\mathbb{T},\mathbb{R}^d)$. Furthermore, some regularity of the vector field will be required, namely Assumption \ref{ass:smooth}, given below.
\begin{assumption}\label{ass:smooth}
Vector field $f \in C^{\alpha+1}(\tilde{\mathbb{T}}\times \mathbb{R}^d,\mathbb{R}^d)$ with $\alpha \geq \nu$ and some set $\tilde{\mathbb{T}}$ with  $\mathbb{T} \subset \tilde{\mathbb{T}} \subset \mathbb{R}$.
\end{assumption}
Assumption \ref{ass:smooth} will, without explicit mention, be in force throughout the discussion of this section. It implies that (i) the model is well specified for sufficiently small $T$ and (ii) the information operator is well behaved. This shall be made precise in the following.

\subsection{Model Correctness and Regularity of the Solution}
Since $\nu \geq 1$, Assumption \ref{ass:smooth} implies $f$ is locally Lipschitz, and the classical existence and uniqueness results for the solution of Equation \eqref{eq:ode} apply. The extra smoothness on $f$ ensures the solution itself is sufficiently smooth for present purposes. These facts are summarised in Theorem \ref{thm:nice_solution}. For proof(s) refer to \cite[chapter 4, paragraph 32]{Arnold1992}.

\begin{theorem}\label{thm:nice_solution}
There exists $T^* > 0$ such that Equation \eqref{eq:ode} admits a unique solution $y^* \in  C^{\alpha+1}([0,T^*),\mathbb{R}^d)$.
\end{theorem}
Theorem \ref{thm:nice_solution} makes apparent the necessity of the next standing assumption.
\begin{assumption}\label{ass:durable_solution}
$T < T^*$. That is,  $\mathbb{T} \subset [0,T^*)$.
\end{assumption}
The model is thus correctly specified in the following sense.
\begin{corollary}[Correct model]\label{cor:correct_model}
The solution $y^*$ of Equation \eqref{eq:ode} on $\mathbb{T}$ is in $\mathbb{Y}$.
\end{corollary}
\begin{proof}
Firstly, $y^* \in  C^{\nu+1}(\mathbb{T},\mathbb{R}^d)$ due to Assumption \ref{ass:smooth}, Theorem \ref{thm:nice_solution}, and Assumption \ref{ass:durable_solution}. Since $D^{\nu+1}y^*$ is continuous and $\mathbb{T}$ is compact, it follows that $D^{\nu+1}y^*$ is bounded and $D^{\nu+1}y^* \in \mathcal{L}_p(\mathbb{T},\mathbb{R}^d)$ for any $p \in [1,\infty]$. Therefore (see e.g., \citealt[Theorem 20.8]{Nielson1997}) $D^m y^* \in \mathrm{AC}(\mathbb{T},\mathbb{R}^d), \ m = 0,\ldots,\nu$. \qed
\end{proof}
Corollary \ref{cor:correct_model} essentially ensures that there is an \emph{a priori} bound on the norm of the MAP estimate, that is $\norm[0]{\hat{y}}_\mathbb{Y} \leq \norm[0]{y^*}_\mathbb{Y}$.

\begin{remark}
It is in general difficult to determine $T^*$ for a given vector field $f$ and initial condition $y_0$, which makes Assumption \ref{ass:durable_solution} hard to verify in general. However, additional conditions can be imposed which assures $T^* = \infty$. An example of such a condition is that the vector field is \emph{uniformly Lipschitz} as mapping of $\mathbb{R}_+ \times \mathbb{R}^d \to \mathbb{R}^d$ \citep[Theorem 8.13]{Kelley2010}. That is, for any $y,y' \in \mathbb{R}^d$ it holds that
\begin{equation*}
\sup_{t \in \mathbb{R}_+} \norm{f(t,y) - f(t,y')} \leq \operatorname{Lip}(f) \norm{y - y'},
\end{equation*}
where $\operatorname{Lip}(f) < \infty$ is a positive constant.
\end{remark}

\subsection{Properties of the Information Operator}
By Proposition \ref{prop:rkhs_equals_sobolev}, $\mathbb{Y}$ correspond to the Sobolev space $H_2^{\nu+1}(\mathbb{T},\mathbb{R}^d)$, hence it is crucial to understand how the Nemytsky operator $\mathcal{S}_f$, and consequently $\mathcal{Z}$, act on Sobolev spaces. For the Nemytsky operator, the work has already been done \citep{Valent2013,Valent1985}, and Theorem \ref{thm:nemytsky_properties} is immediate.

\begin{theorem}\label{thm:nemytsky_properties}
Let $\mathscr{U}$ be an open subset of $H_2^{\nu+1}(\mathbb{T},\mathbb{R}^d)$ such that $y(\mathbb{T}) \subset U$ for any $y \in \mathscr{U}$, where $U$ some open subset of $\mathbb{R}^d$. The Nemytsky operator $\mathcal{S}_{f_i}$, associated with the $i$th coordinate of $f$ is then $C^1$ mapping from $\mathscr{U}$ onto $H_2^\nu(\mathbb{T},\mathbb{R})$ for $i=1,\ldots,d$. If in addition, $U$ is convex and bounded, then for any $y' \in \mathscr{U}$ there is number $c_0(y') > 0$ such that
\begin{equation*}
\norm[0]{\mathcal{S}_{f_i}[y] - \mathcal{S}_{f_i}[y']}_{H_2^\nu} \leq c_0(y') \abs{f_i}_{\nu+1,U} \norm[0]{y-y'}_{H_2^{\nu+1}},
\end{equation*}
for all $y \in \mathscr{U}$, where
\begin{equation*}
\abs{f_i}_{\nu+1,U} \coloneqq \sum_{m =0}^{\nu+1} \sup_{(t,a) \in \mathbb{T}\times U} \abs{D^m f_i(t,a)}.
\end{equation*}
\end{theorem}

\begin{proof}
The first claim is just an application of Theorem 4.1 of \cite[page 32]{Valent2013} and the second claim follows from (ii) in the proof of Theorem 4.5 in \cite[page 37]{Valent2013}. \qed
\end{proof}

Theorem \ref{thm:nemytsky_properties} establishes that $\mathcal{S}_{f_i}$ as a mapping of $\mathscr{U}$ onto $H_2^\nu(\mathbb{T},\mathbb{R})$ is locally Lipschitz. This property is inherited by the information operator.

\begin{proposition}\label{prop:information_properties}
In the same setting as Theorem \ref{thm:nemytsky_properties}. The $i$th coordinate of the information operator, $\mathcal{Z}_i$, is a $C^1$ mapping from $\mathscr{U}$ onto $H_2^\nu(\mathbb{T},\mathbb{R})$, for $i=1,\ldots,d$.  If in addition, $U$ is convex and bounded, then for any $y' \in \mathscr{U}$ there is number $c_1(y',\nu,f_i,U) > 0$ such that
\begin{equation*}
\norm[0]{\mathcal{Z}_i[y] - \mathcal{Z}_i[y']}_{H_2^\nu} \leq c_1(y',\nu,f_i,U) \norm[0]{y-y'}_{H_2^{\nu+1}},
\end{equation*}
for all $y \in \mathscr{U}$.
\end{proposition}

\begin{proof}
The differential operator $D \mathrm{e}_i^\T$ is a $C^1$ mapping of $\mathscr{U}$ onto $H_2^\nu(\mathbb{T},\mathbb{R})$. Consequently, by Theorem \ref{thm:nemytsky_properties} the same holds for the operator $D \mathrm{e}_i^\T - \mathcal{S}_{f_i} = \mathcal{Z}_i$. For the second part, the triangle inequality gives
\begin{equation*}
\begin{split}
\norm[0]{\mathcal{Z}_i[y] - \mathcal{Z}_i[y']}_{H_2^\nu} &\leq \norm[0]{D y_i - D y'_i}_{H_2^\nu} \\
&\quad+ \norm[0]{\mathcal{S}_{f_i}[y] - \mathcal{S}_{f_i}[y']}_{H_2^\nu},
\end{split}
\end{equation*}
and clearly
\begin{equation*}
\norm[0]{D y_i - D y'_i}_{H_2^\nu} \leq \norm[0]{y - y'}_{H_2^{\nu+1}}.
\end{equation*}
Consequently, by Theorem \ref{thm:nemytsky_properties} the statement holds by selecting
\begin{equation*}
 c_1(y',\nu,f_i,U) = 1 +  c_0(y') \abs{f_i}_{\nu+1,U}.
\end{equation*}
\qed
\end{proof}

\subsection{Convergence of the MAP Estimate}\label{subsec:map_convergence}
Proceeding with the convergence analysis of the MAP estimate can finally be done in view of the regularity properties of the solution $y^*$ and the information operator $\mathcal{Z}$ established by Corollary \ref{cor:correct_model} and Proposition \ref{prop:information_properties}. Combining these results with Theorem 4.1 of \cite{Arcangeli2007} leads to Lemma \ref{lem:information_bound}.

\begin{lemma}\label{lem:information_bound}
Let $\rho \in \mathbb{Y}$ with $\norm{\rho}_\mathbb{Y} > \norm[0]{y^*}_\mathbb{Y}$ and $q \in [1,\infty]$. Then there are positive constants $c_2$, $\delta_{0,\nu}$, $r$ (depending on $\rho$), and $ c_3(y^*,\nu,f_i,r) $ such that for any $y \in B(0,\norm{\rho}_\mathbb{Y})$ the following estimate holds for all $\delta < \delta_{0,\nu}$ and $m=0,\ldots,\nu-1$
\begin{equation*}
\begin{split}
\abs{\mathcal{Z}_i[y]}_{H_q^m} &\leq c_2 \delta^{\nu-m-(1/2-1/q)_+}c_3(y^*,\nu,f_i,r)  \norm[0]{y - y^*}_{H_2^{\nu+1}} \\
&\quad+ c_2\delta^{-m} \norm{\mathcal{Z}_i[y] \mid \mathbb{T}_N}_\infty  ,
\end{split}
\end{equation*}
where
\begin{equation*}
\norm{\mathcal{Z}_i[y] \mid \mathbb{T}_N}_\infty \coloneqq \max_{t \in \mathbb{T}_N} \abs{\mathcal{Z}_i[y](t)}.
\end{equation*}
\end{lemma}
\begin{proof}
Firstly, Cauchy--Schwartz inequality yields
\begin{equation*}
\abs{y_i(t)} = \abs[0]{\langle \eta_t^{0,\mathrm{e}_i}, y \rangle_\mathbb{Y}} \leq \sqrt{ R_{ii}(t,t)} \norm{y}_\mathbb{Y},
\end{equation*}
hence there is a positive constant $\tilde{c}$ such that
\begin{equation*}
\norm{y_i}_{\mathcal{L}_\infty} \leq \tilde{c} \norm{y}_\mathbb{Y}.
\end{equation*}
Consequently, there exists a radius $r$ (depending on $\rho$) such that $y(\mathbb{T}) \subset B(0,r)$ whenever $y \in B(0,\norm{\rho}_\mathbb{Y})$. The set $B(0,\norm{\rho}_\mathbb{Y})$ is open in $\mathbb{Y}$ and by Proposition \ref{prop:rkhs_equals_sobolev} it is an open set in $H_2^{\nu+1}(\mathbb{T},\mathbb{R}^d)$. Therefore, all the conditions of Proposition \ref{prop:information_properties} are met for the sets $B(0,\norm{\rho}_\mathbb{Y})$ and $B(0,r)$. In particular, $\mathcal{Z}_i[y] \in H_2^\nu(\mathbb{T})$ for all $y \in B(0,\norm{\rho}_\mathbb{Y})$. Consequently, for appropriate selection of parameters \cite[Theorem 4.1 page 193]{Arcangeli2007} gives
\begin{equation*}
\begin{split}
\abs{\mathcal{Z}_i[y]}_{H_q^m} &\leq c_2  \delta^{\nu-m-(1/2-1/q)_+} \abs{\mathcal{Z}_i[y]}_{H_2^{\nu}} \\
&\quad + c_2\delta^{-m} \norm{\mathcal{Z}_i[y] \mid \mathbb{T}_N}_\infty
\end{split}
\end{equation*}
for all $\delta < \delta_{0,\nu}$ and $m=0,\ldots,\nu-1$. Since $\mathcal{Z}[y^*] = 0$ it follows that
\begin{equation*}
\abs{\mathcal{Z}_i[y]}_{H_2^{\nu}} = \abs[0]{\mathcal{Z}_i[y]-\mathcal{Z}_i[y^*]}_{H_2^{\nu}} \leq \norm[0]{\mathcal{Z}_i[y]-\mathcal{Z}_i[y^*]}_{H_2^{\nu}},
\end{equation*}
and by Proposition \ref{prop:information_properties} the Lemma holds by selecting
\begin{equation*}
c_3(y^*,\nu,f_i,r) =  c_1(y^*,\nu,f_i,B(0,r)),
\end{equation*}
which concludes the proof. \qed
\end{proof}

In view of Lemma \ref{lem:information_bound}, for any estimator $\hat{y}' \in \mathbb{Y}$, its convergence rate can be established provided the following is shown:
\begin{itemize}
\item[(i)] There is $\rho \in \mathbb{Y}$ independent of $\hat{y}'$ such that $y^*,\hat{y}' \in B(0,\norm{\rho}_\mathbb{Y})$
\item[(ii)] A bound proportional to $\delta^\gamma$, $\gamma > 0$, of $\norm{\mathcal{Z}_i[\hat{y}'] \mid \mathbb{T}_N}_\infty$ exists.
\end{itemize}
Neither (i) nor (ii) appear trivial to establish for Gaussian estimators in general (e.g., the methods of \citealt{Schober2019} and \citealt{Tronarp2019c}). However, (i) and (ii) hold for the optimal (MAP) estimate $\hat{y}$, which yields Theorem \ref{thm:convergence}.

\begin{theorem}\label{thm:convergence}
Let $q \in [1,\infty]$, then under the same assumptions as in Lemma \ref{lem:information_bound}, there is a constant $c_4(y^*,\nu,f_i,r)$ such that for $\delta < \delta_{0,\nu}$ the following holds:
\begin{subequations}
\begin{align*}
\abs{\mathcal{E}_i[\hat{y}]}_{H_q^0} &\leq \delta^\nu T^{1/q} c_4(y^*,\nu,f_i,r) \norm{y^*}_\mathbb{Y}, \\
\abs{\mathcal{E}_i[\hat{y}]}_{H_q^m} &\leq   \delta^{\nu+1-m-(1/2-1/q)_+} c_4(y^*,\nu,f_i,r) \norm{y^*}_\mathbb{Y},
\end{align*}
\end{subequations}
where $m = 1,\ldots,\nu$.
\end{theorem}

\begin{proof}
Firstly, note that $\norm{\hat{y}}_\mathbb{Y} \leq \norm{y^*}_\mathbb{Y}$ and $\abs{\mathcal{E}_i[\hat{y}]}_{H_q^m} = \abs{\mathcal{Z}_i[\hat{y}]}_{H_q^{m-1}}$. By definition
\begin{equation*}
\norm{\mathcal{Z}_i[\hat{y}] \mid \mathbb{T}_N}_\infty = 0,
\end{equation*}
hence $\hat{y} \in B(0,\norm{\rho}_\mathbb{Y})$, and Lemma \ref{lem:information_bound} gives for $m=1,\ldots,\nu$
\begin{equation*}
\begin{split}
\abs{\mathcal{Z}_i[\hat{y}]}_{H_q^{m-1}} &\leq   \delta^{\nu+1-m-(1/2-1/q)_+} c_2c_3(y^*,\nu,f_i,r)\\
&\quad\times \norm[0]{\hat{y} - y^*}_{H_2^{\nu+1}}.
\end{split}
\end{equation*}
By Proposition \ref{prop:rkhs_equals_sobolev}, the fact that $\norm{\hat{y}}_\mathbb{Y} \leq \norm{y^*}_\mathbb{Y}$, and the triangle inequality, there exists a constant $c_B$ (independent of $\hat{y}$ and $y^*$) such that
\begin{equation*}
\norm[0]{\hat{y} - y^*}_{H_2^{\nu+1}} \leq c_B \norm{y^*}_\mathbb{Y}
\end{equation*}
and thus the second bound holds by selecting
\begin{equation*}
c_4(y^*,\nu,f_i,r) = c_2 c_B c_3(y^*,\nu,f_i,r).
\end{equation*}
For the first bound, the triangle inequality for integrals gives
\begin{equation*}
\abs{\mathcal{E}_i[\hat{y}](t)} \leq \abs{\mathcal{Z}_i[\hat{y}]}_{H_1^0},
\end{equation*}
and hence
\begin{equation*}
\abs{\mathcal{E}_i[\hat{y}](t)}_{H_q^0} \leq T^{1/q} \abs{\mathcal{Z}_i[\hat{y}]}_{H_1^0},
\end{equation*}
which combined with the second bound gives the first. \qed
\end{proof}

At first glance, it may appear that there is an appalling absence of dependence on $T$ in the constants of the convergence rates provided by Theorem \ref{thm:convergence}. This is not the case, the $T$ dependence have conveniently been hidden in $\norm[0]{y^*}_\mathbb{Y}$ and possibly $c_4(y^*,\nu,f_i,r)$. Now $c_4(y^*,\nu,f_i,r)$ depends on $c_0(y^*)$ and $\abs{f_i}_{\nu+1,B(0,r)}$ and unfortunately an explicit expression for $c_0(y^*)$ is not provided by \cite{Valent2013}, which makes the effect of $c_4(y^*,\nu,f_i,r)$ difficult to untangle. Nevertheless, the factor $\norm[0]{y^*}_\mathbb{Y}$  does indeed depend on the interval length $T$. For example, let $\lambda,y_0 \in \mathbb{R}$ and consider the following ODE
\begin{equation}\label{eq:test_ode}
\dot{y}(t) = \lambda y(t), \ y(0) = y_0.
\end{equation}
Setting $\Sigma(t_0^-) = \mathrm{I}$ and selecting the prior $\mathrm{IWP}(\mathrm{I},\nu)$ gives the following (in this case $\mathcal{A} = D^{\nu+1}$)
\begin{equation}\label{eq:rkhs_norm_test_iwp}
\norm{y^*}_\mathbb{Y}^2 = y_0^2 \Big( \sum_{m = 0}^\nu \lambda^{2m} + \frac{\lambda^{2\nu+1}}{2}\big( \exp(2\lambda T) - 1 \big)  \Big).
\end{equation}
Consequently, the global error can be quite bad when $\lambda > 0$ and $T$ is large even when $\delta$ is very small, which is the usual situation (cf. Theorem 3.4 of \cite{Hairer87}).

In the present context it is instructive to view the solution of \eqref{eq:ode} as a family of a quadrature problems
\begin{equation}\label{eq:integral_equation}
y(t) = y_0 + \int_0^t f(\tau,y(\tau)) \dif \tau,
\end{equation}
where $\dot{y}(t) = f(t,y(t))$ is modelled by an element of $H_2^\nu(\mathbb{T},\mathbb{R}^d)$. In view of Theorem \ref{thm:convergence}, $D^m \dot{\hat{y}}$ converges uniformly to $D^m \dot{y}^*$ at a rate of $\delta^{\nu -m - 1/2}, \ m = 0,\ldots,\nu-1$, thus for $\dot{\hat{y}}$ the same rate as for standard spline interpolation is obtained \citep{Schultz1970}. Furthermore, the rate obtained for $\hat{y}$ by Theorem \ref{thm:convergence} matches the rate for integral approximations using Sobolev kernels \cite[Proposition 1]{Kanagawa2020a}. That is, although dealing with a nonlinear interpolation/integration problem, Assumption \ref{ass:smooth} ensures the problem is still nice enough for the optimal interpolant to enjoy the classical convergence rates.

\section{Selecting the Hyperparameters}
In order to calibrate the credible intervals, the parameters $\Sigma(t_0^-)$ and $\Gamma$ need to be appropriately scaled to the problem being solved. It is practical to work with the parametrisation
$\Sigma(t_0^-) = \sigma^2 \breve{\Sigma}(t_0^-)$ and $\Gamma = \sigma^2 \breve{\Gamma}$ for fixed $\Sigma(t_0^-)$  and $\breve{\Gamma}$. In this case, the quasi maximum likelihood estimate of $\sigma^2$ can be computed cheaply, see Appendix \ref{app:uncertainty_calibration}.

In principle, the parameters $F_m$ ($0 \leq m \leq \nu$) can be estimated via quasi maximum likelihood as well but this would require iterative optimisation. For a given computational budget this may not be advantageous since the convergence rate obtained in Theorem \ref{thm:convergence} holds for any selection of these parameters. Thus it is not clear that spending a portion of a computational budget on estimating $F_m$ ($0 \leq m \leq \nu$) will yield a smaller solution error than solving the MAP problem on a denser grid (smaller $\delta$) for a fixed parameters, with the same total computational budget. The $\mathrm{IWP}(\sigma^2 \breve{\Gamma},\nu)$ class of priors thus seem like a good default choice ($F_m = 0$, $0 \leq m \leq \nu$).

Nevertheless, the parameters could in principle be selected to optimise the constant appearing in Theorem \ref{thm:convergence}. That is, solving the following optimisation problem
\begin{equation}
\!\min_{F_0,\ldots,F_\nu}  c_4(y^*,\nu,f_i,r) \norm[0]{y^*}^2,
\end{equation}
which unfortunately appears to be intractable in general. However, it might be a good idea to use the the second factor, $\norm[0]{y^*}^2$ as a proxy. For instance, consider solving the ODE in \eqref{eq:test_ode} again, but this time with the prior set to $\mathrm{IOUP}(\lambda,1,\nu)$. In this case, $\mathcal{A} = D^{\nu+1} - \lambda D^\nu$, and the RKHS norm becomes
\begin{equation}\label{eq:eq:rkhs_norm_test_ioup}
\norm{y^*}_\mathbb{Y}^2 = y_0^2  \sum_{m = 0}^\nu \lambda^{2m},
\end{equation}
which is strictly smaller than the RKHS norm obtained by $\mathrm{IWP}(\mathrm{I},\nu)$ in \eqref{eq:rkhs_norm_test_iwp}.

\section{Numerical Examples}\label{sec:numerical_examples}
In this section, the MAP estimate as implemented by the iterated extended Kalmans smoother (IEKS) is compared to the methods of \cite{Schober2019} (EKS0), and \cite{Tronarp2019c} (EKS1). In particular the convergence rates of the MAP estimator from Section \ref{sec:convergence} are verified, which appear to generalise to the other methods as well.

In Sections \ref{sec:logistic}, \ref{sec:riccati}, and \ref{sec:fitzhughnagumo} the logistic equation, Riccati equation, and the Fitz--Hugh--Nagumo model are investigated, respectively. The vector field is a polynomial in these cases, which means it is infinitely many times differentiable and Assumption \ref{ass:smooth} is satisfied for any $\nu \geq 1$. Lastly, in Section \ref{sec:nonsmooth}, a case where the vector field is only continuous is given, which means that Assumption \ref{ass:smooth} is violated for any $\nu \geq 1$.

\subsection{The Logistic Equation}\label{sec:logistic}
Consider the logistic equation
\begin{equation*}
\dot{y}(t) = 10 y(t)(1 - y(t)), \quad y(0) = y_0 = 15/100,
\end{equation*}
which has the following solution.
\begin{equation*}
y(t) = \frac{\exp(10t)}{\exp(10t) + 1/y_0 - 1}.
\end{equation*}
The approximate solutions are computed by EKS0, EKS1, and IEKS on the interval $[0,1]$ on a uniform, dense using, grid with interval length $2^{-12}$ using a prior in the class $\mathrm{IWP}(\mathrm{I},\nu)$, $\nu = 1,\ldots,4$. The filter updates only occur on a decimation of this dense grid by a factor of $2^{3+m}, \ m = 1,\ldots,8$, which yields the fill-distances $\delta_m = 2^{m-10}, \ m = 1,\ldots,8$. The $\mathcal{L}_\infty$ error of the zeroth and first derivative estimates of the methods are computed on the dense grid and compared to $\delta^\nu$ and $\delta^{\nu-1/2}$ (predicted rates), respectively. The errors of the approximate solutions versus fill-distance are shown in Figure \ref{fig:logistic_convergence} and it appears that EKS0, EKS1, and IEKS all attain at worst the predicted rates once $\delta$ is small enough. It appears the rate for EKS1/IEKS tapers off for $\nu = 4$ and small $\delta$. However, it can be verified that this is due to numerical instability when computing the smoothing gains as the prediction covariances $\Sigma_F(t_n^-)$ become numerically singular for too small $h_n$ (see \eqref{eq:dt_smoothing_gain}). The results are similar for the derivative of the approximate solution, see Figure \ref{fig:d1_logistic_convergence}.

\begin{figure}[t!]
\centering
\begin{tikzpicture}
  \begin{groupplot}
    [group style={group size= 2 by 2, horizontal sep=0.9cm}, height = 0.2\textwidth, width=0.5\textwidth]

    \nextgroupplot[
    width	= 0.25\textwidth,
    ymode = log,
    xmode = log,
    xmin = 0.002, xmax = 0.25,
    ymin = 1e-12, ymax = 1,
    grid = major,
    xtick = {1e-3,1e-2,1e-1},
    ytick = {1e-12,1e-9,1e-6,1e-3,1e-0},
    ticklabel style = {font=\tiny},
    xlabel style = {font=\scriptsize},
    ylabel style = {font=\scriptsize},
    ylabel = {$\abs[0]{\hat{y} - y^* }_{H^0_\infty}$},
    title style = {font=\footnotesize},
    title = {$\nu = 1$},
    ]

    \addplot[
      draw = black,
      line width = 0.8pt,
    ]
    table[
      x = h,
      y = 1,
    ]{logistic_refs.txt};

    \addplot[
      mark = *,
      mark repeat = {1},
      mark options = {scale=0.7,fill=red!50},
      draw = red!50,
      line width = 0.8pt,
    ]
    table[
      x = h,
      y = 1,
    ]{logistic_eks0_error.txt};

    \addplot[
      mark = square*,
      mark repeat = {1},
      mark options = {scale=0.7,fill=blue!50},
      draw = blue!50,
      line width = 0.8pt,
    ]
    table[
      x = h,
      y = 1,
    ]{logistic_eks1_error.txt};

    \addplot[
      mark = star,
      mark repeat = {1},
      mark options = {scale=0.7,fill=green!50},
      draw = green!50,
      line width = 0.8pt,
    ]
    table[
      x = h,
      y = 1,
    ]{logistic_ieks_error.txt};

    \nextgroupplot[
    width	= 0.25\textwidth,
    ymode = log,
    xmode = log,
    xmin = 0.002, xmax = 0.25,
    ymin = 1e-12, ymax = 1,
    grid = major,
    xtick = {1e-3,1e-2,1e-1},
    ytick = {1e-12,1e-9,1e-6,1e-3,1e-0},
    ticklabel style = {font=\tiny},
    xlabel style = {font=\scriptsize},
    ylabel style = {font=\scriptsize},
    title style = {font=\footnotesize},
    title = {$\nu = 2$},
    ]

    \addplot[
      draw = black,
      line width = 0.8pt,
    ]
    table[
      x = h,
      y = 2,
    ]{logistic_refs.txt};

    \addplot[
      mark = *,
      mark repeat = {1},
      mark options = {scale=0.7,fill=red!50},
      draw = red!50,
      line width = 0.8pt,
    ]
    table[
      x = h,
      y = 2,
    ]{logistic_eks0_error.txt};

    \addplot[
      mark = square*,
      mark repeat = {1},
      mark options = {scale=0.7,fill=blue!50},
      draw = blue!50,
      line width = 0.8pt,
    ]
    table[
      x = h,
      y = 2,
    ]{logistic_eks1_error.txt};

    \addplot[
      mark = star,
      mark repeat = {1},
      mark options = {scale=0.7,fill=green!50},
      draw = green!50,
      line width = 0.8pt,
    ]
    table[
      x = h,
      y = 2,
    ]{logistic_ieks_error.txt};

    \nextgroupplot[
    width	= 0.25\textwidth,
    ymode = log,
    xmode = log,
    xmin = 0.002, xmax = 0.25,
    ymin = 1e-12, ymax = 1,
    grid = major,
    xtick = {1e-3,1e-2,1e-1},
    ytick = {1e-12,1e-9,1e-6,1e-3,1e-0},
    ticklabel style = {font=\tiny},
    xlabel style = {font=\scriptsize},
    ylabel style = {font=\scriptsize},
    xlabel = {$\delta$},
    ylabel = {$\abs[0]{\hat{y} - y^* }_{H^0_\infty}$},
    title style = {font=\footnotesize},
    title = {$\nu = 3$},
    ]

    \addplot[
      draw = black,
      line width = 0.8pt,
    ]
    table[
      x = h,
      y = 3,
    ]{logistic_refs.txt};

    \addplot[
      mark = *,
      mark repeat = {1},
      mark options = {scale=0.7,fill=red!50},
      draw = red!50,
      line width = 0.8pt,
    ]
    table[
      x = h,
      y = 3,
    ]{logistic_eks0_error.txt};

    \addplot[
      mark = square*,
      mark repeat = {1},
      mark options = {scale=0.7,fill=blue!50},
      draw = blue!50,
      line width = 0.8pt,
    ]
    table[
      x = h,
      y = 3,
    ]{logistic_eks1_error.txt};

    \addplot[
      mark = star,
      mark repeat = {1},
      mark options = {scale=0.7,fill=green!50},
      draw = green!50,
      line width = 0.8pt,
    ]
    table[
      x = h,
      y = 3,
    ]{logistic_ieks_error.txt};

    \coordinate (c1) at (rel axis cs:0,1);

    \nextgroupplot[
    width	= 0.25\textwidth,
    ymode = log,
    xmode = log,
    xmin = 0.002, xmax = 0.25,
    ymin = 1e-12, ymax = 1,
    grid = major,
    xtick = {1e-3,1e-2,1e-1},
    ytick = {1e-12,1e-9,1e-6,1e-3,1e-0},
    ticklabel style = {font=\tiny},
    xlabel style = {font=\scriptsize},
    ylabel style = {font=\scriptsize},
    xlabel = {$\delta$},
    title style = {font=\footnotesize},
    title = {$\nu = 4$},
    legend style={at={($(0,0)+(1cm,1cm)$)},legend columns=4,fill=none,draw=black,anchor=center,align=center},
    legend to name=fred,
    legend style = {font=\footnotesize},
    ]

    \addplot[
      draw = black,
      line width = 0.8pt,
    ]
    table[
      x = h,
      y = 4,
    ]{logistic_refs.txt};

    \addplot[
      mark = *,
      mark repeat = {1},
      mark options = {scale=0.7,fill=red!50},
      draw = red!50,
      line width = 0.8pt,
    ]
    table[
      x = h,
      y = 4,
    ]{logistic_eks0_error.txt};

    \addplot[
      mark = square*,
      mark repeat = {1},
      mark options = {scale=0.7,fill=blue!50},
      draw = blue!50,
      line width = 0.8pt,
    ]
    table[
      x = h,
      y = 4,
    ]{logistic_eks1_error.txt};

    \addplot[
      mark = star,
      mark repeat = {1},
      mark options = {scale=0.7,fill=green!50},
      draw = green!50,
      line width = 0.8pt,
    ]
    table[
      x = h,
      y = 4,
    ]{logistic_ieks_error.txt};

    \coordinate (c2) at (rel axis cs:1,1);

    \addlegendentry{$\delta^\nu$};
    \addlegendentry{$\mathsf{EKS0}$};
    \addlegendentry{$\mathsf{EKS1}$};
    \addlegendentry{$\mathsf{IEKS}$};

  \end{groupplot}
  \coordinate (c3) at ($(c1)!.5!(c2)$);
    \node[below] at (c3 |- current bounding box.south)
      {\pgfplotslegendfromname{fred}};
\end{tikzpicture}
\caption{$\mathcal{L}_\infty$ error of the solution estimate as produced by EKS0 (red), EKS1 (blue), IEKS (green), and the predicted MAP rate $\delta^\nu$ (black), versus fill-distance.}\label{fig:logistic_convergence}
\end{figure}
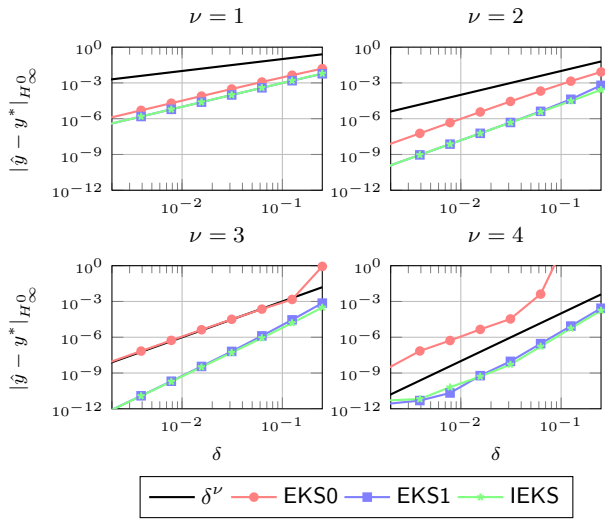

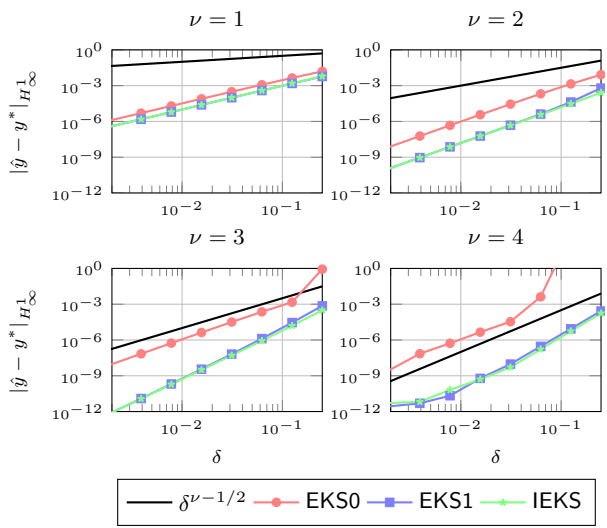
\begin{figure}[t!]
\centering
\begin{tikzpicture}
  \begin{groupplot}
    [group style={group size= 2 by 2, horizontal sep=0.9cm}, height = 0.2\textwidth, width=0.5\textwidth]

    \nextgroupplot[
    width	= 0.25\textwidth,
    ymode = log,
    xmode = log,
    xmin = 0.002, xmax = 0.25,
    ymin = 1e-12, ymax = 1,
    grid = major,
    xtick = {1e-3,1e-2,1e-1},
    ytick = {1e-12,1e-9,1e-6,1e-3,1e-0},
    ticklabel style = {font=\tiny},
    xlabel style = {font=\scriptsize},
    ylabel style = {font=\scriptsize},
    ylabel = {$\abs[0]{ \hat{y} - y^* }_{H^1_\infty}$},
    title style = {font=\footnotesize},
    title = {$\nu = 1$},
    ]

    \addplot[
      draw = black,
      line width = 0.8pt,
    ]
    table[
      x = h,
      y = 1,
    ]{d1_logistic_refs.txt};

    \addplot[
      mark = *,
      mark repeat = {1},
      mark options = {scale=0.7,fill=red!50},
      draw = red!50,
      line width = 0.8pt,
    ]
    table[
      x = h,
      y = 1,
    ]{d1_logistic_eks0_error.txt};

    \addplot[
      mark = square*,
      mark repeat = {1},
      mark options = {scale=0.7,fill=blue!50},
      draw = blue!50,
      line width = 0.8pt,
    ]
    table[
      x = h,
      y = 1,
    ]{d1_logistic_eks1_error.txt};

    \addplot[
      mark = star,
      mark repeat = {1},
      mark options = {scale=0.7,fill=green!50},
      draw = green!50,
      line width = 0.8pt,
    ]
    table[
      x = h,
      y = 1,
    ]{d1_logistic_ieks_error.txt};

    \nextgroupplot[
    width	= 0.25\textwidth,
    ymode = log,
    xmode = log,
    xmin = 0.002, xmax = 0.25,
    ymin = 1e-12, ymax = 1,
    grid = major,
    xtick = {1e-3,1e-2,1e-1},
    ytick = {1e-12,1e-9,1e-6,1e-3,1e-0},
    ticklabel style = {font=\tiny},
    xlabel style = {font=\scriptsize},
    ylabel style = {font=\scriptsize},
    title style = {font=\footnotesize},
    title = {$\nu = 2$},
    ]

    \addplot[
      draw = black,
      line width = 0.8pt,
    ]
    table[
      x = h,
      y = 2,
    ]{d1_logistic_refs.txt};

    \addplot[
      mark = *,
      mark repeat = {1},
      mark options = {scale=0.7,fill=red!50},
      draw = red!50,
      line width = 0.8pt,
    ]
    table[
      x = h,
      y = 2,
    ]{d1_logistic_eks0_error.txt};

    \addplot[
      mark = square*,
      mark repeat = {1},
      mark options = {scale=0.7,fill=blue!50},
      draw = blue!50,
      line width = 0.8pt,
    ]
    table[
      x = h,
      y = 2,
    ]{d1_logistic_eks1_error.txt};

    \addplot[
      mark = star,
      mark repeat = {1},
      mark options = {scale=0.7,fill=green!50},
      draw = green!50,
      line width = 0.8pt,
    ]
    table[
      x = h,
      y = 2,
    ]{d1_logistic_ieks_error.txt};

    \nextgroupplot[
    width	= 0.25\textwidth,
    ymode = log,
    xmode = log,
    xmin = 0.002, xmax = 0.25,
    ymin = 1e-12, ymax = 1,
    grid = major,
    xtick = {1e-3,1e-2,1e-1},
    ytick = {1e-12,1e-9,1e-6,1e-3,1e-0},
    ticklabel style = {font=\tiny},
    xlabel style = {font=\scriptsize},
    ylabel style = {font=\scriptsize},
    xlabel = {$\delta$},
    ylabel = {$\abs[0]{\hat{y} - y^* }_{H^1_\infty}$},
    title style = {font=\footnotesize},
    title = {$\nu = 3$},
    ]

    \addplot[
      draw = black,
      line width = 0.8pt,
    ]
    table[
      x = h,
      y = 3,
    ]{d1_logistic_refs.txt};

    \addplot[
      mark = *,
      mark repeat = {1},
      mark options = {scale=0.7,fill=red!50},
      draw = red!50,
      line width = 0.8pt,
    ]
    table[
      x = h,
      y = 3,
    ]{d1_logistic_eks0_error.txt};

    \addplot[
      mark = square*,
      mark repeat = {1},
      mark options = {scale=0.7,fill=blue!50},
      draw = blue!50,
      line width = 0.8pt,
    ]
    table[
      x = h,
      y = 3,
    ]{d1_logistic_eks1_error.txt};

    \addplot[
      mark = star,
      mark repeat = {1},
      mark options = {scale=0.7,fill=green!50},
      draw = green!50,
      line width = 0.8pt,
    ]
    table[
      x = h,
      y = 3,
    ]{d1_logistic_ieks_error.txt};

        \coordinate (c1) at (rel axis cs:0,1);

    \nextgroupplot[
    width	= 0.25\textwidth,
    ymode = log,
    xmode = log,
    xmin = 0.002, xmax = 0.25,
    ymin = 1e-12, ymax = 1,
    grid = major,
    xtick = {1e-3,1e-2,1e-1},
    ytick = {1e-12,1e-9,1e-6,1e-3,1e-0},
    ticklabel style = {font=\tiny},
    xlabel style = {font=\scriptsize},
    ylabel style = {font=\scriptsize},
    xlabel = {$\delta$},
    title style = {font=\footnotesize},
    title = {$\nu = 4$},
    legend style={at={($(0,0)+(1cm,1cm)$)},legend columns=4,fill=none,draw=black,anchor=center,align=center},
    legend to name=fred,
    legend style = {font=\footnotesize},
    ]

    \addplot[
      draw = black,
      line width = 0.8pt,
    ]
    table[
      x = h,
      y = 4,
    ]{d1_logistic_refs.txt};

    \addplot[
      mark = *,
      mark repeat = {1},
      mark options = {scale=0.7,fill=red!50},
      draw = red!50,
      line width = 0.8pt,
    ]
    table[
      x = h,
      y = 4,
    ]{d1_logistic_eks0_error.txt};

    \addplot[
      mark = square*,
      mark repeat = {1},
      mark options = {scale=0.7,fill=blue!50},
      draw = blue!50,
      line width = 0.8pt,
    ]
    table[
      x = h,
      y = 4,
    ]{d1_logistic_eks1_error.txt};

    \addplot[
      mark = star,
      mark repeat = {1},
      mark options = {scale=0.7,fill=green!50},
      draw = green!50,
      line width = 0.8pt,
    ]
    table[
      x = h,
      y = 4,
    ]{d1_logistic_ieks_error.txt};

    \coordinate (c2) at (rel axis cs:1,1);

    \addlegendentry{$\delta^{\nu-1/2}$};
    \addlegendentry{$\mathsf{EKS0}$};
    \addlegendentry{$\mathsf{EKS1}$};
    \addlegendentry{$\mathsf{IEKS}$};

  \end{groupplot}
  \coordinate (c3) at ($(c1)!.5!(c2)$);
    \node[below] at (c3 |- current bounding box.south)
      {\pgfplotslegendfromname{fred}};
\end{tikzpicture}
\caption{$\mathcal{L}_\infty$ error of the derivative estimate as produced by EKS0 (red), EKS1 (blue), IEKS (green), and the predicted MAP rate $\delta^{\nu-1/2}$ (black), versus fill-distance.}\label{fig:d1_logistic_convergence}
\end{figure}

Solution estimates by EKS0 and EKS1 are illustrutated in Figure \ref{fig:logistic_demo1} for $\nu = 2$ and $\delta = 2^{-4}$  (IEKS is very similar EKS1 and therefore not shown). The credible intervals are calibrated via the quasi maximum likelihood method, see \ref{app:uncertainty_calibration}. While both methods produce credible intervals that cover the true solution, those of EKS1 are much tighter. That is, here the EKS1 estimate is of higher quality than that of EKS0, which is particularly clear when looking at the derivative estimates.

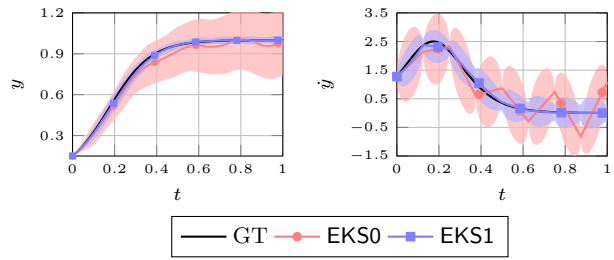
\begin{figure}[t!]
\centering
\begin{tikzpicture}
  \begin{groupplot}
    [group style={group size= 2 by 1, horizontal sep=1.5cm}, height = 0.2\textwidth, width=0.5\textwidth]

    \nextgroupplot[
    width	= 0.25\textwidth,
    xmin = 0, xmax = 1,
    ymin = 0.15, ymax = 1.2,
    grid = major,
    xtick = {0,0.2,0.4,0.6,0.8,1},
    ytick = {0,0.3,0.6,0.9,1.2},
    ticklabel style = {font=\tiny},
    xlabel style = {font=\scriptsize},
    ylabel style = {font=\scriptsize},
    xlabel = {$t$},
    ylabel = {$y$},
    title style = {font=\footnotesize},
    ]

    \addplot[
      draw = black,
      line width = 0.8pt,
    ]
    table[
      x = t,
      y = y,
    ]{logistic/logistic_demo1_gt.txt};

    \addplot[
      mark = *,
      mark repeat = {200},
      mark options = {scale=0.5,fill=red!50},
      draw = red!50,
      line width = 0.7pt,
      ]
      table[
        x = t,
        y = eks0,
      ]{logistic/logistic_demo1_eks0.txt};

     \addplot[name path=upper_eks0,draw=none] table[x=t,y=eks0ub]{logistic/logistic_demo1_eks0.txt};
     \addplot[name path=lower_eks0,draw=none] table[x=t,y=eks0lb]{logistic/logistic_demo1_eks0.txt};
     \addplot[fill=red!30,fill opacity=0.7] fill between[of=upper_eks0 and lower_eks0];

      \addplot[
        mark = square*,
        mark repeat = {200},
        mark options = {scale=0.5,fill=blue!50},
        draw = blue!50,
        line width = 0.7pt,
        ]
        table[
          x = t,
          y = eks1,
        ]{logistic/logistic_demo1_eks1.txt};

       \addplot[name path=upper_eks1,draw=none] table[x=t,y=eks1ub]{logistic/logistic_demo1_eks1.txt};
       \addplot[name path=lower_eks1,draw=none] table[x=t,y=eks1lb]{logistic/logistic_demo1_eks1.txt};
       \addplot[fill=blue!30,fill opacity=0.7] fill between[of=upper_eks1 and lower_eks1];

       \coordinate (c1) at (rel axis cs:0,1);

    \nextgroupplot[
    width	= 0.25\textwidth,
    xmin = 0, xmax = 1,
    ymin = -1.5, ymax = 3.5,
    grid = major,
    xtick = {0,0.2,0.4,0.6,0.8,1},
    ytick = {-1.5,-0.5,0.5,1.5,2.5,3.5},
    ticklabel style = {font=\tiny},
    xlabel style = {font=\scriptsize},
    ylabel style = {font=\scriptsize},
    xlabel = {$t$},
    ylabel = {$\dot{y}$},
    title style = {font=\footnotesize},
    legend style={at={($(0,0)+(1cm,1cm)$)},legend columns=3,fill=none,draw=black,anchor=center,align=center},
    legend to name=fred,
    legend style = {font=\footnotesize},
    ]

    \addplot[
      draw = black,
      line width = 0.8pt,
    ]
    table[
      x = t,
      y = dy,
    ]{logistic/logistic_demo1_gt.txt};

    \addplot[
      mark = *,
      mark repeat = {200},
      mark options = {scale=0.7,fill=red!50},
      draw = red!50,
      line width = 0.8pt,
      ]
      table[
        x = t,
        y = deks0,
      ]{logistic/logistic_demo1_eks0.txt};

      \addplot[name path=upper_deks0,draw=none,forget plot] table[x=t,y=deks0ub]{logistic/logistic_demo1_eks0.txt};
      \addplot[name path=lower_deks0,draw=none,forget plot] table[x=t,y=deks0lb]{logistic/logistic_demo1_eks0.txt};
      \addplot[fill=red!30, fill opacity=0.7,forget plot] fill between[of=upper_deks0 and lower_deks0];

      \addplot[
        mark = square*,
        mark repeat = {200},
        mark options = {scale=0.7,fill=blue!50},
        draw = blue!50,
        line width = 0.8pt,
        ]
        table[
          x = t,
          y = deks1,
        ]{logistic/logistic_demo1_eks1.txt};

        \addplot[name path=upper_deks1,draw=none,forget plot] table[x=t,y=deks1ub]{logistic/logistic_demo1_eks1.txt};
        \addplot[name path=lower_deks1,draw=none,forget plot] table[x=t,y=deks1lb]{logistic/logistic_demo1_eks1.txt};
        \addplot[fill=blue!30, fill opacity=0.7,forget plot] fill between[of=upper_deks1 and lower_deks1];

        \coordinate (c2) at (rel axis cs:1,1);

    \addlegendentry{$\mathrm{GT}$};
    \addlegendentry{$\mathsf{EKS0}$};
    \addlegendentry{$\mathsf{EKS1}$};

  \end{groupplot}
  \coordinate (c3) at ($(c1)!.5!(c2)$);
    \node[below] at (c3 |- current bounding box.south)
      {\pgfplotslegendfromname{fred}};
\end{tikzpicture}
\caption{Reconstruction of the logistic map (left) and its derivative (right) with two standard deviation credible bands for EKS0 (red) and EKS1 (blue).}\label{fig:logistic_demo1}
\end{figure}

\subsection{A Riccati Equation}\label{sec:riccati}
The convergence rates are examined for a Riccati equation as well. That is, consider the following ODE
\begin{equation*}
\dot{y}(t) = -c\frac{y^3(t)}{2}, \quad y(0) = y_0 = 1,
\end{equation*}
which has the following solution
\begin{equation*}
y(t) = \frac{1}{\sqrt{ct + 1/y_0^2}}.
\end{equation*}
Just as for the logistic map, the solution is approximated by EKS0, EKS1, and IEKS on the interval $[0,1]$, using a $\mathrm{IWP}(\mathrm{I},\nu)$, $\nu = 1,\ldots,4$, for various fill-distances $\delta$. The $\mathcal{L}_\infty$ errors of the zeroth and first derivative estimates are shown in Figures \ref{fig:riccati_convergence} and \ref{fig:d1_riccati_convergence}, respectively. The general results are the same as before, EKS1 and IEKS are very similar, and EKS0 is some orders of magnitude worse while still appearing to converge at a similar rate as the former. The numerical instability in the computation of smoothing gains is still present for large $\nu$ and small $\delta$.

Additionally, the output of the solvers for $\nu = 2$ is visualised for step-sizes of $h = 0.125$ and $h = 0.25$  in Figures \ref{fig:riccati_demo1} and \ref{fig:riccati_demo2}, respectively. It can be seen that already for $h = 0.25$, the solution estimate and uncertainty quantification of the IEKS, while EKS0 and EKS1 leave room for improvement in terms of both accuracy and uncertainty quantification.  By halving the step-size EKS1 and IEKS become near identical (wherefore IEKS is not shown in Figure \ref{fig:riccati_demo1}), though the error of the EKS0 is still oscillating quite a bit, particularly for the derivative.

\begin{figure}[t!]
\centering
\begin{tikzpicture}
  \begin{groupplot}
    [group style={group size= 2 by 2, horizontal sep=0.9cm}, height = 0.2\textwidth, width=0.5\textwidth]

    \nextgroupplot[
    width	= 0.25\textwidth,
    ymode = log,
    xmode = log,
    xmin = 0.0039, xmax = 0.25,
    ymin = 1e-12, ymax = 1,
    grid = major,
    xtick = {1e-3,1e-2,1e-1},
    ytick = {1e-12,1e-9,1e-6,1e-3,1e-0},
    ticklabel style = {font=\tiny},
    xlabel style = {font=\scriptsize},
    ylabel style = {font=\scriptsize},
    ylabel = {$\abs[0]{\hat{y} - y^* }_{H^0_\infty}$},
    title style = {font=\footnotesize},
    title = {$\nu = 1$},
    ]

    \addplot[
      draw = black,
      line width = 0.8pt,
    ]
    table[
      x = h,
      y = 1,
    ]{riccati_refs.txt};

    \addplot[
      mark = *,
      mark repeat = {1},
      mark options = {scale=0.7,fill=red!50},
      draw = red!50,
      line width = 0.8pt,
    ]
    table[
      x = h,
      y = 1,
    ]{riccati_eks0_error.txt};

    \addplot[
      mark = square*,
      mark repeat = {1},
      mark options = {scale=0.7,fill=blue!50},
      draw = blue!50,
      line width = 0.8pt,
    ]
    table[
      x = h,
      y = 1,
    ]{riccati_eks1_error.txt};

    \addplot[
      mark = star,
      mark repeat = {1},
      mark options = {scale=0.7,fill=green!50},
      draw = green!50,
      line width = 0.8pt,
    ]
    table[
      x = h,
      y = 1,
    ]{riccati_ieks_error.txt};

    \nextgroupplot[
    width	= 0.25\textwidth,
    ymode = log,
    xmode = log,
    xmin = 0.0039, xmax = 0.25,
    ymin = 1e-12, ymax = 1,
    grid = major,
    xtick = {1e-3,1e-2,1e-1},
    ytick = {1e-12,1e-9,1e-6,1e-3,1e-0},
    ticklabel style = {font=\tiny},
    xlabel style = {font=\scriptsize},
    ylabel style = {font=\scriptsize},
    title style = {font=\footnotesize},
    title = {$\nu = 2$},
    ]

    \addplot[
      draw = black,
      line width = 0.8pt,
    ]
    table[
      x = h,
      y = 2,
    ]{riccati_refs.txt};

    \addplot[
      mark = *,
      mark repeat = {1},
      mark options = {scale=0.7,fill=red!50},
      draw = red!50,
      line width = 0.8pt,
    ]
    table[
      x = h,
      y = 2,
    ]{riccati_eks0_error.txt};

    \addplot[
      mark = square*,
      mark repeat = {1},
      mark options = {scale=0.7,fill=blue!50},
      draw = blue!50,
      line width = 0.8pt,
    ]
    table[
      x = h,
      y = 2,
    ]{riccati_eks1_error.txt};

    \addplot[
      mark = star,
      mark repeat = {1},
      mark options = {scale=0.7,fill=green!50},
      draw = green!50,
      line width = 0.8pt,
    ]
    table[
      x = h,
      y = 2,
    ]{riccati_ieks_error.txt};

    \nextgroupplot[
    width	= 0.25\textwidth,
    ymode = log,
    xmode = log,
    xmin = 0.0039, xmax = 0.25,
    ymin = 1e-12, ymax = 1,
    grid = major,
    xtick = {1e-3,1e-2,1e-1},
    ytick = {1e-12,1e-9,1e-6,1e-3,1e-0},
    ticklabel style = {font=\tiny},
    xlabel style = {font=\scriptsize},
    ylabel style = {font=\scriptsize},
    xlabel = {$\delta$},
    ylabel = {$\abs[0]{\hat{y} - y^* }_{H^0_\infty}$},
    title style = {font=\footnotesize},
    title = {$\nu = 3$},
    ]

    \addplot[
      draw = black,
      line width = 0.8pt,
    ]
    table[
      x = h,
      y = 3,
    ]{riccati_refs.txt};

    \addplot[
      mark = *,
      mark repeat = {1},
      mark options = {scale=0.7,fill=red!50},
      draw = red!50,
      line width = 0.8pt,
    ]
    table[
      x = h,
      y = 3,
    ]{riccati_eks0_error.txt};

    \addplot[
      mark = square*,
      mark repeat = {1},
      mark options = {scale=0.7,fill=blue!50},
      draw = blue!50,
      line width = 0.8pt,
    ]
    table[
      x = h,
      y = 3,
    ]{riccati_eks1_error.txt};

    \addplot[
      mark = star,
      mark repeat = {1},
      mark options = {scale=0.7,fill=green!50},
      draw = green!50,
      line width = 0.8pt,
    ]
    table[
      x = h,
      y = 3,
    ]{riccati_ieks_error.txt};

    \coordinate (c1) at (rel axis cs:0,1);

    \nextgroupplot[
    width	= 0.25\textwidth,
    ymode = log,
    xmode = log,
    xmin = 0.0039, xmax = 0.25,
    ymin = 1e-12, ymax = 1,
    grid = major,
    xtick = {1e-3,1e-2,1e-1},
    ytick = {1e-12,1e-9,1e-6,1e-3,1e-0},
    ticklabel style = {font=\tiny},
    xlabel style = {font=\scriptsize},
    ylabel style = {font=\scriptsize},
    xlabel = {$\delta$},
    title style = {font=\footnotesize},
    title = {$\nu = 4$},
    legend style={at={($(0,0)+(1cm,1cm)$)},legend columns=4,fill=none,draw=black,anchor=center,align=center},
    legend to name=fred,
    legend style = {font=\footnotesize},
    ]

    \addplot[
      draw = black,
      line width = 0.8pt,
    ]
    table[
      x = h,
      y = 4,
    ]{riccati_refs.txt};

    \addplot[
      mark = *,
      mark repeat = {1},
      mark options = {scale=0.7,fill=red!50},
      draw = red!50,
      line width = 0.8pt,
    ]
    table[
      x = h,
      y = 4,
    ]{riccati_eks0_error.txt};

    \addplot[
      mark = square*,
      mark repeat = {1},
      mark options = {scale=0.7,fill=blue!50},
      draw = blue!50,
      line width = 0.8pt,
    ]
    table[
      x = h,
      y = 4,
    ]{riccati_eks1_error.txt};

    \addplot[
      mark = star,
      mark repeat = {1},
      mark options = {scale=0.7,fill=green!50},
      draw = green!50,
      line width = 0.8pt,
    ]
    table[
      x = h,
      y = 4,
    ]{riccati_ieks_error.txt};

    \coordinate (c2) at (rel axis cs:1,1);

    \addlegendentry{$\delta^\nu$};
    \addlegendentry{$\mathsf{EKS0}$};
    \addlegendentry{$\mathsf{EKS1}$};
    \addlegendentry{$\mathsf{IEKS}$};

  \end{groupplot}
  \coordinate (c3) at ($(c1)!.5!(c2)$);
    \node[below] at (c3 |- current bounding box.south)
      {\pgfplotslegendfromname{fred}};
\end{tikzpicture}
\caption{$\mathcal{L}_\infty$ error of the solution estimate as produced by EKS0 (red), EKS1 (blue), IEKS (green), and the predicted MAP rate $\delta^\nu$ (black), versus fill-distance.}\label{fig:riccati_convergence}
\end{figure}
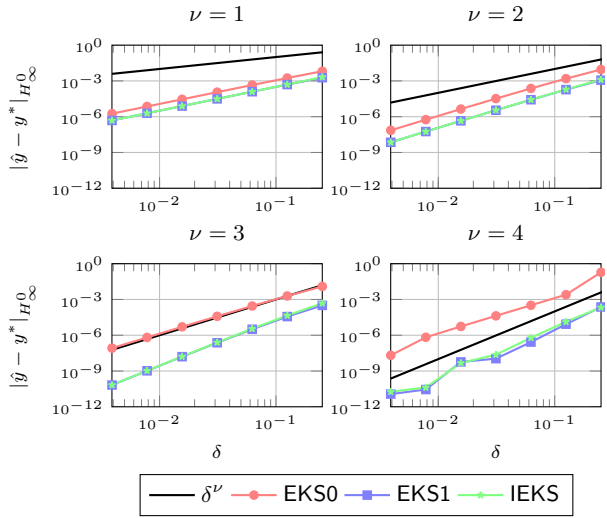

\begin{figure}[t!]
\centering
\begin{tikzpicture}
  \begin{groupplot}
    [group style={group size= 2 by 2, horizontal sep=0.9cm}, height = 0.2\textwidth, width=0.5\textwidth]

    \nextgroupplot[
    width	= 0.25\textwidth,
    ymode = log,
    xmode = log,
    xmin = 0.0039, xmax = 0.25,
    ymin = 1e-9, ymax = 1,
    grid = major,
    xtick = {1e-3,1e-2,1e-1},
    ytick = {1e-9,1e-6,1e-3,1e-0},
    ticklabel style = {font=\tiny},
    xlabel style = {font=\scriptsize},
    ylabel style = {font=\scriptsize},
    ylabel = {$\abs[0]{ \hat{y} - y^* }_{H^1_\infty}$},
    title style = {font=\footnotesize},
    title = {$\nu = 1$},
    ]

    \addplot[
      draw = black,
      line width = 0.8pt,
    ]
    table[
      x = h,
      y = 1,
    ]{d1_riccati_refs.txt};

    \addplot[
      mark = *,
      mark repeat = {1},
      mark options = {scale=0.7,fill=red!50},
      draw = red!50,
      line width = 0.8pt,
    ]
    table[
      x = h,
      y = 1,
    ]{d1_riccati_eks0_error.txt};

    \addplot[
      mark = square*,
      mark repeat = {1},
      mark options = {scale=0.7,fill=blue!50},
      draw = blue!50,
      line width = 0.8pt,
    ]
    table[
      x = h,
      y = 1,
    ]{d1_riccati_eks1_error.txt};

    \addplot[
      mark = star,
      mark repeat = {1},
      mark options = {scale=0.7,fill=green!50},
      draw = green!50,
      line width = 0.8pt,
    ]
    table[
      x = h,
      y = 1,
    ]{d1_riccati_ieks_error.txt};

    \nextgroupplot[
    width	= 0.25\textwidth,
    ymode = log,
    xmode = log,
    xmin = 0.0039, xmax = 0.25,
    ymin = 1e-9, ymax = 1,
    grid = major,
    xtick = {1e-3,1e-2,1e-1},
    ytick = {1e-9,1e-6,1e-3,1e-0},
    ticklabel style = {font=\tiny},
    xlabel style = {font=\scriptsize},
    ylabel style = {font=\scriptsize},
    title style = {font=\footnotesize},
    title = {$\nu = 2$},
    ]

    \addplot[
      draw = black,
      line width = 0.8pt,
    ]
    table[
      x = h,
      y = 2,
    ]{d1_riccati_refs.txt};

    \addplot[
      mark = *,
      mark repeat = {1},
      mark options = {scale=0.7,fill=red!50},
      draw = red!50,
      line width = 0.8pt,
    ]
    table[
      x = h,
      y = 2,
    ]{d1_riccati_eks0_error.txt};

    \addplot[
      mark = square*,
      mark repeat = {1},
      mark options = {scale=0.7,fill=blue!50},
      draw = blue!50,
      line width = 0.8pt,
    ]
    table[
      x = h,
      y = 2,
    ]{d1_riccati_eks1_error.txt};

    \addplot[
      mark = star,
      mark repeat = {1},
      mark options = {scale=0.7,fill=green!50},
      draw = green!50,
      line width = 0.8pt,
    ]
    table[
      x = h,
      y = 2,
    ]{d1_riccati_ieks_error.txt};

    \nextgroupplot[
    width	= 0.25\textwidth,
    ymode = log,
    xmode = log,
    xmin = 0.0039, xmax = 0.25,
    ymin = 1e-9, ymax = 1,
    grid = major,
    xtick = {1e-3,1e-2,1e-1},
    ytick = {1e-9,1e-6,1e-3,1e-0},
    ticklabel style = {font=\tiny},
    xlabel style = {font=\scriptsize},
    ylabel style = {font=\scriptsize},
    xlabel = {$\delta$},
    ylabel = {$\abs[0]{\hat{y} - y^* }_{H^1_\infty}$},
    title style = {font=\footnotesize},
    title = {$\nu = 3$},
    ]

    \addplot[
      draw = black,
      line width = 0.8pt,
    ]
    table[
      x = h,
      y = 3,
    ]{d1_riccati_refs.txt};

    \addplot[
      mark = *,
      mark repeat = {1},
      mark options = {scale=0.7,fill=red!50},
      draw = red!50,
      line width = 0.8pt,
    ]
    table[
      x = h,
      y = 3,
    ]{d1_riccati_eks0_error.txt};

    \addplot[
      mark = square*,
      mark repeat = {1},
      mark options = {scale=0.7,fill=blue!50},
      draw = blue!50,
      line width = 0.8pt,
    ]
    table[
      x = h,
      y = 3,
    ]{d1_riccati_eks1_error.txt};

    \addplot[
      mark = star,
      mark repeat = {1},
      mark options = {scale=0.7,fill=green!50},
      draw = green!50,
      line width = 0.8pt,
    ]
    table[
      x = h,
      y = 3,
    ]{d1_riccati_ieks_error.txt};

        \coordinate (c1) at (rel axis cs:0,1);

    \nextgroupplot[
    width	= 0.25\textwidth,
    ymode = log,
    xmode = log,
    xmin = 0.0039, xmax = 0.25,
    ymin = 1e-9, ymax = 1,
    grid = major,
    xtick = {1e-3,1e-2,1e-1},
    ytick = {1e-9,1e-6,1e-3,1e-0},
    ticklabel style = {font=\tiny},
    xlabel style = {font=\scriptsize},
    ylabel style = {font=\scriptsize},
    xlabel = {$\delta$},
    title style = {font=\footnotesize},
    title = {$\nu = 4$},
    legend style={at={($(0,0)+(1cm,1cm)$)},legend columns=4,fill=none,draw=black,anchor=center,align=center},
    legend to name=fred,
    legend style = {font=\footnotesize},
    ]

    \addplot[
      draw = black,
      line width = 0.8pt,
    ]
    table[
      x = h,
      y = 4,
    ]{d1_riccati_refs.txt};

    \addplot[
      mark = *,
      mark repeat = {1},
      mark options = {scale=0.7,fill=red!50},
      draw = red!50,
      line width = 0.8pt,
    ]
    table[
      x = h,
      y = 4,
    ]{d1_riccati_eks0_error.txt};

    \addplot[
      mark = square*,
      mark repeat = {1},
      mark options = {scale=0.7,fill=blue!50},
      draw = blue!50,
      line width = 0.8pt,
    ]
    table[
      x = h,
      y = 4,
    ]{d1_riccati_eks1_error.txt};

    \addplot[
      mark = star,
      mark repeat = {1},
      mark options = {scale=0.7,fill=green!50},
      draw = green!50,
      line width = 0.8pt,
    ]
    table[
      x = h,
      y = 4,
    ]{d1_riccati_ieks_error.txt};

    \coordinate (c2) at (rel axis cs:1,1);

    \addlegendentry{$\delta^{\nu-1/2}$};
    \addlegendentry{$\mathsf{EKS0}$};
    \addlegendentry{$\mathsf{EKS1}$};
    \addlegendentry{$\mathsf{IEKS}$};

  \end{groupplot}
  \coordinate (c3) at ($(c1)!.5!(c2)$);
    \node[below] at (c3 |- current bounding box.south)
      {\pgfplotslegendfromname{fred}};
\end{tikzpicture}
\caption{$\mathcal{L}_\infty$ error of the derivative estimate as produced by EKS0 (red), EKS1 (blue), IEKS (green), and the predicted MAP rate $\delta^{\nu-1/2}$ (black), versus fill-distance.}\label{fig:d1_riccati_convergence}
\end{figure}
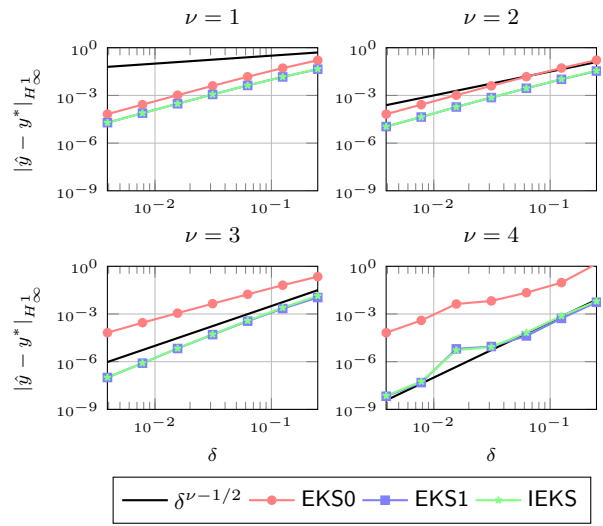

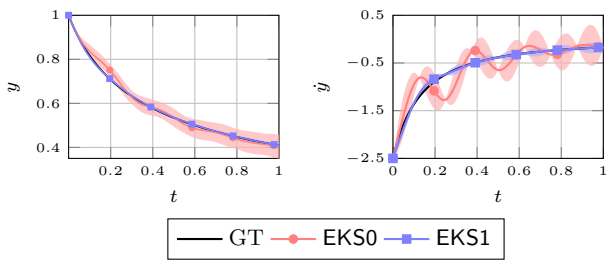
\begin{figure}[t!]
\centering
\begin{tikzpicture}
  \begin{groupplot}
    [group style={group size= 2 by 1, horizontal sep=1.5cm}, height = 0.2\textwidth, width=0.5\textwidth]

    \nextgroupplot[
    width	= 0.25\textwidth,
    xmin = 0, xmax = 1,
    ymin = 0.35, ymax = 1,
    grid = major,
    xtick = {0.2,0.4,0.6,0.8,1},
    ytick = {0.4,0.6,0.8,1},
    ticklabel style = {font=\tiny},
    xlabel style = {font=\scriptsize},
    ylabel style = {font=\scriptsize},
    xlabel = {$t$},
    ylabel = {$y$},
    title style = {font=\footnotesize},
    ]

    \addplot[
      draw = black,
      line width = 0.7pt,
    ]
    table[
      x = t,
      y = y,
    ]{riccati_demo1_gt.txt};

    \addplot[
      mark = *,
      mark repeat = {200},
      mark options = {scale=0.5,fill=red!50},
      draw = red!50,
      line width = 0.6pt,
      ]
      table[
        x = t,
        y = eks0,
      ]{riccati_demo1_eks0.txt};

     \addplot[name path=upper_eks0,draw=none] table[x=t,y=eks0ub]{riccati_demo1_eks0.txt};
     \addplot[name path=lower_eks0,draw=none] table[x=t,y=eks0lb]{riccati_demo1_eks0.txt};
     \addplot[fill=red!30,fill opacity=0.7] fill between[of=upper_eks0 and lower_eks0];

      \addplot[
        mark = square*,
        mark repeat = {200},
        mark options = {scale=0.5,fill=blue!50},
        draw = blue!50,
        line width = 0.6pt,
        ]
        table[
          x = t,
          y = eks1,
        ]{riccati_demo1_eks1.txt};

       \addplot[name path=upper_eks1,draw=none] table[x=t,y=eks1ub]{riccati_demo1_eks1.txt};
       \addplot[name path=lower_eks1,draw=none] table[x=t,y=eks1lb]{riccati_demo1_eks1.txt};
       \addplot[fill=blue!30,fill opacity=1] fill between[of=upper_eks1 and lower_eks1];

       \coordinate (c1) at (rel axis cs:0,1);

    \nextgroupplot[
    width	= 0.25\textwidth,
    xmin = 0, xmax = 1,
    ymin = -2.5, ymax = 0.5,
    grid = major,
    xtick = {0,0.2,0.4,0.6,0.8,1},
    ytick = {-2.5,-1.5,-0.5,0.5},
    ticklabel style = {font=\tiny},
    xlabel style = {font=\scriptsize},
    ylabel style = {font=\scriptsize},
    xlabel = {$t$},
    ylabel = {$\dot{y}$},
    title style = {font=\footnotesize},
    legend style={at={($(0,0)+(1cm,1cm)$)},legend columns=3,fill=none,draw=black,anchor=center,align=center},
    legend to name=fred,
    legend style = {font=\footnotesize},
    ]

    \addplot[
      draw = black,
      line width = 0.8pt,
    ]
    table[
      x = t,
      y = dy,
    ]{riccati_demo1_gt.txt};

    \addplot[
      mark = *,
      mark repeat = {200},
      mark options = {scale=0.7,fill=red!50},
      draw = red!50,
      line width = 0.8pt,
      ]
      table[
        x = t,
        y = deks0,
      ]{riccati_demo1_eks0.txt};

      \addplot[name path=upper_deks0,draw=none,forget plot] table[x=t,y=deks0ub]{riccati_demo1_eks0.txt};
      \addplot[name path=lower_deks0,draw=none,forget plot] table[x=t,y=deks0lb]{riccati_demo1_eks0.txt};
      \addplot[fill=red!30, fill opacity=0.7,forget plot] fill between[of=upper_deks0 and lower_deks0];

      \addplot[
        mark = square*,
        mark repeat = {200},
        mark options = {scale=0.7,fill=blue!50},
        draw = blue!50,
        line width = 0.8pt,
        ]
        table[
          x = t,
          y = deks1,
        ]{riccati_demo1_eks1.txt};

        \addplot[name path=upper_deks1,draw=none,forget plot] table[x=t,y=deks1ub]{riccati_demo1_eks1.txt};
        \addplot[name path=lower_deks1,draw=none,forget plot] table[x=t,y=deks1lb]{riccati_demo1_eks1.txt};
        \addplot[fill=blue!30, fill opacity=0.7,forget plot] fill between[of=upper_deks1 and lower_deks1];

        \coordinate (c2) at (rel axis cs:1,1);

    \addlegendentry{$\mathrm{GT}$};
    \addlegendentry{$\mathsf{EKS0}$};
    \addlegendentry{$\mathsf{EKS1}$};

  \end{groupplot}
  \coordinate (c3) at ($(c1)!.5!(c2)$);
    \node[below] at (c3 |- current bounding box.south)
      {\pgfplotslegendfromname{fred}};
\end{tikzpicture}
\caption{Reconstruction of the Riccati map (left) and its derivative (right) with two standard deviation credible bands for EKS0 (red) and EKS1 (blue), using a step size of $h = 0.125$.}\label{fig:riccati_demo1}
\end{figure}

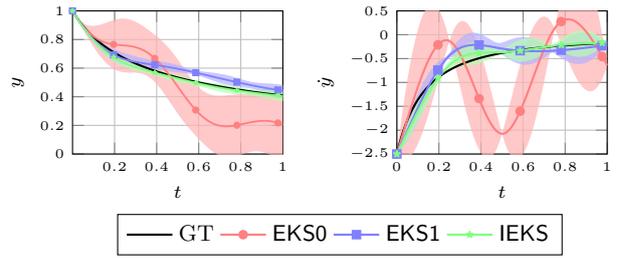
\begin{figure}[t!]
\centering
\begin{tikzpicture}
  \begin{groupplot}
    [group style={group size= 2 by 1, horizontal sep=1.5cm}, height = 0.2\textwidth, width=0.5\textwidth]

    \nextgroupplot[
    width	= 0.25\textwidth,
    xmin = 0, xmax = 1,
    ymin = 0, ymax = 1,
    grid = major,
    xtick = {0.2,0.4,0.6,0.8,1},
    ytick = {0,0.2,0.4,0.6,0.8,1},
    ticklabel style = {font=\tiny},
    xlabel style = {font=\scriptsize},
    ylabel style = {font=\scriptsize},
    xlabel = {$t$},
    ylabel = {$y$},
    title style = {font=\footnotesize},
    ]

    \addplot[
      draw = black,
      line width = 0.8pt,
    ]
    table[
      x = t,
      y = y,
    ]{riccati_demo2_gt.txt};

    \addplot[
      mark = *,
      mark repeat = {200},
      mark options = {scale=0.5,fill=red!50},
      draw = red!50,
      line width = 0.7pt,
      ]
      table[
        x = t,
        y = eks0,
      ]{riccati_demo2_eks0.txt};

     \addplot[name path=upper_eks0,draw=none] table[x=t,y=eks0ub]{riccati_demo2_eks0.txt};
     \addplot[name path=lower_eks0,draw=none] table[x=t,y=eks0lb]{riccati_demo2_eks0.txt};
     \addplot[fill=red!30,fill opacity=0.7] fill between[of=upper_eks0 and lower_eks0];

     \addplot[
       mark = square*,
       mark repeat = {200},
       mark options = {scale=0.5,fill=blue!50},
       draw = blue!50,
       line width = 0.7pt,
       ]
       table[
         x = t,
         y = eks1,
       ]{riccati_demo2_eks1.txt};

      \addplot[name path=upper_eks1,draw=none] table[x=t,y=eks1ub]{riccati_demo2_eks1.txt};
      \addplot[name path=lower_eks1,draw=none] table[x=t,y=eks1lb]{riccati_demo2_eks1.txt};
      \addplot[fill=blue!30,fill opacity=0.7] fill between[of=upper_eks1 and lower_eks1];

      \addplot[
        mark = star,
        mark repeat = {200},
        mark options = {scale=0.5,fill=green!50},
        draw = green!50,
        line width = 0.7pt,
        ]
        table[
          x = t,
          y = ieks,
        ]{riccati_demo2_ieks.txt};

       \addplot[name path=upper_ieks,draw=none] table[x=t,y=ieksub]{riccati_demo2_ieks.txt};
       \addplot[name path=lower_ieks,draw=none] table[x=t,y=iekslb]{riccati_demo2_ieks.txt};
       \addplot[fill=green!30,fill opacity=0.7] fill between[of=upper_ieks and lower_ieks];

       \coordinate (c1) at (rel axis cs:0,1);

    \nextgroupplot[
    width	= 0.25\textwidth,
    xmin = 0, xmax = 1,
    ymin = -2.5, ymax = 0.5,
    grid = major,
    xtick = {0,0.2,0.4,0.6,0.8,1},
    ytick = {-2.5,-2,-1.5,-1,-0.5,0,0.5},
    ticklabel style = {font=\tiny},
    xlabel style = {font=\scriptsize},
    ylabel style = {font=\scriptsize},
    xlabel = {$t$},
    ylabel = {$\dot{y}$},
    title style = {font=\footnotesize},
    legend style={at={($(0,0)+(1cm,1cm)$)},legend columns=4,fill=none,draw=black,anchor=center,align=center},
    legend to name=fred,
    legend style = {font=\footnotesize},
    ]

    \addplot[
      draw = black,
      line width = 0.8pt,
    ]
    table[
      x = t,
      y = dy,
    ]{riccati_demo2_gt.txt};

    \addplot[
      mark = *,
      mark repeat = {200},
      mark options = {scale=0.7,fill=red!50},
      draw = red!50,
      line width = 0.8pt,
      ]
      table[
        x = t,
        y = deks0,
      ]{riccati_demo2_eks0.txt};

      \addplot[name path=upper_deks0,draw=none,forget plot] table[x=t,y=deks0ub]{riccati_demo2_eks0.txt};
      \addplot[name path=lower_deks0,draw=none,forget plot] table[x=t,y=deks0lb]{riccati_demo2_eks0.txt};
      \addplot[fill=red!30, fill opacity=0.7,forget plot] fill between[of=upper_deks0 and lower_deks0];

      \addplot[
        mark = square*,
        mark repeat = {200},
        mark options = {scale=0.7,fill=blue!50},
        draw = blue!50,
        line width = 0.8pt,
        ]
        table[
          x = t,
          y = deks1,
        ]{riccati_demo2_eks1.txt};

        \addplot[name path=upper_deks1,draw=none,forget plot] table[x=t,y=deks1ub]{riccati_demo2_eks1.txt};
        \addplot[name path=lower_deks1,draw=none,forget plot] table[x=t,y=deks1lb]{riccati_demo2_eks1.txt};
        \addplot[fill=blue!30, fill opacity=0.7,forget plot] fill between[of=upper_deks1 and lower_deks1];

        \addplot[
          mark = star,
          mark repeat = {200},
          mark options = {scale=0.7,fill=green!50},
          draw = green!50,
          line width = 0.8pt,
          ]
          table[
            x = t,
            y = dieks,
          ]{riccati_demo2_ieks.txt};

          \addplot[name path=upper_dieks,draw=none,forget plot] table[x=t,y=dieksub]{riccati_demo2_ieks.txt};
          \addplot[name path=lower_dieks,draw=none,forget plot] table[x=t,y=diekslb]{riccati_demo2_ieks.txt};
          \addplot[fill=green!30, fill opacity=0.7,forget plot] fill between[of=upper_dieks and lower_dieks];

        \coordinate (c2) at (rel axis cs:1,1);

    \addlegendentry{$\mathrm{GT}$};
    \addlegendentry{$\mathsf{EKS0}$};
    \addlegendentry{$\mathsf{EKS1}$};
    \addlegendentry{$\mathsf{IEKS}$};

  \end{groupplot}
  \coordinate (c3) at ($(c1)!.5!(c2)$);
    \node[below] at (c3 |- current bounding box.south)
      {\pgfplotslegendfromname{fred}};
\end{tikzpicture}
\caption{Reconstruction of the Riccati map (left) and its derivative (right) with two standard deviation credible bands for EKS0 (red), EKS1 (blue), and IEKS (green), using a step size of $h = 0.25$.}\label{fig:riccati_demo2}
\end{figure}

\subsection{The Fitz--Hugh--Nagumo Model}\label{sec:fitzhughnagumo}
Consider the Fitz--Hugh--Nagumo model, which is given by
\begin{equation}
D \begin{pmatrix} y_1(t) \\ y_2(t) \end{pmatrix} = \begin{pmatrix} c ( y_1(t) - y_1^3(t)/3 + y_2(t) ) \\ - \frac{1}{c}( y_1(t) - a + b y_2(t) ) \end{pmatrix}.
\end{equation}
The initial conditions and parameters are set to $y_2(0) = -y_1(0) = 1$, and $(a,b,c) = (0.2,0.2,2)$, respectively. The solution is estimated by EKS0, EKS1, and IEKS with an $\mathrm{IWP}(\mathrm{I},\nu)$ prior $(1\leq \nu \leq 4)$ on a uniform grid with $2^{12}+1$ points on the interval $[0,2.5]$, using the same decimation scheme as previously. As this ODE does not have a closed form solution, it is approximated with $\texttt{ode45}$\footnote{This is an adaptive embedded Runge--Kutta $4/5$ method.} in MATLAB, which is called with the parameters $\mathtt{RelTol} = 10^{-14}$, and $\mathtt{AbsTol} = 10^{-14}$. The approximate $\mathcal{L}_2$ error of the zeroth and first order derivative estimates of $y_1^*$ are shown in Figures \ref{fig:fitz-hugh-nagumo_convergence} and \ref{fig:d1_fitz-hugh-nagumo_convergence}, respectively.  The results appear to be consistent with the findings from the previous experiments.

Examples of the solver output of EKS1 and IEKS for $\nu = 2$ and $h = 0.4375$ is in Figures \ref{fig:fitz-hugh-nagumo_demo1} and \ref{fig:fitz-hugh-nagumo_demo2} for the first and second coordinates of $y$, respectively. The estimate and uncertainty quantification of the IEKS can be seen to be quite good, except for a slight undershoot in the estimate of $\dot{y}_1$ at $t = 1$. The performance of EKS1 is poorer, it overshoots quite a bit in its estimate of $y_1$ at around $t = 1.5$, which is not appropriately reflected in its credible interval.

\begin{figure}[t!]
\centering
\begin{tikzpicture}
  \begin{groupplot}
    [group style={group size= 2 by 2, horizontal sep=0.9cm}, height = 0.2\textwidth, width=0.5\textwidth]

    \nextgroupplot[
    width	= 0.25\textwidth,
    ymode = log,
    xmode = log,
    xmin = 0.0049, xmax = 0.625,
    ymin = 1e-13, ymax = 10,
    grid = major,
    xtick = {1e-3,1e-2,1e-1},
    ytick = {1e-11,1e-8,1e-5,1e-2,1e+1},
    ticklabel style = {font=\tiny},
    xlabel style = {font=\scriptsize},
    ylabel style = {font=\scriptsize},
    ylabel = {$\abs[0]{\hat{y}_1 - y_1^* }_{H^0_\infty}$},
    title style = {font=\footnotesize},
    title = {$\nu = 1$},
    ]

    \addplot[
      draw = black,
      line width = 0.8pt,
    ]
    table[
      x = h,
      y = 1,
    ]{fitz-hugh-nagumo_refs1.txt};

    \addplot[
      mark = *,
      mark repeat = {1},
      mark options = {scale=0.7,fill=red!50},
      draw = red!50,
      line width = 0.8pt,
    ]
    table[
      x = h,
      y = 1,
    ]{fitz-hugh-nagumo_eks0_error1.txt};

    \addplot[
      mark = square*,
      mark repeat = {1},
      mark options = {scale=0.7,fill=blue!50},
      draw = blue!50,
      line width = 0.8pt,
    ]
    table[
      x = h,
      y = 1,
    ]{fitz-hugh-nagumo_eks1_error1.txt};

    \addplot[
      mark = star,
      mark repeat = {1},
      mark options = {scale=0.7,fill=green!50},
      draw = green!50,
      line width = 0.8pt,
    ]
    table[
      x = h,
      y = 1,
    ]{fitz-hugh-nagumo_ieks_error1.txt};

    \nextgroupplot[
    width	= 0.25\textwidth,
    ymode = log,
    xmode = log,
    xmin = 0.0049, xmax = 0.625,
    ymin = 1e-13, ymax = 10,
    grid = major,
    xtick = {1e-3,1e-2,1e-1},
    ytick = {1e-11,1e-8,1e-5,1e-2,1e+1},
    ticklabel style = {font=\tiny},
    xlabel style = {font=\scriptsize},
    ylabel style = {font=\scriptsize},
    title style = {font=\footnotesize},
    title = {$\nu = 2$},
    ]

    \addplot[
      draw = black,
      line width = 0.8pt,
    ]
    table[
      x = h,
      y = 2,
    ]{fitz-hugh-nagumo_refs1.txt};

    \addplot[
      mark = *,
      mark repeat = {1},
      mark options = {scale=0.7,fill=red!50},
      draw = red!50,
      line width = 0.8pt,
    ]
    table[
      x = h,
      y = 2,
    ]{fitz-hugh-nagumo_eks0_error1.txt};

    \addplot[
      mark = square*,
      mark repeat = {1},
      mark options = {scale=0.7,fill=blue!50},
      draw = blue!50,
      line width = 0.8pt,
    ]
    table[
      x = h,
      y = 2,
    ]{fitz-hugh-nagumo_eks1_error1.txt};

    \addplot[
      mark = star,
      mark repeat = {1},
      mark options = {scale=0.7,fill=green!50},
      draw = green!50,
      line width = 0.8pt,
    ]
    table[
      x = h,
      y = 2,
    ]{fitz-hugh-nagumo_ieks_error1.txt};

    \nextgroupplot[
    width	= 0.25\textwidth,
    ymode = log,
    xmode = log,
    xmin = 0.0049, xmax = 0.625,
    ymin = 1e-13, ymax = 10,
    grid = major,
    xtick = {1e-3,1e-2,1e-1},
    ytick = {1e-11,1e-8,1e-5,1e-2,1e+1},
    ticklabel style = {font=\tiny},
    xlabel style = {font=\scriptsize},
    ylabel style = {font=\scriptsize},
    xlabel = {$\delta$},
    ylabel = {$\abs[0]{\hat{y}_1 - y_1^* }_{H^0_\infty}$},
    title style = {font=\footnotesize},
    title = {$\nu = 3$},
    ]

    \addplot[
      draw = black,
      line width = 0.8pt,
    ]
    table[
      x = h,
      y = 3,
    ]{fitz-hugh-nagumo_refs1.txt};

    \addplot[
      mark = *,
      mark repeat = {1},
      mark options = {scale=0.7,fill=red!50},
      draw = red!50,
      line width = 0.8pt,
    ]
    table[
      x = h,
      y = 3,
    ]{fitz-hugh-nagumo_eks0_error1.txt};

    \addplot[
      mark = square*,
      mark repeat = {1},
      mark options = {scale=0.7,fill=blue!50},
      draw = blue!50,
      line width = 0.8pt,
    ]
    table[
      x = h,
      y = 3,
    ]{fitz-hugh-nagumo_eks1_error1.txt};

    \addplot[
      mark = star,
      mark repeat = {1},
      mark options = {scale=0.7,fill=green!50},
      draw = green!50,
      line width = 0.8pt,
    ]
    table[
      x = h,
      y = 3,
    ]{fitz-hugh-nagumo_ieks_error1.txt};

    \coordinate (c1) at (rel axis cs:0,1);

    \nextgroupplot[
    width	= 0.25\textwidth,
    ymode = log,
    xmode = log,
    xmin = 0.0049, xmax = 0.625,
    ymin = 1e-13, ymax = 10,
    grid = major,
    xtick = {1e-3,1e-2,1e-1},
    ytick = {1e-11,1e-8,1e-5,1e-2,1e+1},
    ticklabel style = {font=\tiny},
    xlabel style = {font=\scriptsize},
    ylabel style = {font=\scriptsize},
    xlabel = {$\delta$},
    title style = {font=\footnotesize},
    title = {$\nu = 4$},
    legend style={at={($(0,0)+(1cm,1cm)$)},legend columns=4,fill=none,draw=black,anchor=center,align=center},
    legend to name=fred,
    legend style = {font=\footnotesize},
    ]

    \addplot[
      draw = black,
      line width = 0.8pt,
    ]
    table[
      x = h,
      y = 4,
    ]{fitz-hugh-nagumo_refs1.txt};

    \addplot[
      mark = *,
      mark repeat = {1},
      mark options = {scale=0.7,fill=red!50},
      draw = red!50,
      line width = 0.8pt,
    ]
    table[
      x = h,
      y = 4,
    ]{fitz-hugh-nagumo_eks0_error1.txt};

    \addplot[
      mark = square*,
      mark repeat = {1},
      mark options = {scale=0.7,fill=blue!50},
      draw = blue!50,
      line width = 0.8pt,
    ]
    table[
      x = h,
      y = 4,
    ]{fitz-hugh-nagumo_eks1_error1.txt};

    \addplot[
      mark = star,
      mark repeat = {1},
      mark options = {scale=0.7,fill=green!50},
      draw = green!50,
      line width = 0.8pt,
    ]
    table[
      x = h,
      y = 4,
    ]{fitz-hugh-nagumo_ieks_error1.txt};

    \coordinate (c2) at (rel axis cs:1,1);

    \addlegendentry{$\delta^\nu$};
    \addlegendentry{$\mathsf{EKS0}$};
    \addlegendentry{$\mathsf{EKS1}$};
    \addlegendentry{$\mathsf{IEKS}$};

  \end{groupplot}
  \coordinate (c3) at ($(c1)!.5!(c2)$);
    \node[below] at (c3 |- current bounding box.south)
      {\pgfplotslegendfromname{fred}};
\end{tikzpicture}
\caption{$\mathcal{L}_\infty$ error of the solution estimate as produced by EKS0 (red), EKS1 (blue), IEKS (green), and the predicted MAP rate $\delta^\nu$ (black), versus fill-distance.}\label{fig:fitz-hugh-nagumo_convergence}
\end{figure}
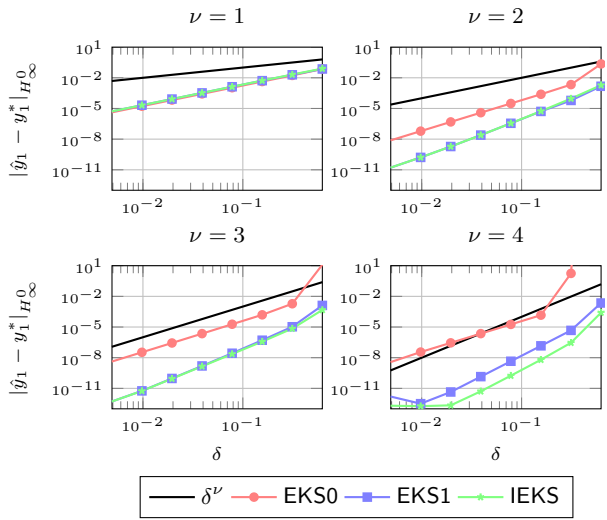

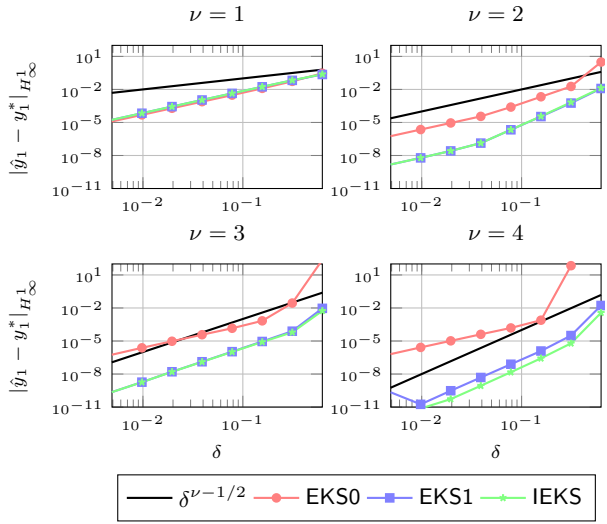
\begin{figure}[t!]
\centering
\begin{tikzpicture}
  \begin{groupplot}
    [group style={group size= 2 by 2, horizontal sep=0.9cm}, height = 0.2\textwidth, width=0.5\textwidth]

    \nextgroupplot[
    width	= 0.25\textwidth,
    ymode = log,
    xmode = log,
    xmin = 0.0049, xmax = 0.625,
    ymin = 1e-11, ymax = 100,
    grid = major,
    xtick = {1e-3,1e-2,1e-1},
    ytick = {1e-11,1e-8,1e-5,1e-2,1e+1},
    ticklabel style = {font=\tiny},
    xlabel style = {font=\scriptsize},
    ylabel style = {font=\scriptsize},
    ylabel = {$\abs[0]{ \hat{y}_1 - y_1^* }_{H^1_\infty}$},
    title style = {font=\footnotesize},
    title = {$\nu = 1$},
    ]

    \addplot[
      draw = black,
      line width = 0.8pt,
    ]
    table[
      x = h,
      y = 1,
    ]{d1_fitz-hugh-nagumo_refs1.txt};

    \addplot[
      mark = *,
      mark repeat = {1},
      mark options = {scale=0.7,fill=red!50},
      draw = red!50,
      line width = 0.8pt,
    ]
    table[
      x = h,
      y = 1,
    ]{d1_fitz-hugh-nagumo_eks0_error1.txt};

    \addplot[
      mark = square*,
      mark repeat = {1},
      mark options = {scale=0.7,fill=blue!50},
      draw = blue!50,
      line width = 0.8pt,
    ]
    table[
      x = h,
      y = 1,
    ]{d1_fitz-hugh-nagumo_eks1_error1.txt};

    \addplot[
      mark = star,
      mark repeat = {1},
      mark options = {scale=0.7,fill=green!50},
      draw = green!50,
      line width = 0.8pt,
    ]
    table[
      x = h,
      y = 1,
    ]{d1_fitz-hugh-nagumo_ieks_error1.txt};

    \nextgroupplot[
    width	= 0.25\textwidth,
    ymode = log,
    xmode = log,
    xmin = 0.0049, xmax = 0.625,
    ymin = 1e-11, ymax = 100,
    grid = major,
    xtick = {1e-3,1e-2,1e-1},
    ytick = {1e-11,1e-8,1e-5,1e-2,1e+1},
    ticklabel style = {font=\tiny},
    xlabel style = {font=\scriptsize},
    ylabel style = {font=\scriptsize},
    title style = {font=\footnotesize},
    title = {$\nu = 2$},
    ]

    \addplot[
      draw = black,
      line width = 0.8pt,
    ]
    table[
      x = h,
      y = 2,
    ]{d1_fitz-hugh-nagumo_refs1.txt};

    \addplot[
      mark = *,
      mark repeat = {1},
      mark options = {scale=0.7,fill=red!50},
      draw = red!50,
      line width = 0.8pt,
    ]
    table[
      x = h,
      y = 2,
    ]{d1_fitz-hugh-nagumo_eks0_error1.txt};

    \addplot[
      mark = square*,
      mark repeat = {1},
      mark options = {scale=0.7,fill=blue!50},
      draw = blue!50,
      line width = 0.8pt,
    ]
    table[
      x = h,
      y = 2,
    ]{d1_fitz-hugh-nagumo_eks1_error1.txt};

    \addplot[
      mark = star,
      mark repeat = {1},
      mark options = {scale=0.7,fill=green!50},
      draw = green!50,
      line width = 0.8pt,
    ]
    table[
      x = h,
      y = 2,
    ]{d1_fitz-hugh-nagumo_ieks_error1.txt};

    \nextgroupplot[
    width	= 0.25\textwidth,
    ymode = log,
    xmode = log,
    xmin = 0.0049, xmax = 0.625,
    ymin = 1e-11, ymax = 100,
    grid = major,
    xtick = {1e-3,1e-2,1e-1},
    ytick = {1e-11,1e-8,1e-5,1e-2,1e+1},
    ticklabel style = {font=\tiny},
    xlabel style = {font=\scriptsize},
    ylabel style = {font=\scriptsize},
    xlabel = {$\delta$},
    ylabel = {$\abs[0]{\hat{y}_1 - y_1^* }_{H^1_\infty}$},
    title style = {font=\footnotesize},
    title = {$\nu = 3$},
    ]

    \addplot[
      draw = black,
      line width = 0.8pt,
    ]
    table[
      x = h,
      y = 3,
    ]{d1_fitz-hugh-nagumo_refs1.txt};

    \addplot[
      mark = *,
      mark repeat = {1},
      mark options = {scale=0.7,fill=red!50},
      draw = red!50,
      line width = 0.8pt,
    ]
    table[
      x = h,
      y = 3,
    ]{d1_fitz-hugh-nagumo_eks0_error1.txt};

    \addplot[
      mark = square*,
      mark repeat = {1},
      mark options = {scale=0.7,fill=blue!50},
      draw = blue!50,
      line width = 0.8pt,
    ]
    table[
      x = h,
      y = 3,
    ]{d1_fitz-hugh-nagumo_eks1_error1.txt};

    \addplot[
      mark = star,
      mark repeat = {1},
      mark options = {scale=0.7,fill=green!50},
      draw = green!50,
      line width = 0.8pt,
    ]
    table[
      x = h,
      y = 3,
    ]{d1_fitz-hugh-nagumo_ieks_error1.txt};

        \coordinate (c1) at (rel axis cs:0,1);

    \nextgroupplot[
    width	= 0.25\textwidth,
    ymode = log,
    xmode = log,
    xmin = 0.0049, xmax = 0.625,
    ymin = 1e-11, ymax = 100,
    grid = major,
    xtick = {1e-3,1e-2,1e-1},
    ytick = {1e-11,1e-8,1e-5,1e-2,1e+1},
    ticklabel style = {font=\tiny},
    xlabel style = {font=\scriptsize},
    ylabel style = {font=\scriptsize},
    xlabel = {$\delta$},
    title style = {font=\footnotesize},
    title = {$\nu = 4$},
    legend style={at={($(0,0)+(1cm,1cm)$)},legend columns=4,fill=none,draw=black,anchor=center,align=center},
    legend to name=fred,
    legend style = {font=\footnotesize},
    ]

    \addplot[
      draw = black,
      line width = 0.8pt,
    ]
    table[
      x = h,
      y = 4,
    ]{d1_fitz-hugh-nagumo_refs1.txt};

    \addplot[
      mark = *,
      mark repeat = {1},
      mark options = {scale=0.7,fill=red!50},
      draw = red!50,
      line width = 0.8pt,
    ]
    table[
      x = h,
      y = 4,
    ]{d1_fitz-hugh-nagumo_eks0_error1.txt};

    \addplot[
      mark = square*,
      mark repeat = {1},
      mark options = {scale=0.7,fill=blue!50},
      draw = blue!50,
      line width = 0.8pt,
    ]
    table[
      x = h,
      y = 4,
    ]{d1_fitz-hugh-nagumo_eks1_error1.txt};

    \addplot[
      mark = star,
      mark repeat = {1},
      mark options = {scale=0.7,fill=green!50},
      draw = green!50,
      line width = 0.8pt,
    ]
    table[
      x = h,
      y = 4,
    ]{d1_fitz-hugh-nagumo_ieks_error1.txt};

    \coordinate (c2) at (rel axis cs:1,1);

    \addlegendentry{$\delta^{\nu-1/2}$};
    \addlegendentry{$\mathsf{EKS0}$};
    \addlegendentry{$\mathsf{EKS1}$};
    \addlegendentry{$\mathsf{IEKS}$};

  \end{groupplot}
  \coordinate (c3) at ($(c1)!.5!(c2)$);
    \node[below] at (c3 |- current bounding box.south)
      {\pgfplotslegendfromname{fred}};
\end{tikzpicture}
\caption{$\mathcal{L}_\infty$ error of the derivative estimate as produced by EKS0 (red), EKS1 (blue), IEKS (green), and the predicted MAP rate $\delta^{\nu-1/2}$ (black), versus fill-distance.}\label{fig:d1_fitz-hugh-nagumo_convergence}
\end{figure}

\begin{figure}[t!]
\centering
\begin{tikzpicture}
  \begin{groupplot}
    [group style={group size= 2 by 1, horizontal sep=1.5cm}, height = 0.2\textwidth, width=0.5\textwidth]

    \nextgroupplot[
    width	= 0.25\textwidth,
    xmin = 0, xmax = 2.5,
    ymin = -1, ymax = 2.5,
    grid = major,
    xtick = {0,0.5,1,1.5,2,2.5},
    ytick = {-1,-0.5,0,0.5,1,1.5,2,2.5},
    ticklabel style = {font=\tiny},
    xlabel style = {font=\scriptsize},
    ylabel style = {font=\scriptsize},
    xlabel = {$t$},
    ylabel = {$y_1$},
    title style = {font=\footnotesize},
    ]

    \addplot[
      draw = black,
      line width = 0.8pt,
    ]
    table[
      x = t,
      y = y,
    ]{fitz-hugh-nagumo_demo1_gt.txt};

     \addplot[
       mark = square*,
       mark repeat = {200},
       mark options = {scale=0.5,fill=blue!50},
       draw = blue!50,
       line width = 0.7pt,
       ]
       table[
         x = t,
         y = eks1,
       ]{fitz-hugh-nagumo_demo1_eks1.txt};

      \addplot[name path=upper_eks1,draw=none] table[x=t,y=eks1ub]{fitz-hugh-nagumo_demo1_eks1.txt};
      \addplot[name path=lower_eks1,draw=none] table[x=t,y=eks1lb]{fitz-hugh-nagumo_demo1_eks1.txt};
      \addplot[fill=blue!30,fill opacity=0.7] fill between[of=upper_eks1 and lower_eks1];

      \addplot[
        mark = star,
        mark repeat = {200},
        mark options = {scale=0.5,fill=green!50},
        draw = green!50,
        line width = 0.7pt,
        ]
        table[
          x = t,
          y = ieks,
        ]{fitz-hugh-nagumo_demo1_ieks.txt};

       \addplot[name path=upper_ieks,draw=none] table[x=t,y=ieksub]{fitz-hugh-nagumo_demo1_ieks.txt};
       \addplot[name path=lower_ieks,draw=none] table[x=t,y=iekslb]{fitz-hugh-nagumo_demo1_ieks.txt};
       \addplot[fill=green!30,fill opacity=0.7] fill between[of=upper_ieks and lower_ieks];

       \coordinate (c1) at (rel axis cs:0,1);

    \nextgroupplot[
    width	= 0.25\textwidth,
    xmin = 0, xmax = 2.5,
    ymin = -2, ymax = 4,
    grid = major,
    xtick = {0,0.5,1,1.5,2,2.5},
    ytick = {-2,-1,0,1,2,3,4},
    ticklabel style = {font=\tiny},
    xlabel style = {font=\scriptsize},
    ylabel style = {font=\scriptsize},
    xlabel = {$t$},
    ylabel = {$\dot{y}_1$},
    title style = {font=\footnotesize},
    legend style={at={($(0,0)+(1cm,1cm)$)},legend columns=3,fill=none,draw=black,anchor=center,align=center},
    legend to name=fred,
    legend style = {font=\footnotesize},
    ]

    \addplot[
      draw = black,
      line width = 0.8pt,
    ]
    table[
      x = t,
      y = dy,
    ]{fitz-hugh-nagumo_demo1_gt.txt};

      \addplot[
        mark = square*,
        mark repeat = {200},
        mark options = {scale=0.7,fill=blue!50},
        draw = blue!50,
        line width = 0.8pt,
        ]
        table[
          x = t,
          y = deks1,
        ]{fitz-hugh-nagumo_demo1_eks1.txt};

        \addplot[name path=upper_deks1,draw=none,forget plot] table[x=t,y=deks1ub]{fitz-hugh-nagumo_demo1_eks1.txt};
        \addplot[name path=lower_deks1,draw=none,forget plot] table[x=t,y=deks1lb]{fitz-hugh-nagumo_demo1_eks1.txt};
        \addplot[fill=blue!30, fill opacity=0.7,forget plot] fill between[of=upper_deks1 and lower_deks1];

        \addplot[
          mark = star,
          mark repeat = {200},
          mark options = {scale=0.7,fill=green!50},
          draw = green!50,
          line width = 0.8pt,
          ]
          table[
            x = t,
            y = dieks,
          ]{fitz-hugh-nagumo_demo1_ieks.txt};

          \addplot[name path=upper_dieks,draw=none,forget plot] table[x=t,y=dieksub]{fitz-hugh-nagumo_demo1_ieks.txt};
          \addplot[name path=lower_dieks,draw=none,forget plot] table[x=t,y=diekslb]{fitz-hugh-nagumo_demo1_ieks.txt};
          \addplot[fill=green!30, fill opacity=0.7,forget plot] fill between[of=upper_dieks and lower_dieks];

        \coordinate (c2) at (rel axis cs:1,1);

    \addlegendentry{$\mathrm{GT}$};
    \addlegendentry{$\mathsf{EKS1}$};
    \addlegendentry{$\mathsf{IEKS}$};

  \end{groupplot}
  \coordinate (c3) at ($(c1)!.5!(c2)$);
    \node[below] at (c3 |- current bounding box.south)
      {\pgfplotslegendfromname{fred}};
\end{tikzpicture}
\caption{Reconstruction of the first coordinate, $y_1$, in the Fitz--Hugh--Nagumo model (left) and its derivative (right) with two standard deviation credible bands for EKS1 (blue) and IEKS (green), using a step size of $h = 0.25$.}\label{fig:fitz-hugh-nagumo_demo1}
\end{figure}
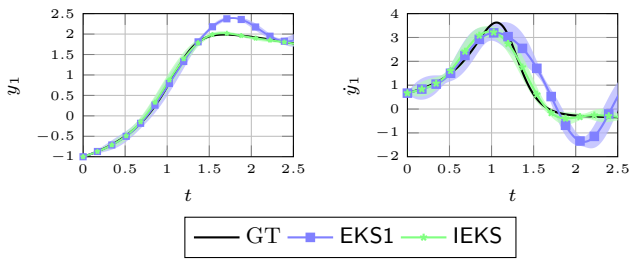

\begin{figure}[t!]
\centering
\begin{tikzpicture}
  \begin{groupplot}
    [group style={group size= 2 by 1, horizontal sep=1.5cm}, height = 0.2\textwidth, width=0.5\textwidth]

    \nextgroupplot[
    width	= 0.25\textwidth,
    xmin = 0, xmax = 2.5,
    ymin = -0.5, ymax = 1.5,
    grid = major,
    xtick = {0,0.5,1,1.5,2,2.5},
    ytick = {-0.5,0,0.5,1,1.5},
    ticklabel style = {font=\tiny},
    xlabel style = {font=\scriptsize},
    ylabel style = {font=\scriptsize},
    xlabel = {$t$},
    ylabel = {$y_2$},
    title style = {font=\footnotesize},
    ]

    \addplot[
      draw = black,
      line width = 0.8pt,
    ]
    table[
      x = t,
      y = y,
    ]{fitz-hugh-nagumo_demo2_gt.txt};

     \addplot[
       mark = square*,
       mark repeat = {200},
       mark options = {scale=0.5,fill=blue!50},
       draw = blue!50,
       line width = 0.7pt,
       ]
       table[
         x = t,
         y = eks1,
       ]{fitz-hugh-nagumo_demo2_eks1.txt};

      \addplot[name path=upper_eks1,draw=none] table[x=t,y=eks1ub]{fitz-hugh-nagumo_demo2_eks1.txt};
      \addplot[name path=lower_eks1,draw=none] table[x=t,y=eks1lb]{fitz-hugh-nagumo_demo2_eks1.txt};
      \addplot[fill=blue!30,fill opacity=0.7] fill between[of=upper_eks1 and lower_eks1];

      \addplot[
        mark = star,
        mark repeat = {200},
        mark options = {scale=0.5,fill=green!50},
        draw = green!50,
        line width = 0.7pt,
        ]
        table[
          x = t,
          y = ieks,
        ]{fitz-hugh-nagumo_demo2_ieks.txt};

       \addplot[name path=upper_ieks,draw=none] table[x=t,y=ieksub]{fitz-hugh-nagumo_demo2_ieks.txt};
       \addplot[name path=lower_ieks,draw=none] table[x=t,y=iekslb]{fitz-hugh-nagumo_demo2_ieks.txt};
       \addplot[fill=green!30,fill opacity=0.7] fill between[of=upper_ieks and lower_ieks];

       \coordinate (c1) at (rel axis cs:0,1);

    \nextgroupplot[
    width	= 0.25\textwidth,
    xmin = 0, xmax = 2.5,
    ymin = -2, ymax = 1,
    grid = major,
    xtick = {0,0.5,1,1.5,2,2.5},
    ytick = {-2,-1,0,1},
    ticklabel style = {font=\tiny},
    xlabel style = {font=\scriptsize},
    ylabel style = {font=\scriptsize},
    xlabel = {$t$},
    ylabel = {$\dot{y}_2$},
    title style = {font=\footnotesize},
    legend style={at={($(0,0)+(1cm,1cm)$)},legend columns=3,fill=none,draw=black,anchor=center,align=center},
    legend to name=fred,
    legend style = {font=\footnotesize},
    ]

    \addplot[
      draw = black,
      line width = 0.8pt,
    ]
    table[
      x = t,
      y = dy,
    ]{fitz-hugh-nagumo_demo2_gt.txt};

      \addplot[
        mark = square*,
        mark repeat = {200},
        mark options = {scale=0.7,fill=blue!50},
        draw = blue!50,
        line width = 0.8pt,
        ]
        table[
          x = t,
          y = deks1,
        ]{fitz-hugh-nagumo_demo2_eks1.txt};

        \addplot[name path=upper_deks1,draw=none,forget plot] table[x=t,y=deks1ub]{fitz-hugh-nagumo_demo2_eks1.txt};
        \addplot[name path=lower_deks1,draw=none,forget plot] table[x=t,y=deks1lb]{fitz-hugh-nagumo_demo2_eks1.txt};
        \addplot[fill=blue!30, fill opacity=0.7,forget plot] fill between[of=upper_deks1 and lower_deks1];

        \addplot[
          mark = star,
          mark repeat = {200},
          mark options = {scale=0.7,fill=green!50},
          draw = green!50,
          line width = 0.8pt,
          ]
          table[
            x = t,
            y = dieks,
          ]{fitz-hugh-nagumo_demo2_ieks.txt};

          \addplot[name path=upper_dieks,draw=none,forget plot] table[x=t,y=dieksub]{fitz-hugh-nagumo_demo2_ieks.txt};
          \addplot[name path=lower_dieks,draw=none,forget plot] table[x=t,y=diekslb]{fitz-hugh-nagumo_demo2_ieks.txt};
          \addplot[fill=green!30, fill opacity=0.7,forget plot] fill between[of=upper_dieks and lower_dieks];

        \coordinate (c2) at (rel axis cs:1,1);

    \addlegendentry{$\mathrm{GT}$};
    \addlegendentry{$\mathsf{EKS1}$};
    \addlegendentry{$\mathsf{IEKS}$};

  \end{groupplot}
  \coordinate (c3) at ($(c1)!.5!(c2)$);
    \node[below] at (c3 |- current bounding box.south)
      {\pgfplotslegendfromname{fred}};
\end{tikzpicture}
\caption{Reconstruction of the first coordinate, $y_2$, in the Fitz--Hugh--Nagumo model  (left) and its derivative (right) with two standard deviation credible bands for EKS1 (blue) and IEKS (green), using a step size of $h = 0.25$.}\label{fig:fitz-hugh-nagumo_demo2}
\end{figure}
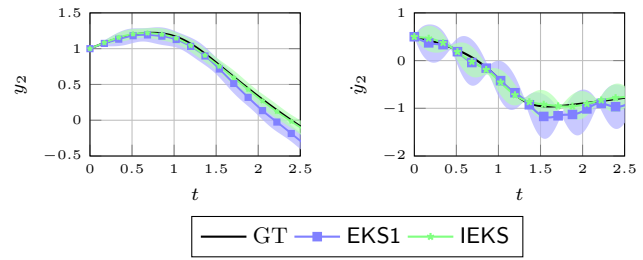

\subsection{A Non-smooth Example}\label{sec:nonsmooth}
Let the vector field $f$ be given by
\begin{equation}
f(y) =
\begin{cases}
\kappa, \quad y \leq b, \\
\kappa + \lambda(y - b), \quad y > b,
\end{cases}
\end{equation}
and consider the following ODE:
\begin{equation}
\dot{y}(t) = f(y(t)), \quad y(0) = y_0 \leq b.
\end{equation}
If $\kappa > 0$, the solution is given by
\begin{equation}\label{eq:nonsmooth_solution}
y^*(t) =
\begin{cases}
 y_0 + \kappa t, \quad t \leq \tau^*,\\
 b + \frac{1}{\lambda}(\exp(\lambda(t-\tau^*)) - 1) \kappa, \quad t > \tau^*,
\end{cases}
\end{equation}
where $\tau^* = (b-y_0)/\kappa$. While $f$ is continuous, it has a discontinuity in its derivative at $y = b$ and therefore Assumption \ref{ass:smooth} is violated for all $\nu \geq 1$. Nonetheless the solution is approximated by EKS0, EKS1, and IEKS using an $\mathrm{IWP}$ prior of smoothness $0 \leq \nu \leq 4$, and the parameters are set to $y_0 = 0$, $b = 1$, $\kappa = 2(b- y_0)$, and $\lambda = - 5$. The $\mathcal{L}_\infty$ errors of the zeroth and first derivative of the approximate solutions are shown in Figures \ref{fig:nonsmooth_convergence} and \ref{fig:d1_nonsmooth_convergence}, respectively. Additionally, a comparison of the solver outputs of EKS1 and IEKS is shown in Figure \ref{fig:nonsmooth_demo} for $\nu = 2$ and $h = 0.25$.

The estimates still appear to converge as seen in Figures \ref{fig:nonsmooth_convergence} and \ref{fig:d1_nonsmooth_convergence}. However, while the rate predicted by Theorem \ref{thm:convergence} appears to still be obtained for $\nu = 1$, a rate reduction is observed for $\nu > 1$ (in comparison to the rate of Theorem \ref{thm:convergence}). As Assumption \ref{ass:smooth} is violated, these results cannot be explained by the present theory.

However, note that Theorem \ref{thm:convergence} was obtained using $y^* \in \mathbb{Y}$ (Corollary \ref{cor:correct_model}) and $\mathcal{S}_f$ is locally Lipschitz (Theorem \ref{thm:nemytsky_properties}), together with the sampling inequalities of \cite{Arcangeli2007}. These properties of $f$ and $y^*$ may be obtainable by other means than invoking Assumption \ref{ass:smooth}. This could explain the results for $\nu = 1$.

On the other hand, in the setting of numerical integration, reduction in convergence rates when the RKHS is smoother than the integrand has been investigated by \citealt{Kanagawa2020a}. If these results can be extended to the setting of solving ODEs, it could explain the results for $\nu > 1$.

\begin{figure}[t!]
\centering
\begin{tikzpicture}
  \begin{groupplot}
    [group style={group size= 2 by 2, horizontal sep=0.9cm}, height = 0.2\textwidth, width=0.5\textwidth]

    \nextgroupplot[
    width	= 0.25\textwidth,
    ymode = log,
    xmode = log,
    xmin = 0.0039, xmax = 0.125,
    ymin = 1e-10, ymax = 1,
    grid = major,
    xtick = {1e-3,1e-2,1e-1},
    ytick = {1e-12,1e-9,1e-6,1e-3,1e-0},
    ticklabel style = {font=\tiny},
    xlabel style = {font=\scriptsize},
    ylabel style = {font=\scriptsize},
    ylabel = {$\abs[0]{\hat{y} - y^* }_{H^0_\infty}$},
    title style = {font=\footnotesize},
    title = {$\nu = 1$},
    ]

    \addplot[
      draw = black,
      line width = 0.8pt,
    ]
    table[
      x = h,
      y = 1,
    ]{nonsmooth_refs.txt};

    \addplot[
      mark = *,
      mark repeat = {1},
      mark options = {scale=0.7,fill=red!50},
      draw = red!50,
      line width = 0.8pt,
    ]
    table[
      x = h,
      y = 1,
    ]{nonsmooth_eks0_error.txt};

    \addplot[
      mark = square*,
      mark repeat = {1},
      mark options = {scale=0.7,fill=blue!50},
      draw = blue!50,
      line width = 0.8pt,
    ]
    table[
      x = h,
      y = 1,
    ]{nonsmooth_eks1_error.txt};

    \addplot[
      mark = star,
      mark repeat = {1},
      mark options = {scale=0.7,fill=green!50},
      draw = green!50,
      line width = 0.8pt,
    ]
    table[
      x = h,
      y = 1,
    ]{nonsmooth_ieks_error.txt};

    \nextgroupplot[
    width	= 0.25\textwidth,
    ymode = log,
    xmode = log,
    xmin = 0.0039, xmax = 0.125,
    ymin = 1e-10, ymax = 1,
    grid = major,
    xtick = {1e-3,1e-2,1e-1},
    ytick = {1e-12,1e-9,1e-6,1e-3,1e-0},
    ticklabel style = {font=\tiny},
    xlabel style = {font=\scriptsize},
    ylabel style = {font=\scriptsize},
    title style = {font=\footnotesize},
    title = {$\nu = 2$},
    ]

    \addplot[
      draw = black,
      line width = 0.8pt,
    ]
    table[
      x = h,
      y = 2,
    ]{nonsmooth_refs.txt};

    \addplot[
      mark = *,
      mark repeat = {1},
      mark options = {scale=0.7,fill=red!50},
      draw = red!50,
      line width = 0.8pt,
    ]
    table[
      x = h,
      y = 2,
    ]{nonsmooth_eks0_error.txt};

    \addplot[
      mark = square*,
      mark repeat = {1},
      mark options = {scale=0.7,fill=blue!50},
      draw = blue!50,
      line width = 0.8pt,
    ]
    table[
      x = h,
      y = 2,
    ]{nonsmooth_eks1_error.txt};

    \addplot[
      mark = star,
      mark repeat = {1},
      mark options = {scale=0.7,fill=green!50},
      draw = green!50,
      line width = 0.8pt,
    ]
    table[
      x = h,
      y = 2,
    ]{nonsmooth_ieks_error.txt};

    \nextgroupplot[
    width	= 0.25\textwidth,
    ymode = log,
    xmode = log,
    xmin = 0.0039, xmax = 0.125,
    ymin = 1e-10, ymax = 1,
    grid = major,
    xtick = {1e-3,1e-2,1e-1},
    ytick = {1e-12,1e-9,1e-6,1e-3,1e-0},
    ticklabel style = {font=\tiny},
    xlabel style = {font=\scriptsize},
    ylabel style = {font=\scriptsize},
    xlabel = {$\delta$},
    ylabel = {$\abs[0]{\hat{y} - y^* }_{H^0_\infty}$},
    title style = {font=\footnotesize},
    title = {$\nu = 3$},
    ]

    \addplot[
      draw = black,
      line width = 0.8pt,
    ]
    table[
      x = h,
      y = 3,
    ]{nonsmooth_refs.txt};

    \addplot[
      mark = *,
      mark repeat = {1},
      mark options = {scale=0.7,fill=red!50},
      draw = red!50,
      line width = 0.8pt,
    ]
    table[
      x = h,
      y = 3,
    ]{nonsmooth_eks0_error.txt};

    \addplot[
      mark = square*,
      mark repeat = {1},
      mark options = {scale=0.7,fill=blue!50},
      draw = blue!50,
      line width = 0.8pt,
    ]
    table[
      x = h,
      y = 3,
    ]{nonsmooth_eks1_error.txt};

    \addplot[
      mark = star,
      mark repeat = {1},
      mark options = {scale=0.7,fill=green!50},
      draw = green!50,
      line width = 0.8pt,
    ]
    table[
      x = h,
      y = 3,
    ]{nonsmooth_ieks_error.txt};

    \coordinate (c1) at (rel axis cs:0,1);

    \nextgroupplot[
    width	= 0.25\textwidth,
    ymode = log,
    xmode = log,
    xmin = 0.0039, xmax = 0.125,
    ymin = 1e-10, ymax = 1,
    grid = major,
    xtick = {1e-3,1e-2,1e-1},
    ytick = {1e-12,1e-9,1e-6,1e-3,1e-0},
    ticklabel style = {font=\tiny},
    xlabel style = {font=\scriptsize},
    ylabel style = {font=\scriptsize},
    xlabel = {$\delta$},
    title style = {font=\footnotesize},
    title = {$\nu = 4$},
    legend style={at={($(0,0)+(1cm,1cm)$)},legend columns=4,fill=none,draw=black,anchor=center,align=center},
    legend to name=fred,
    legend style = {font=\footnotesize},
    ]

    \addplot[
      draw = black,
      line width = 0.8pt,
    ]
    table[
      x = h,
      y = 4,
    ]{nonsmooth_refs.txt};

    \addplot[
      mark = *,
      mark repeat = {1},
      mark options = {scale=0.7,fill=red!50},
      draw = red!50,
      line width = 0.8pt,
    ]
    table[
      x = h,
      y = 4,
    ]{nonsmooth_eks0_error.txt};

    \addplot[
      mark = square*,
      mark repeat = {1},
      mark options = {scale=0.7,fill=blue!50},
      draw = blue!50,
      line width = 0.8pt,
    ]
    table[
      x = h,
      y = 4,
    ]{nonsmooth_eks1_error.txt};

    \addplot[
      mark = star,
      mark repeat = {1},
      mark options = {scale=0.7,fill=green!50},
      draw = green!50,
      line width = 0.8pt,
    ]
    table[
      x = h,
      y = 4,
    ]{nonsmooth_ieks_error.txt};

    \coordinate (c2) at (rel axis cs:1,1);

    \addlegendentry{$\delta^\nu$};
    \addlegendentry{$\mathsf{EKS0}$};
    \addlegendentry{$\mathsf{EKS1}$};
    \addlegendentry{$\mathsf{IEKS}$};

  \end{groupplot}
  \coordinate (c3) at ($(c1)!.5!(c2)$);
    \node[below] at (c3 |- current bounding box.south)
      {\pgfplotslegendfromname{fred}};
\end{tikzpicture}
\caption{$\mathcal{L}_\infty$ error of the solution estimate as produced by EKS0 (red), EKS1 (blue), IEKS (green), and the predicted MAP rate $\delta^\nu$ (black), versus fill-distance.}\label{fig:nonsmooth_convergence}
\end{figure}
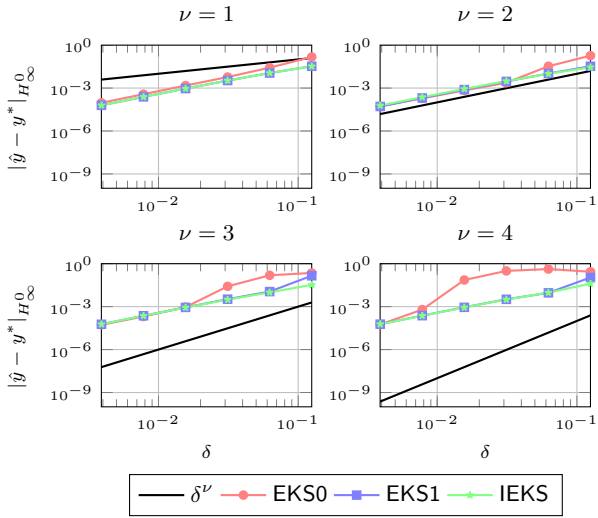

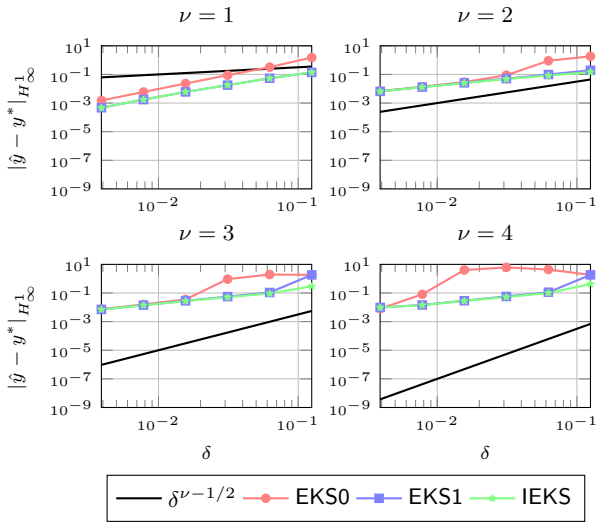
\begin{figure}[t!]
\centering
\begin{tikzpicture}
  \begin{groupplot}
    [group style={group size= 2 by 2, horizontal sep=0.9cm}, height = 0.2\textwidth, width=0.5\textwidth]

    \nextgroupplot[
    width	= 0.25\textwidth,
    ymode = log,
    xmode = log,
    xmin = 0.0039, xmax = 0.125,
    ymin = 1e-9, ymax = 10,
    grid = major,
    xtick = {1e-3,1e-2,1e-1},
    ytick = {1e-9,1e-7,1e-5,1e-3,1e-1,1e+1},
    ticklabel style = {font=\tiny},
    xlabel style = {font=\scriptsize},
    ylabel style = {font=\scriptsize},
    ylabel = {$\abs[0]{ \hat{y} - y^* }_{H^1_\infty}$},
    title style = {font=\footnotesize},
    title = {$\nu = 1$},
    ]

    \addplot[
      draw = black,
      line width = 0.8pt,
    ]
    table[
      x = h,
      y = 1,
    ]{d1_nonsmooth_refs.txt};

    \addplot[
      mark = *,
      mark repeat = {1},
      mark options = {scale=0.7,fill=red!50},
      draw = red!50,
      line width = 0.8pt,
    ]
    table[
      x = h,
      y = 1,
    ]{d1_nonsmooth_eks0_error.txt};

    \addplot[
      mark = square*,
      mark repeat = {1},
      mark options = {scale=0.7,fill=blue!50},
      draw = blue!50,
      line width = 0.8pt,
    ]
    table[
      x = h,
      y = 1,
    ]{d1_nonsmooth_eks1_error.txt};

    \addplot[
      mark = star,
      mark repeat = {1},
      mark options = {scale=0.7,fill=green!50},
      draw = green!50,
      line width = 0.8pt,
    ]
    table[
      x = h,
      y = 1,
    ]{d1_nonsmooth_ieks_error.txt};

    \nextgroupplot[
    width	= 0.25\textwidth,
    ymode = log,
    xmode = log,
    xmin = 0.0039, xmax = 0.125,
    ymin = 1e-9, ymax = 10,
    grid = major,
    xtick = {1e-3,1e-2,1e-1},
    ytick = {1e-9,1e-7,1e-5,1e-3,1e-1,1e+1},
    ticklabel style = {font=\tiny},
    xlabel style = {font=\scriptsize},
    ylabel style = {font=\scriptsize},
    title style = {font=\footnotesize},
    title = {$\nu = 2$},
    ]

    \addplot[
      draw = black,
      line width = 0.8pt,
    ]
    table[
      x = h,
      y = 2,
    ]{d1_nonsmooth_refs.txt};

    \addplot[
      mark = *,
      mark repeat = {1},
      mark options = {scale=0.7,fill=red!50},
      draw = red!50,
      line width = 0.8pt,
    ]
    table[
      x = h,
      y = 2,
    ]{d1_nonsmooth_eks0_error.txt};

    \addplot[
      mark = square*,
      mark repeat = {1},
      mark options = {scale=0.7,fill=blue!50},
      draw = blue!50,
      line width = 0.8pt,
    ]
    table[
      x = h,
      y = 2,
    ]{d1_nonsmooth_eks1_error.txt};

    \addplot[
      mark = star,
      mark repeat = {1},
      mark options = {scale=0.7,fill=green!50},
      draw = green!50,
      line width = 0.8pt,
    ]
    table[
      x = h,
      y = 2,
    ]{d1_nonsmooth_ieks_error.txt};

    \nextgroupplot[
    width	= 0.25\textwidth,
    ymode = log,
    xmode = log,
    xmin = 0.0039, xmax = 0.125,
    ymin = 1e-9, ymax = 10,
    grid = major,
    xtick = {1e-3,1e-2,1e-1},
    ytick = {1e-9,1e-7,1e-5,1e-3,1e-1,1e+1},
    ticklabel style = {font=\tiny},
    xlabel style = {font=\scriptsize},
    ylabel style = {font=\scriptsize},
    xlabel = {$\delta$},
    ylabel = {$\abs[0]{\hat{y} - y^* }_{H^1_\infty}$},
    title style = {font=\footnotesize},
    title = {$\nu = 3$},
    ]

    \addplot[
      draw = black,
      line width = 0.8pt,
    ]
    table[
      x = h,
      y = 3,
    ]{d1_nonsmooth_refs.txt};

    \addplot[
      mark = *,
      mark repeat = {1},
      mark options = {scale=0.7,fill=red!50},
      draw = red!50,
      line width = 0.8pt,
    ]
    table[
      x = h,
      y = 3,
    ]{d1_nonsmooth_eks0_error.txt};

    \addplot[
      mark = square*,
      mark repeat = {1},
      mark options = {scale=0.7,fill=blue!50},
      draw = blue!50,
      line width = 0.8pt,
    ]
    table[
      x = h,
      y = 3,
    ]{d1_nonsmooth_eks1_error.txt};

    \addplot[
      mark = star,
      mark repeat = {1},
      mark options = {scale=0.7,fill=green!50},
      draw = green!50,
      line width = 0.8pt,
    ]
    table[
      x = h,
      y = 3,
    ]{d1_nonsmooth_ieks_error.txt};

        \coordinate (c1) at (rel axis cs:0,1);

    \nextgroupplot[
    width	= 0.25\textwidth,
    ymode = log,
    xmode = log,
    xmin = 0.0039, xmax = 0.125,
    ymin = 1e-9, ymax = 10,
    grid = major,
    xtick = {1e-3,1e-2,1e-1},
    ytick = {1e-9,1e-7,1e-5,1e-3,1e-1,1e+1},
    ticklabel style = {font=\tiny},
    xlabel style = {font=\scriptsize},
    ylabel style = {font=\scriptsize},
    xlabel = {$\delta$},
    title style = {font=\footnotesize},
    title = {$\nu = 4$},
    legend style={at={($(0,0)+(1cm,1cm)$)},legend columns=4,fill=none,draw=black,anchor=center,align=center},
    legend to name=fred,
    legend style = {font=\footnotesize},
    ]

    \addplot[
      draw = black,
      line width = 0.8pt,
    ]
    table[
      x = h,
      y = 4,
    ]{d1_nonsmooth_refs.txt};

    \addplot[
      mark = *,
      mark repeat = {1},
      mark options = {scale=0.7,fill=red!50},
      draw = red!50,
      line width = 0.8pt,
    ]
    table[
      x = h,
      y = 4,
    ]{d1_nonsmooth_eks0_error.txt};

    \addplot[
      mark = square*,
      mark repeat = {1},
      mark options = {scale=0.7,fill=blue!50},
      draw = blue!50,
      line width = 0.8pt,
    ]
    table[
      x = h,
      y = 4,
    ]{d1_nonsmooth_eks1_error.txt};

    \addplot[
      mark = star,
      mark repeat = {1},
      mark options = {scale=0.7,fill=green!50},
      draw = green!50,
      line width = 0.8pt,
    ]
    table[
      x = h,
      y = 4,
    ]{d1_nonsmooth_ieks_error.txt};

    \coordinate (c2) at (rel axis cs:1,1);

    \addlegendentry{$\delta^{\nu-1/2}$};
    \addlegendentry{$\mathsf{EKS0}$};
    \addlegendentry{$\mathsf{EKS1}$};
    \addlegendentry{$\mathsf{IEKS}$};

  \end{groupplot}
  \coordinate (c3) at ($(c1)!.5!(c2)$);
    \node[below] at (c3 |- current bounding box.south)
      {\pgfplotslegendfromname{fred}};
\end{tikzpicture}
\caption{$\mathcal{L}_\infty$ error of the derivative estimate as produced by EKS0 (red), EKS1 (blue), IEKS (green), and the predicted MAP rate $\delta^\nu$ (black), versus fill-distance.}\label{fig:d1_nonsmooth_convergence}
\end{figure}

\begin{figure}[t!]
\centering
\begin{tikzpicture}
  \begin{groupplot}
    [group style={group size= 2 by 1, horizontal sep=1.5cm}, height = 0.2\textwidth, width=0.5\textwidth]

    \nextgroupplot[
    width	= 0.25\textwidth,
    xmin = 0, xmax = 1,
    ymin = 0, ymax = 1.4,
    grid = major,
    xtick = {0.2,0.4,0.6,0.8,1},
    ytick = {0,0.2,0.4,0.6,0.8,1,1.2,1.4},
    ticklabel style = {font=\tiny},
    xlabel style = {font=\scriptsize},
    ylabel style = {font=\scriptsize},
    xlabel = {$t$},
    ylabel = {$y$},
    title style = {font=\footnotesize},
    ]

    \addplot[
      draw = black,
      line width = 0.8pt,
    ]
    table[
      x = t,
      y = y,
    ]{nonsmooth_demo_gt.txt};

     \addplot[
       mark = square*,
       mark repeat = {200},
       mark options = {scale=0.5,fill=blue!50},
       draw = blue!50,
       line width = 0.7pt,
       ]
       table[
         x = t,
         y = eks1,
       ]{nonsmooth_demo_eks1.txt};

      \addplot[name path=upper_eks1,draw=none] table[x=t,y=eks1ub]{nonsmooth_demo_eks1.txt};
      \addplot[name path=lower_eks1,draw=none] table[x=t,y=eks1lb]{nonsmooth_demo_eks1.txt};
      \addplot[fill=blue!30,fill opacity=0.7] fill between[of=upper_eks1 and lower_eks1];

      \addplot[
        mark = star,
        mark repeat = {200},
        mark options = {scale=0.5,fill=green!50},
        draw = green!50,
        line width = 0.7pt,
        ]
        table[
          x = t,
          y = ieks,
        ]{nonsmooth_demo_ieks.txt};

       \addplot[name path=upper_ieks,draw=none] table[x=t,y=ieksub]{nonsmooth_demo_ieks.txt};
       \addplot[name path=lower_ieks,draw=none] table[x=t,y=iekslb]{nonsmooth_demo_ieks.txt};
       \addplot[fill=green!30,fill opacity=0.7] fill between[of=upper_ieks and lower_ieks];

       \coordinate (c1) at (rel axis cs:0,1);

    \nextgroupplot[
    width	= 0.25\textwidth,
    xmin = 0, xmax = 1,
    ymin = -0.5, ymax = 2.5,
    grid = major,
    xtick = {0,0.2,0.4,0.6,0.8,1},
    ytick = {-0.5,0,0.5,1,1.5},
    ticklabel style = {font=\tiny},
    xlabel style = {font=\scriptsize},
    ylabel style = {font=\scriptsize},
    xlabel = {$t$},
    ylabel = {$\dot{y}$},
    title style = {font=\footnotesize},
    legend style={at={($(0,0)+(1cm,1cm)$)},legend columns=4,fill=none,draw=black,anchor=center,align=center},
    legend to name=fred,
    legend style = {font=\footnotesize},
    ]

    \addplot[
      draw = black,
      line width = 0.8pt,
    ]
    table[
      x = t,
      y = dy,
    ]{nonsmooth_demo_gt.txt};

      \addplot[
        mark = square*,
        mark repeat = {200},
        mark options = {scale=0.7,fill=blue!50},
        draw = blue!50,
        line width = 0.8pt,
        ]
        table[
          x = t,
          y = deks1,
        ]{nonsmooth_demo_eks1.txt};

        \addplot[name path=upper_deks1,draw=none,forget plot] table[x=t,y=deks1ub]{nonsmooth_demo_eks1.txt};
        \addplot[name path=lower_deks1,draw=none,forget plot] table[x=t,y=deks1lb]{nonsmooth_demo_eks1.txt};
        \addplot[fill=blue!30, fill opacity=0.7,forget plot] fill between[of=upper_deks1 and lower_deks1];

        \addplot[
          mark = star,
          mark repeat = {200},
          mark options = {scale=0.7,fill=green!50},
          draw = green!50,
          line width = 0.8pt,
          ]
          table[
            x = t,
            y = dieks,
          ]{nonsmooth_demo_ieks.txt};

          \addplot[name path=upper_dieks,draw=none,forget plot] table[x=t,y=dieksub]{nonsmooth_demo_ieks.txt};
          \addplot[name path=lower_dieks,draw=none,forget plot] table[x=t,y=diekslb]{nonsmooth_demo_ieks.txt};
          \addplot[fill=green!30, fill opacity=0.7,forget plot] fill between[of=upper_dieks and lower_dieks];

        \coordinate (c2) at (rel axis cs:1,1);

    \addlegendentry{$\mathrm{GT}$};
    \addlegendentry{$\mathsf{EKS1}$};
    \addlegendentry{$\mathsf{IEKS}$};

  \end{groupplot}
  \coordinate (c3) at ($(c1)!.5!(c2)$);
    \node[below] at (c3 |- current bounding box.south)
      {\pgfplotslegendfromname{fred}};
\end{tikzpicture}
\caption{Reconstruction of the solution to the non-smooth ODE (left) and its derivative (right) with two standard deviation credible bands for EKS1 (blue) and IEKS (green), using a step size of $h = 0.25$.}\label{fig:nonsmooth_demo}
\end{figure}
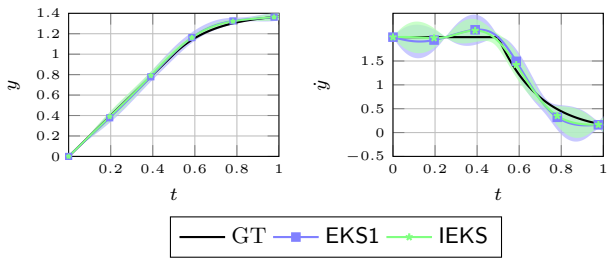

\section{Conclusion}\label{sec:discussion}
In this paper, the maximum a posteriori estimate associated with the Bayesian solution of initial value problems \citep{Cockayne2019a} was examined and it was shown to enjoy fast convergence rates to the true solution.

In the present setting the MAP estimate is just taken as a given, in the sense that IEKS is not guaranteed to produced the global optimum of the MAP problem. It would therefore be fruitful to study the MAP problem more carefully. In particular, establishing conditions on the vector field and the fill-distance under which the MAP problem admits a unique local optimum would be a point for future research. On the algorithmic side, other MAP estimators can be considered, such as Levenberg--Marquardt \citep{Sarkka2020a} or alternate direction method of multipliers \citep{Boyd2011,Gao2019a}.

Furthermore, the empirical findings of Section \ref{sec:numerical_examples} suggests, although not being MAP estimators, EKS0 and EKS1 can likely be given convergence statements similar to Theorem \ref{thm:convergence}. It is not immediately clear what the most effective approach for this purpose is. On one hand, one can attempt to significantly extend the results of \cite{Kersting2018}, which is more in line with how convergence rates are obtained for classical solvers. On the other hand, it seems like the methodology developed here can be extended for local convergence analysis as well by considering the filter update as an interpolation problem in some RKHS on each interval $[t_{n-1},t_n]$.

\begin{acknowledgements}
The authors have had productive discussions with Toni Karvonen and Hans Kersting.
\end{acknowledgements}

\appendix


\section{Computing Transition Densities}\label{app:discretisation}
An effective method for computing the parameters of the transition density in  \eqref{eq:discrete_time_parameters} is the \emph{matrix fraction decomposition} \citep{VanLoan1978,Axelsson2014,Sarkka2019}. Define the matrix valued function
\begin{equation*}
\Xi(h) = \exp \Bigg(  \begin{pmatrix}F & \mathrm{E}_\nu \Gamma \mathrm{E}_\nu^\T \\  0 & -F^\T\end{pmatrix}h \Bigg).
\end{equation*}
It can then be shown that $\Xi$ has the following structure
\begin{equation*}
\Xi(h) = \begin{pmatrix} \Xi_{11}(h) & \Xi_{12}(h) \\ 0 & \Xi_{22}(h) \end{pmatrix},
\end{equation*}
and \citep{Axelsson2014}
\begin{subequations}
\begin{align}
A(h) &= \Xi_{11}(h), \\
Q(h) &= \Xi_{12}(h) \Xi_{11}^\T(h).
\end{align}
\end{subequations}
Furthermore, the Green's functions can be evaluated by the same means by noting that (see \eqref{eq:greens_functions})
\begin{equation*}
G_X(t,\tau) = \theta(t-\tau)A(t-\tau) \mathrm{E}_\nu \Gamma^{1/2}.
\end{equation*}

\section{Calibration}\label{app:uncertainty_calibration}
For a full statistical treatment of the inference problem, the parameters $F_m$ $m=0,\ldots,\nu$, $\Gamma$ and $\Sigma(t_0^-)$ need to be estimated. Of particular importance in terms of calibrating uncertainty properly are $\Sigma(t_0^-)$ and $\Gamma$ (see \eqref{eq:prior_measure_x}), which the present discussion is just concerned with.

It can be shown that the logarithm of (quasi-) likelihood as produced by the Gaussian inference methods is, up to an unimportant constant, given by (cf. \citealt{Tronarp2019d})
\begin{equation*}
\begin{split}
\ell &= -\frac{1}{2}\log \det S(t_0) - \frac{1}{2}\begin{pmatrix} y_0 \\ f(0,y_0) \end{pmatrix}^\T S^{-1}(t_0) \begin{pmatrix} y_0 \\ f(0,y_0) \end{pmatrix} \\
 &\quad- \frac{1}{2}\sum_{n=1}^N \log \det S(t_n) - \frac{1}{2}\sum_{n=1}^N\norm[0]{\zeta(t_n) - C(t_n)\mu_F(t_n^-)}_{S(t_n)}^2.
\end{split}
\end{equation*}
Additionally, if  $\Sigma(t_0^-) = \sigma^2\breve{\Sigma}(t_0^-)$ and $\Gamma = \sigma^2 \breve{\Gamma}$ for some positive definite matrices $\breve{\Sigma}_F(t_0^-)$ and $\breve{\Gamma}$, then it can be shown that the log-likelihood, up to some unimportant constant, reduces to (see Appendix C of \citealt{Tronarp2019c} for details)\footnote{There is a slight difference in the log-likelihood expression from that of \cite{Tronarp2019c}. This is because here the initial conditions are inferred while \cite{Tronarp2019c} encodes them directly in the prior.}
\begin{equation*}
\begin{split}
\ell(\sigma) &=  - \frac{1}{2\sigma^2}\begin{pmatrix} y_0 \\ f(0,y_0) \end{pmatrix}^\T \breve{S}^{-1}(t_0) \begin{pmatrix} y_0 \\ f(0,y_0) \end{pmatrix} \\
 &\quad - \frac{1}{2\sigma^2}\sum_{n=1}^N\norm[0]{\zeta(t_n) - C(t_n)\mu_F(t_n^-)}_{\breve{S}(t_n)}^2 \\
 &\quad -\frac{d(N+2)}{2} \log \sigma^2,
\end{split}
\end{equation*}
where $\breve{\cdotp}$ denotes the output of the filter using the parameters $(\breve{\Sigma}(t_0^-),\breve{\Gamma})$ rather than $(\Sigma(t_0^-),\Gamma)$. This yields the following proposition, which is proven in Appendix C of \cite{Tronarp2019c}, \emph{mutatis mutandis}.

\begin{proposition}\label{prop:s2_mle}
Let $\Sigma(t_0^-) = \sigma^2\breve{\Sigma}(t_0^-)$ and $\Gamma = \sigma^2 \breve{\Gamma}$ for some positive definite matrices $\breve{\Sigma}(t_0^-)$ and $\breve{\Gamma}$, then the (quasi-) maximum likelihood estimate of $\sigma^2$ is given by
\begin{equation}
\begin{split}
\hat{\sigma}_N^2 &= \frac{1}{d(N+2)}\begin{pmatrix} y_0 \\ f(0,y_0) \end{pmatrix}^\T \breve{S}^{-1}(t_0) \begin{pmatrix} y_0 \\ f(0,y_0) \end{pmatrix} \\
&\quad+\frac{1}{d(N+2)}\sum_{n=1}^N\norm[0]{\zeta(t_n) - C(t_n)\mu_F(t_n^-)}_{\breve{S}(t_n)}^2.
\end{split}
\end{equation}
\end{proposition}

Bounds for worst case overconfidence and underconfidence under maximum likelihood estimation of $\sigma^2$ has recently been obtained by \cite{Karvonen2020}. These results appear to carry over to the present setting for affine vector fields. However, it is not immediately clear how to generalise this to a larger class of vector fields.

%
%

\bibliographystyle{spbasic}      
\bibliography{refs}   

\begin{thebibliography}{67}
\providecommand{\natexlab}[1]{#1}
\providecommand{\url}[1]{{#1}}
\providecommand{\urlprefix}{URL }
\expandafter\ifx\csname urlstyle\endcsname\relax
  \providecommand{\doi}[1]{DOI~\discretionary{}{}{}#1}\else
  \providecommand{\doi}{DOI~\discretionary{}{}{}\begingroup
  \urlstyle{rm}\Url}\fi
\providecommand{\eprint}[2][]{\url{#2}}

\bibitem[{Abdulle and Garegnani(2020)}]{Abdulle2020}
Abdulle A, Garegnani G (2020) Random time step probabilistic methods for
  uncertainty quantification in chaotic and geometric numerical integration.
  Statistics and Computing pp 1--26

\bibitem[{Adams and Fournier(2003)}]{Adams2003}
Adams RA, Fournier JJ (2003) Sobolev spaces, vol 140. Elsevier

\bibitem[{Arcang{\'e}li et~al.(2007)Arcang{\'e}li, de~Silanes, and
  Torrens}]{Arcangeli2007}
Arcang{\'e}li R, de~Silanes MCL, Torrens JJ (2007) An extension of a bound for
  functions in {S}obolev spaces, with applications to (m, s)-spline
  interpolation and smoothing. Numerische Mathematik 107(2):181--211

\bibitem[{Arnol'd(1992)}]{Arnold1992}
Arnol'd VI (1992) Ordinary Differential Equations. Springer-Verlag Berlin
  Heidelberg

\bibitem[{Axelsson and Gustafsson(2014)}]{Axelsson2014}
Axelsson P, Gustafsson F (2014) Discrete-time solutions to the continuous-time
  differential {L}yapunov equation with applications to {K}alman filtering.
  IEEE Transactions on Automatic Control 60(3):632--643

\bibitem[{Bell(1994)}]{Bell1994}
Bell BM (1994) The iterated {K}alman smoother as a {G}auss--{N}ewton method.
  SIAM Journal on Optimization 4(3):626--636

\bibitem[{Boyd et~al.(2011)Boyd, Parikh, Chu, Peleato, and Eckstein}]{Boyd2011}
Boyd S, Parikh N, Chu E, Peleato B, Eckstein J (2011) Distributed optimization
  and statistical learning via the alternating direction method of multipliers.
  Foundations and Trends in Machine learning 3(1):1--122

\bibitem[{Butcher(2008)}]{Butcher2008}
Butcher JC (2008) Numerical Methods for Ordinary Differential Equations, 2nd
  edn. John Wiley \& Sons, Inc.

\bibitem[{Chkrebtii et~al.(2016)Chkrebtii, Campbell, Calderhead, and
  Girolami}]{Chkrebtii2016}
Chkrebtii OA, Campbell DA, Calderhead B, Girolami MA (2016) {Bayesian} solution
  uncertainty quantification for differential equations. Bayesian Analysis
  11(4):1239--1267

\bibitem[{Cockayne et~al.(2019)Cockayne, Oates, Sullivan, and
  Girolami}]{Cockayne2019a}
Cockayne J, Oates CJ, Sullivan TJ, Girolami M (2019) Bayesian probabilistic
  numerical methods. SIAM Review 61(4):756--789

\bibitem[{Conrad et~al.(2017)Conrad, Girolami, S{\"a}rkk{\"a}, Stuart, and
  Zygalakis}]{Conrad2017}
Conrad PR, Girolami M, S{\"a}rkk{\"a} S, Stuart A, Zygalakis K (2017)
  Statistical analysis of differential equations: introducing probability
  measures on numerical solutions. Statistics and Computing 27(4):1065--1082

\bibitem[{Cox and O'Sullivan(1990)}]{Cox1990}
Cox DD, O'Sullivan F (1990) Asymptotic analysis of penalized likelihood and
  related estimators. The Annals of Statistics pp 1676--1695

\bibitem[{Deuflhard and Bornemann(2002)}]{Deuflhard2002}
Deuflhard P, Bornemann F (2002) Scientific Computing with Ordinary Differential
  Equations. Springer

\bibitem[{Gao et~al.(2019)Gao, Tronarp, and S{\"a}rkk{\"a}}]{Gao2019a}
Gao R, Tronarp F, S{\"a}rkk{\"a} S (2019) Iterated extended {K}alman
  smoother-based variable splitting for ${L}_1$-regularized state estimation.
  IEEE Transactions on Signal Processing 67(19):5078--5092

\bibitem[{Gin{\'e} and Nickl(2016)}]{Gine2016}
Gin{\'e} E, Nickl R (2016) Mathematical Foundations of Infinite-Dimensional
  Statistical Models. Cambridge University Press

\bibitem[{Girosi et~al.(1995)Girosi, Jones, and Poggio}]{Girosi1995}
Girosi F, Jones M, Poggio T (1995) Regularization theory and neural networks
  architectures. Neural computation 7(2):219--269

\bibitem[{Hairer and Wanner(1996)}]{Hairer1996}
Hairer E, Wanner G (1996) Solving Ordinary Differential Equations {II}: Stiff
  and Differential-Algebraic Problems. Springer

\bibitem[{Hairer et~al.(1987)Hairer, N{\o}rsett, and Wanner}]{Hairer87}
Hairer E, N{\o}rsett S, Wanner G (1987) {Solving Ordinary Differential
  Equations I -- Nonstiff Problems}. Springer

\bibitem[{Hartikainen and S{\"a}rkk{\"a}(2010)}]{Hartikainen2010}
Hartikainen J, S{\"a}rkk{\"a} S (2010) Kalman filtering and smoothing solutions
  to temporal {G}aussian process regression models. In: 2010 IEEE international
  workshop on machine learning for signal processing, IEEE, pp 379--384

\bibitem[{Hennig and Hauberg(2014)}]{Hennig2014}
Hennig P, Hauberg S (2014) Probabilistic solutions to differential equations
  and their application to {Riemannian} statistics. In: {Proc. of the 17th int.
  Conf. on Artificial Intelligence and Statistics ({AISTATS})}, JMLR, W\&CP,
  vol~33

\bibitem[{Hennig et~al.(2015)Hennig, Osborne, and Girolami}]{Hennig2015}
Hennig P, Osborne MA, Girolami M (2015) Probabilistic numerics and uncertainty
  in computations. Proceedings of the Royal Society A: Mathematical, Physical
  and Engineering Sciences 471(2179):20150142

\bibitem[{John et~al.(2019)John, Heuveline, and Schober}]{John2019}
John D, Heuveline V, Schober M (2019) {GOODE}: A {G}aussian off-the-shelf
  ordinary differential equation solver. In: Chaudhuri K, Salakhutdinov R (eds)
  Proceedings of the 36th International Conference on Machine Learning, PMLR,
  Long Beach, California, USA, Proceedings of Machine Learning Research,
  vol~97, pp 3152--3162

\bibitem[{Kalman and Bucy(1961)}]{KalmanBucy1961}
Kalman R, Bucy R (1961) New results in linear filtering and prediction theory.
  Transactions of the ASME, Journal of Basic Engineering 83:95--108

\bibitem[{Kalman(1960)}]{Kalman1960}
Kalman RE (1960) A new approach to linear filtering and prediction problems.
  Journal of Basic Engineering 82(1):35--45

\bibitem[{Kanagawa et~al.(2018)Kanagawa, Hennig, Sejdinovic, and
  Sriperumbudur}]{Kanagawa2018a}
Kanagawa M, Hennig P, Sejdinovic D, Sriperumbudur BK (2018) Gaussian processes
  and kernel methods: A review on connections and equivalences. arXiv preprint
  arXiv:180702582

\bibitem[{Kanagawa et~al.(2020)Kanagawa, Sriperumbudur, and
  Fukumizu}]{Kanagawa2020a}
Kanagawa M, Sriperumbudur BK, Fukumizu K (2020) Convergence analysis of
  deterministic kernel-based quadrature rules in misspecified settings.
  Foundations of Computational Mathematics 20:155--194

\bibitem[{Karvonen and Sarkk{\"a}(2016)}]{Karvonen2016}
Karvonen T, Sarkk{\"a} S (2016) Approximate state-space {G}aussian processes
  via spectral transformation. In: 2016 IEEE 26th International Workshop on
  Machine Learning for Signal Processing (MLSP)

\bibitem[{Karvonen et~al.(2020)Karvonen, Wynne, Tronarp, Oates, and
  S{\"a}rkk{\"a}}]{Karvonen2020}
Karvonen T, Wynne G, Tronarp F, Oates CJ, S{\"a}rkk{\"a} S (2020) Maximum
  likelihood estimation and uncertainty quantification for {G}aussian process
  approximation of deterministic functions. arXiv preprint arXiv:200110965

\bibitem[{Kelley and Peterson(2010)}]{Kelley2010}
Kelley WG, Peterson AC (2010) The Theory of Differential Equations: Classical
  and Qualitative. Springer Science \& Business Media

\bibitem[{{Kersting} and {Hennig}(2016)}]{Kersting2016}
{Kersting} H, {Hennig} P (2016) Active uncertainty calibration in {Bayesian}
  {ODE} solvers. In: {Uncertainty in Artificial Intelligence (UAI) 2016}, AUAI,
  New York City, NY, USA

\bibitem[{Kersting et~al.(2018)Kersting, Sullivan, and Hennig}]{Kersting2018}
Kersting H, Sullivan TJ, Hennig P (2018) Convergence rates of {G}aussian {ODE}
  filters. arXiv preprint arXiv:180709737

\bibitem[{Kersting et~al.(2020)Kersting, Kr{\"a}mer, Schiegg, Daniel, Tiemann,
  and Hennig}]{Kersting2020}
Kersting H, Kr{\"a}mer N, Schiegg M, Daniel C, Tiemann M, Hennig P (2020)
  Differentiable likelihoods for fast inversion of 'likelihood-free' dynamical
  systems. arXiv preprint arXiv:200209301

\bibitem[{Kimeldorf and Wahba(1971)}]{Kimeldorf1971}
Kimeldorf G, Wahba G (1971) Some results on {T}chebycheffian spline functions.
  Journal of mathematical analysis and applications 33(1):82--95

\bibitem[{Kimeldorf and Wahba(1970)}]{Kimeldorf1970}
Kimeldorf GS, Wahba G (1970) A correspondence between {B}ayesian estimation on
  stochastic processes and smoothing by splines. The Annals of Mathematical
  Statistics 41(2):495--502

\bibitem[{Knoth(1989)}]{Knoth1989}
Knoth O (1989) A globalization scheme for the generalized {G}auss--{N}ewton
  method. Numerische Mathematik 56(6):591--607

\bibitem[{Lie et~al.(2019)Lie, Stuart, and Sullivan}]{Lie2019}
Lie HC, Stuart AM, Sullivan TJ (2019) Strong convergence rates of probabilistic
  integrators for ordinary differential equations. Statistics and Computing
  29(6):1265--1283

\bibitem[{Magnani et~al.(2017)Magnani, Kersting, Schober, and
  Hennig}]{Magnani2017}
Magnani E, Kersting H, Schober M, Hennig P (2017) {Bayesian Filtering for ODEs
  with Bounded Derivatives}. {arXiv:170908471 [csNA]}

\bibitem[{Marcus and Mizel(1973)}]{Marcus1973}
Marcus M, Mizel VJ (1973) Nemytsky operators on {S}obolev spaces. Arch Rational
  Mech Anal 51:347--370

\bibitem[{Matsuda and Miyatake(2019)}]{Matsuda2019}
Matsuda T, Miyatake Y (2019) Estimation of ordinary differential equation
  models with discretization error quantification. arXiv preprint
  arXiv:190710565

\bibitem[{Nielson(1997)}]{Nielson1997}
Nielson OA (1997) An Introduction to Integration and Measure Theory. John Wiley
  \& Sons, Inc., New York.

\bibitem[{Oates and Sullivan(2019)}]{Oates2019a}
Oates CJ, Sullivan TJ (2019) A modern retrospective on probabilistic numerics.
  Statistics and Computing 29(6):1335--1351

\bibitem[{{\O}ksendal(2003)}]{Oksendal2003}
{\O}ksendal B (2003) Stochastic Differential Equations - An Introduction with
  Applications. Springer

\bibitem[{Rasmussen and Williams(2006)}]{Rasmussen2006}
Rasmussen CE, Williams CKI (2006) Gaussian Processes for Machine learning.
  {MIT} {P}ress

\bibitem[{Rauch et~al.(1965)Rauch, Tung, and Striebel}]{RauchTungStriebel1965}
Rauch HE, Tung F, Striebel CT (1965) Maximum likelihood estimates of linear
  dynamic system. AIAA Journal 3(8):1445--1450

\bibitem[{S\"arkk\"a(2013)}]{Sarkka2013}
S\"arkk\"a S (2013) {B}ayesian Filtering and Smoothing. Cambridge University
  Press

\bibitem[{S\"arkk\"a and Solin(2019)}]{Sarkka2019}
S\"arkk\"a S, Solin A (2019) Applied Stochastic Differential Equations.
  Cambridge University Press

\bibitem[{S{\"a}rkk{\"a} and Svensson(2020)}]{Sarkka2020a}
S{\"a}rkk{\"a} S, Svensson L (2020) Levenberg--{M}arquardt and line-search
  extended {K}alman smoothers. In: 2020 IEEE International Conference on
  Acoustics, Speech and Signal Processing, IEEE, Virtual location

\bibitem[{S\"arkk\"a et~al.(2013)S\"arkk\"a, Solin, and
  Hartikainen}]{Sarkka2013a}
S\"arkk\"a S, Solin A, Hartikainen J (2013) Spatiotemporal learning via
  infinite-dimensional {B}ayesian filtering and smoothing: A look at {G}aussian
  process regression through {K}alman filtering. IEEE Signal Processing
  Magazine 30(4):51--61

\bibitem[{Schober et~al.(2014)Schober, Duvenaud, and Hennig}]{Schober2014}
Schober M, Duvenaud DK, Hennig P (2014) Probabilistic {ODE} solvers with
  {R}unge-{K}utta means. In: Ghahramani Z, Welling M, Cortes C, Lawrence ND,
  Weinberger KQ (eds) Advances in Neural Information Processing Systems 27,
  Curran Associates, Inc., Montr\'eal, Canada, pp 739--747

\bibitem[{Schober et~al.(2019)Schober, S{\"a}rkk{\"a}, and
  Hennig}]{Schober2019}
Schober M, S{\"a}rkk{\"a} S, Hennig P (2019) A probabilistic model for the
  numerical solution of initial value problems. Statistics and Computing
  29(1):99--122

\bibitem[{Schultz(1970)}]{Schultz1970}
Schultz MH (1970) Error bounds for polynomial spline interpolation. Mathematics
  of Computation 24(111):507--515

\bibitem[{Schumaker(1982)}]{Schumaker1982}
Schumaker LL (1982) Optimal spline solutions of systems of ordinary
  differential equations. In: Differential Equations, Springer, pp 272--283

\bibitem[{Sidhu and Weinert(1979)}]{Sidhu1979}
Sidhu GS, Weinert HL (1979) Vector-valued {L}g-splines {I}. interpolating
  splines. Journal of Mathematical Analysis and Applications 70(2):505--529

\bibitem[{Skilling(1992)}]{Skilling1991}
Skilling J (1992) Bayesian solution of ordinary differential equations. In:
  Maximum entropy and Bayesian methods, Springer, pp 23--37

\bibitem[{Solin and S{\"a}rkk{\"a}(2014)}]{Solin2014}
Solin A, S{\"a}rkk{\"a} S (2014) Gaussian quadratures for state space
  approximation of scale mixtures of squared exponential covariance functions.
  In: 2014 IEEE International Workshop on Machine Learning for Signal
  Processing (MLSP)

\bibitem[{Teymur et~al.(2016)Teymur, Zygalakis, and Calderhead}]{Teymur2016}
Teymur O, Zygalakis K, Calderhead B (2016) Probabilistic linear multistep
  methods. In: Advances in Neural Information Processing Systems (NIPS)

\bibitem[{Teymur et~al.(2018)Teymur, Lie, Sullivan, and
  Calderhead}]{Teymur2018a}
Teymur O, Lie HC, Sullivan T, Calderhead B (2018) Implicit probabilistic
  integrators for {ODEs}. In: {Advances in Neural Information Processing
  Systems (NIPS)}

\bibitem[{Tronarp et~al.(2018)Tronarp, Karvonen, and
  S{\"a}rkk{\"a}}]{Tronarp2018d}
Tronarp F, Karvonen T, S{\"a}rkk{\"a} S (2018) Mixture representation of the
  {M}at{\'e}rn class with applications in state space approximations and
  {B}ayesian quadrature. In: 2018 IEEE 28th International Workshop on Machine
  Learning for Signal Processing (MLSP)

\bibitem[{Tronarp et~al.(2019{\natexlab{a}})Tronarp, Karvonen, and
  S{\"a}rkk{\"a}}]{Tronarp2019d}
Tronarp F, Karvonen T, S{\"a}rkk{\"a} S (2019{\natexlab{a}}) Student's $ t
  $-filters for noise scale estimation. IEEE Signal Processing Letters
  26(2):352--356

\bibitem[{Tronarp et~al.(2019{\natexlab{b}})Tronarp, Kersting, S{\"a}rkk{\"a},
  and Hennig}]{Tronarp2019c}
Tronarp F, Kersting H, S{\"a}rkk{\"a} S, Hennig P (2019{\natexlab{b}})
  Probabilistic solutions to ordinary differential equations as nonlinear
  {B}ayesian filtering: a new perspective. Statistics and Computing
  29(6):1297--1315

\bibitem[{van~der Vaart and van Zanten(2008)}]{VanDerVaart2008}
van~der Vaart AW, van Zanten JH (2008) Reproducing kernel {H}ilbert spaces of
  {G}aussian priors. In: Pushing the limits of contemporary statistics:
  contributions in honor of Jayanta K. Ghosh, Institute of Mathematical
  Statistics, pp 200--222

\bibitem[{Valent(1985)}]{Valent1985}
Valent T (1985) A property of multiplication in {S}obolev spaces. {S}ome
  applications. Rendiconti del Seminario Matematico della Universit{\`a} di
  Padova 74:63--73

\bibitem[{Valent(2013)}]{Valent2013}
Valent T (2013) Boundary value problems of finite elasticity: local theorems on
  existence, uniqueness, and analytic dependence on data, vol~31. Springer
  Science \& Business Media

\bibitem[{Van~Loan(1978)}]{VanLoan1978}
Van~Loan C (1978) Computing integrals involving the matrix exponential. IEEE
  Transactions on Automatic Control 23(3):395--404

\bibitem[{Wahba(1973)}]{Wahba1973}
Wahba G (1973) A class of approximate solutions to linear operator equations.
  Journal of Approximation Theory 9(1):61--77

\bibitem[{Wang et~al.(2018)Wang, Cockayne, and Oates}]{Wang2018}
Wang J, Cockayne J, Oates CJ (2018) A role for symmetry in the {B}ayesian
  solution of differential equations. Bayesian Analysis

\bibitem[{Weinert and Kailath(1974)}]{Weinert1974}
Weinert HL, Kailath T (1974) Stochastic interpretations and recursive
  algorithms for spline functions. The Annals of Statistics 2(4):787--794

\end{thebibliography}

\end{document}